\documentclass{amsart}
\usepackage[pdftex]{graphicx}
\usepackage{graphics,graphicx}
\usepackage{mathrsfs}
\usepackage[all]{xy}
\newtheorem{theorem}{Theorem}[section]
\newtheorem{pro}[theorem]{Proposition}
\newtheorem{cor}[theorem]{Corollary}
\newtheorem{lemma}[theorem]{Lemma}

\theoremstyle{definition}
\newtheorem{definition}[theorem]{Definition}

\theoremstyle{remark}

\numberwithin{equation}{section}

\begin{document}

\title[Chebyshev morphisms on affine varieties]{Dynamics of Chebyshev endomorphisms on some affine algebraic varieties}

\author{ Keisuke Uchimura}
%    Address of record for the research reported here%
\address{Department of Mathematics, Tokai University,  Hiratsuka, 259-1292, Japan}
\email{uchimura@tokai-u.jp. }

%    General info
\renewcommand{\thefootnote}{\fnsymbol{footnote}}
\footnote[0]{This work was supported by the Research Institute for Mathematical Sciences, a Joint Usage/Research Center located in Kyoto University.}

\begin{abstract}
The  Chebyshev polynomials \(T_d\) in one variable are typical chaotic maps on   \({\mathbb C}\).   Chebyshev endomorphisms  \(P_{A_n}^{d} : {\mathbb C}^n \to {\mathbb C}^n\)  \ are also chaotic.  We consider the action of the dihedral group \(D_{n+1}\) \ on  \({\mathbb C}^n\).   The  endomorphism  \(P_{A_n}^{d}\) maps any  \(D_{n+1}-\)orbit  of \({\bf z}  \in {\mathbb C}^n\)  to a \ \(D_{n+1}\)-orbit of \(P_{A_n}^{d}({\bf z})\).  The  endomorphism  \(P_{A_n}^{d}\)  induces a mapping  on \ \({\mathbb C}^n/D_{n+1}\).

Using invariant theory we embed \ \({\mathbb C}^n/D_{n+1}\) \ as an affine subvariety \(X\) in \({\mathbb C}^m\).  Then we have morphisms \(g_d\) on \(X\).  We study the cases \ \(n = 2\) \ and  \(3\).  In these cases the morphisms \(g_d\) are  defined over \({\mathbb Z}\).  We find a class of affine subvarieties \(V\) of \(X\) which are invariant under \(g_d\).  These varieties are concerned with branch loci or critical loci.The class contains \ \({\mathbb C}^2\),  a cuspidal cubic,  a parabola,  a quadric hypersurface in \({\mathbb C}^4\),  an affine algebraic surface in \({\mathbb C}^4\)  which is birationally equivalent to an affine  quadric cone in \({\mathbb C}^3\),  and others.  For each affine variety \(V\) in the class, there exists a polynomial parametrization \(P_V\) satisfying \ \(g_d\mid_V(P_V(y_1, \dots , y_k)) = P_V(T_d(y_1), \dots , T_d(y_k))\), \ where \ \(T_d(z)\) \ is a Chebyshev polynomial in one variable.  Then we  determine the set of bounded orbits of \ \(g_d\mid_V\) in each invariant set \(V\) and give relations between them.  
\end{abstract}

\maketitle

\section{Introduction}
  Chebyshev endomorphisms  \(P_{A_n}^d\) of degree \(d\) on  \({\mathbb C}^n\)  are originally defined by  Veselov  \cite{V}  and Hoffman and Withers \cite{HW}.  The endomorphism \(P_{A_n}^{d}\) is associated to the Lie algebra of type \(A_n\) .      We studied dynamics of \(P_{A_2}^d\)   on   \({\mathbb C}^2\)  and   \(P_{A_3}^d\)  on  \({\mathbb C}^3\)  in  \cite{U2}  and  \cite{U3}, respectively.

In this paper,  we study the action of the dihedral group \(D_{n+1}\) \ of order \(2(n+1)\) \ on  \({\mathbb C}^n\).  

Petersen  and Sinclair \cite{PS}  studied conjugate reciprocal polynomials \(f\):
\[f(x) =  x^{N} + \sum_{ i=1} ^Nc_ix^{N-i}  
\quad \mbox{where} \quad c_N=1, \quad  c_{N-i} = \bar{c}_i \quad \mbox{for} \quad 1 \le i \le N-1 .\]
They defined the set \quad \(W_N \subset {\mathbb R}^{N-1}\)\quad such that  \(W_N\)  is one to one correspondence with the set of conjugate reciprocal polynomials of degree \(N\)  with all roots on the unit circle.  Theorem 1.3 in  \cite{PS} states that the group of isometries of  \(W_N\)  is isomorphic to the dihedral group  \(D_N\) of order \(2N\).

We find that the set  \(W_{n+1}\) coincides with the set of  bounded orbits of the Chebyshev endomorphism  \(P_{A_n}^d\).  The  set \(K(f)\) of  bounded orbits   of a regular polynomial  endomorphism \(f\) on  \({\mathbb C}^n\)  is defined by 
 \[K(f) : = \{  z \in {\mathbb C}^n : \quad \{ f^k(z) : k = 1, 2, \dots \ \}\quad \mbox{is   bounded} \},\]
where \(f^k\) denotes the \(k\)-th iterate of \(f\). 
See \cite{U1}, \cite{U2} and \cite{U3}. We call \(K(f)\) the \(K\) set of \(f\).

In this paper we extend this group action.  We consider the action of   \(D_{n+1}\)  on  \({\mathbb C}^n\).  We find that \ \(P_{A_n}^d\) \ maps  any  \(D_{n+1}-\)orbit  of \({\bf z} \in  {\mathbb C}^n\)  to a \ \(D_{n+1}\)-orbit of \(P_{A_n}^{d}({\bf z})\).   Then we study maps  on the orbit space \({\mathbb C}^n/D_{n+1}\) .  Indeed.  The dihedral group  \(D_{n+1}\)  is generated by two elements \(R\)  and  \(C\).  That is,  \(D_{n+1} = <R, C>\) .   We will show that 
%%%%%%%%%%%%%%%%%%%%%%%%%%%%%%%equation1.1
\begin{equation}
R^d \circ  P_{A_n}^d =  P_{A_n}^d \circ R \quad \mbox{and} \quad C \circ  P_{A_n}^d =  P_{A_n}^d \circ C.
\end{equation}
Hence  we can define the induced map of  \( P_{A_n}^d\)  on  \({\mathbb C}^n/D_{n+1}\) . 

We define a subspace  \(R_n\)  of  \({\mathbb C}^n\)  by
\[R_n : = \{(z_1,z_2, \dots, z_n) \mid z_j \in {\mathbb C}, \quad z_{n+1-j} = \bar{z}_j,\quad \mbox{for any} \quad j = 1, \dots, n\}.\]
Note that \quad \(R_n \simeq {\mathbb R}^n.\)  We will show that \( P_{A_n}^d\), \(R\)  and \(C\)   admit an invariant subspace \(R_n\).  These maps restricted to \(R_n\) satisfy the equality  (1.1).  The generators \(R\) and \(C\) restricted to \(R_n\) can be represented by matrices in \(GL(n,  {\mathbb R})\).  Hence we regard \(D_{n+1}\)  as a finite matrix group in  \(GL(n,  {\mathbb R})\).  We denote the map \( P_{A_n}^d\) restricted to \(R_n\)  by  \(f_d(x_1,  \dots, x_n)\).  The map \(f_d(x_1,  \dots, x_n)\)  is  a polynomial map over  \({\mathbb Z}\).  

Next, we consider complexification.  We regard  \(x_1,  \dots, x_n\)  as  complex variables and study the orbit space  \({\mathbb C}^n/D_{n+1}\)   where  \(D_{n+1}\) is the finite matrix group in \ \(GL(n,  {\mathbb R})\).  We use invariant theory.  See \cite{CLS}, \cite{DK} and \cite{St}.  We consider the set of all polynomials which are invariant under the action of \(D_{n+1}\)  : 
\[ {\mathbb C}[x_1,  \dots, x_n ]^{ D_{n+1}} : = \{f \in {\mathbb C}[x_1,  \dots, x_n ]  : \forall \pi \in D_{n+1}(f = f\circ \pi)\}.\]
Then the orbit space \({\mathbb C}^n/D_{n+1}\)  can be regarded as  an affine algebraic variety whose coordinate ring is \ \({\mathbb C}[x_1,  \dots, x_n ]^{ D_{n+1}}\).  The affine variety \({\mathbb C}^n/D_{n+1}\) is called an orbit variety and  it is embedded as an affine subvariety into   \({\mathbb C}^m\),  where \(m\) is the number of generators of \ \({\mathbb C}[x_1,  \dots, x_n ]^{ D_{n+1}}\).
By selecting a basis  \(\{p_1,  \dots, p_m\}\)  of \({\mathbb C}[x_1,  \dots, x_n ]^{ D_{n+1}},\)  we can construct a map \(g_d(p_1,  \dots, p_m)\) on  \({\mathbb C}^n/D_{n+1}\)  from the map  \(f_d\).  The map  \(g_d\) is really a morphism defined over  \({\mathbb R}\).  In Sections 3 and 4, we will study  morphisms  \(g_d\) on  \({\mathbb C}^2/D_{3}\)  and   \({\mathbb C}^3/D_{4}\).  We will show that these are  morphisms defined over  \({\mathbb Z}\).

In Section 3, we consider  the dynamics of  \(g_d\) on  \({\mathbb C}^2/D_{3}  \simeq {\mathbb C}^2.\)  We will give a polynomial parametrization \ \(P_{{\mathbb C}^2}\) of \ \( {\mathbb C}^2\) satisfying 
 %%%%%%%%%%%%%%%%%%%%%%%%%%%%%%%equation1.2
\begin{equation}
g_d(P_{ {\mathbb C}^2}(p, q))  = P_{ {\mathbb C}^2}(T_d(p),  T_d(q)), 
\end{equation}
where \ \(T_d(z)\) 
\ is the \ \(d^{th}\) \ Chebyshev polynomial in one variable.

In the theory of complex dynamics, it is important to analyze the orbit of the set of critical points of a map.  In our case,  Chebyshev endomorphism \( P_{A_n}^d\)  is postcritically  finite.  More precisely, the set of critical values of  \( P_{A_n}^d\)  is invariant under  \( P_{A_n}^d\).  Inspired by this fact,  we  consider the branch locus of a morphism \(g_d\) for \(d = 2, 3\).  Then we find that these two branch loci coincide and the branch locus consists of two affine curves.   Each of these affine curves  is  invariant under the morphisms \(g_d\) for any \(d \in {\mathbb N}\).  

We see that \(g_2\) is a finite morphism and its branch locus consists of two curves  \(C_C\)  and  \(C_D\).  (The same holds for \(g_3\) .)  The curve \(C_C\) is the cuspidal cubic  \(V(p_1^3 - p_2^2)\)  and  \(C_D\)  is a parabola \(V(p_1^2 + 18p_1 - 8p_2 -27)\)  which corresponds to the relative orbit variety (deltoid)\(/D_3\).  The deltoid is the boundary of  \(K(P_{A_2}^d)\).  And it is equal to the set of critical values of \(P_{A_2}^d\)  in   \({R}_2\).  See \cite[p.998]{U2}.  
%The pullbacks of the branch divisors  \(C_C\)  and  \(C_D\)  by \(g_2\)  are written in the form 

%(The similar holds for \(g_3\).)
%\medskip

We will show that 
%%%%%%%%%%%%%%%%%%%%%%%%%%%%%%%equation
\begin{equation*}
g_d(C_C) = C_C, \enskip  \mbox{and}  \quad g_d(C_D )= C_D \quad \mbox{for any} \quad g_d\enskip  (d \in {\mathbb N}) .
\end{equation*}
We will show the dynamics of \(g_d\) on such affine algebraic curves \(V\) where \ \(V = C_C\) \ or \ \(C_D\).  For each affine curve  \(V\),  there exists a polynomial parametrization \(P_V\) satisfying
%%%%%%%%%%%%%%%%%%%%%%%%%%%%%%%equation1.3
\begin{equation}
g_d(P_V(z)) = P_V(T_d(z)).
\end{equation}

From (1.2) and (1.3), we can determine the sets \ \(K(g_d\mid_{V})\), \ for \ \(V = {\mathbb C}^2, \ C_C,\) \ and \(C_D\).  The  set  \(K(g_d\mid_{{\mathbb C}^2})\) \ is a closed region on the real slice \ \({\mathbb R}^2\) \ of \ \({\mathbb C}^2\).  It is invariant under \(g_d\).
The \(K\) sets \ \(K(g_d\mid_{C_C})\) \ and \ \(K(g_d\mid_{C_D})\) \ are two arcs with same end points.  Then we connect these arcs and get a Jordan curve \(\gamma\) on which \(g_d\) is invariant and chaotic. The Jordan curve  \(\gamma\)  is the boundary of the closed region \ \(K(g_d\mid_{{\mathbb C}^2})\)  in  \({\mathbb R}^2\).

Let \(\bar{C_C}\) be the projective closure of  \(C_C\).  The  morphism  \(g_d\)  restricted  to  \(C_C\)  can be extended to a 
morphism  \(\bar{g_d}\)  from \(\bar{C_C}\)  to \(\bar{C_C}\).  We can extend this morphism to a morphism 
from \({\mathbb P}^2({\mathbb C})\)  to   \({\mathbb P}^2({\mathbb C})\).   The same holds for 
\(\bar{C_D}\).

The curves \ \(C_C\) \ and \  \(C_D\) \ are invariant under \ \(g_d\enskip  (d \in {\mathbb N}) \). We will also show that there is an infinite number of invariant affine algebraic curves under \ \(g_d\enskip  (d \in {\mathbb N}) \).

 In Section 4, we study the orbit variety   \({\mathbb C}^3/D_{4}\)  and  consider the dynamics of morphisms  \(g_d\) on  \({\mathbb C}^3/D_{4}\).  We can choose a fundamental system of  invariants of \ \( {\mathbb C}[x_1,  x_2, x_n ]^{ D_{4}} \) such that 
\[ {\mathbb C}[x_1,  x_2, x_3 ]^{ D_{4}}  =  {\mathbb C}[p_1,  p_2, p_3, p_4 ], \]
with \(p_j \in  {\mathbb Z}[x_1,  x_2, x_3 ] \)  and that the  syzygy   ideal   \(I_F\)  for   \(F = (p_1, p_2, p_3, p_4)\)  is written as \(I_F = (p_2p_4-p_3^2)\).  Then \({\mathbb C}^3/D_{4}  \simeq V(I_F)\).  That is , \({\mathbb C}^3/D_{4}\)  has the structure of   the affine quadric hypersurface   \(V(I_F)\)  in  \({\mathbb C}^{4}\). We denote \(V(I_F)\)  by \(X_Q\).   We will show that \(g_d : X_Q \to X_Q\)  is a morphism defined over \({\mathbb Z}\).  

We may choose any other fundamental system of invariants of \ \({\mathbb C}[x,  y, z ]^{ D_{4}} \).  Let \ \(F'\) \ be the ideal generated by that fundamental system of invariants.  Then we have another  syzygy  ideal  \ \(I_{F'}\), \ an  affine algebraic set \ \(V(I_{F'})\) \ and morphisms \ \(g'_d\) on \ \(V(I_{F'})\).  We will show that there is an isomorphism \ \(\varphi : V(I_F)  \to V(I_{F'})\)  such that \ \(\varphi \circ g_d = g'_d \circ \varphi\).

As in (1.2), we will give a polynomial parametrization \ \(P_{X_Q}\) of \(X_Q\) satisfying 
%%%%%%%%%%%%%%%%%%%%%%%%%%%%%%%equation1.4
\begin{equation}
g_d(P_{X_Q}(u, v, w)) = P_{X_Q}(T_d(u), T_d(v), T_d(w)).
\end{equation}

The morphism \(g_2\)   is finite and deg\enskip\(g_2 = 8\) .   The hypersurface \(X_Q\)  is normal.  As in the case \ \({\mathbb C}^2/D_{3}\),  we consider the branch locus of  \(g_2\).  The branch locus of  \(g_2\)  consists of three subvarieties;\\
(1) a line \(L_{1}\),   (2) an affine algebraic surface  \(S_P\),  (3)   an affine algebraic surface  \(S_A\).\\
The subvarieties are written in the form
\[ L_{1} = V(p_2, p_3, p_4),  S_{P}  =  V(p_1^2-p_4,  p_2p_4-p_3^2) ,   S_A = V(A_{h}, p_2p_4-p_3^2) ,\]
where \(A_h\)  is a polynomial in \ \( {\mathbb Z}[p_1,  p_2,  p_3, p_4 ] \). 
Branch divisors   \(S_P\)  and   \(S_A\)  have the similar properties as those in  the former case.  The  line  \(L_{1}\)  is preperiodic for any  \(g_d\) \((d \in {\mathbb N})\).  Set \ \(C_1 = g_2(L_1)\).   Then the affine curve \(C_1\) is invariant under \(g_d\) for any \ \(d \in {\mathbb N}\).

We will show the dynamics of \(g_d\) on the affine varieties \(V\),  where \ \(V = S_A\) \ or \ \(S_P\).  For  each affine variety \(V\),  there exists a polynomial parametrization \(P_V\) satisfying 
%%%%%%%%%%%%%%%%%%%%%%%%%%%%%%%equation1.5
\begin{equation}
g_d(P_V(w, t)) = P_V(T_d(w), T_d(t)).
\end{equation}

Then,  we will show that  \(g_d(S_P) = S_P\)  and  \(g_d(S_A) = S_A\)   for any  morphism  \(g_d\). 
The singular locus of \(S_P\) is the line  \( L_{2} = V(p_1, p_3, p_4)\).  The  line  \(L_{2}\)  is also preperiodic for any  \(g_d\) \((d \in {\mathbb N})\).
Set \ \(C_2 = g_2(L_2)\).  Then the affine curve \(C_2\) is invariant under \(g_d\) for any \ \(d \in {\mathbb N}\).

From (1.4) and (1.5), we can determine the set \ \(K(g_d\mid_{V})\) \ with \ \(d \ge 2\), \ for \ \(V = X_Q, \ S_A\) \ and \(S_P\).  The union \  \(K(g_d\mid_{S_A})\cup K(g_d\mid_{S_P})\) \ is a real 2- dimensional surface  and is invariant under \(g_d\). 
We have \ \(S_A \cap S_P = C_1 \cup C_2 \) \ and \ 
\(\ K(g_d\mid_{S_A})  \cap  K(g_d\mid_{S_P}) = K(g_d\mid_{C_1})  \cup  K(g_d\mid_{C_2}) .\)

The  surface  \(S_A\)  has relation to the relative  orbit variety  \(\mathcal{A}/D_4\),  where  \(\mathcal{A}\) is an astroidalhedron defined in \cite[p.205]{U3}.     It  is shown in \cite{U3} that \(\mathcal{A}\)   is a ruled  surface in the sense of differential geometry.  The projective closure of \(S_A\) is a birationally ruled surface. Indeed.  The  affine algebraic surface  \(S_A\)  is  birationally  equivalent to  an affine quadric cone \(S_C\).  The projective closure of \(S_C\) is a birationally ruled surface.  
%The  blowing-up of the projective closure of the cone \(S_C\)  at the vertex \(P_0\)  is isomorphic to the Hiltzeburch surface  \(F_2\).  
 Let  \(\psi : S_C -\to S_A\)  and \(\phi : S_A -\to S_C\)  be rational maps such that  \(\phi \circ \psi = id_{S_C}\) and \(\psi \circ \phi = id_{S_A}\).  We will show that the map \(\phi \circ g_d \circ \psi \)  on  \(S_C\)   is a morphism defined over \({\mathbb Z}\) and is described by Chebyshev polynomials in one variable.   It can be extended to a morphism on \({\mathbb P}^3({\mathbb C})\).

The singular locus of \(S_A\) is an affine algebraic curve  \(C_A\).  The curve  \(C_A\) corresponds to the relative orbit variety \ \(\alpha_3/D_4\), \ where \(\alpha_3\) \ is the astroid curve in  \({\mathbb R}^3\).   See \cite[p.206]{U3}.   We will show that  \(g_d(C_A) =  C_A\),  for any  \(d \in {\mathbb N}\).

We consider the dynamics of on the affine curves \(V\), where \ \(V = C_1, C_2\) \ or \ \(C_A\).  For each affine curve \(V\),  there exists a polynomial parametrization \(P_V\) satisfying
%%%%%%%%%%%%%%%%%%%%%%%%%%%%%%%equation1.6
\begin{equation}
g_d(P_V(z)) = P_V(T_d(z)).
\end{equation}

We explain our main ideas of the proofs.  We studied complex dynamical system of  
\(P_{A_n}^d\)  on \({\mathbb C}^n\)  in  \cite{U2, U3}.  In this paper, we make use of these results for the research of morphisms on affine algebraic varieties.  In particular we construct parametric representation  \(P_V\) of affine
algebraic varieties \(V\)  by using parametrizations of several sets appeared in complex dynamical system of \  \(P_{A_n}^d\) \  on \({\mathbb C}^n\).   From (1.2),  (1.3),  (1.4), (1.5) and (1.6),  we see that the morphisms \(g_d\) on the affine algebraic varieties behave well in accordance with these parametric representations.  

We also use the theory of  Gr\"obner  bases of ideals.  See  \cite{CLS} .

 K\"u\c{c}\"uksakalli \cite{K1, K2}  studied  Chebyshev maps \  \(P_{A_2}^d\) \ over finite field,  based on \cite{U2}.  Our results will have some relations to arithmetic dynamics. 
%%%%%%%%%%%%%%%Section2
\section{From holomorphic maps to morphisms on varieties}
%%%%%%%%Subsection2.1
\subsection{\rm{Chebyshev endomorphisms and dihedral groups}}\hspace{0cm}

Chebyshev endomorphism  \(P_{A_n}^d\)  of degree  \(d\)  on  \({\mathbb C}^n\)  is associate to the Lie algebra of type \({A_n}\).  For short, we denote \(P_{A_n}^d\) by \(T_d\).

We begin with constructing   Chebyshev endomorphisms \(T_d\) on  \( {\mathbb C}^{n}\).\\
Let \ \(t_j, \ 1 \le j \le n+1, \) be variables satisfying  \(t_1t_2 \cdot \cdot \cdot t_{n+1} = 1\).\\
Let  \ \(z_j, \ 1 \le j \le n, \)   be the \(j\) -th  elementary symmetric polynomial  in  \(t_1, t_2, \dots  ,  t_{n+1}\).

 We define a map \(\Phi\) from   \(( {\mathbb C}^*)^{n+1}\) \ to \ \( {\mathbb C}^{n}\)  by
  %%%%%%%%%%%%%%%%%%%%%%%%%%%%%%%equation2.1
\begin{equation}
\Phi(t_1, \dots, t_{n+1}) = (z_1, \dots, z_{n}). 
\end{equation}
 Let 
\[(z_1^{(d)}, \dots, z_{n}^{(d)}) : = \Phi (t_1^{d}, \dots, t_{n+1}^{d}),  \ d \in  \  {\mathbb N}. \]
Since  \ \(z_j^{(d)}\) \ is a symmetric polynomial  in   \(t_1, t_2,  \dots , t_{n+1} \) \ with coefficients in \({\mathbb Z}\), \ each \ \(z_j^{(d)}\) \ can be expressed as a polynomial in \ \(z_1, \dots , z_{n}\)  \ with coefficients in \({\mathbb Z}\).  We define Chebyshev morphism  \(T_{d}\) on \ \( {\mathbb C}^{n}\)  of degree \  \(d \ge 1\) \ over \({\mathbb Z}\) \  by
  %%%%%%%%%%%%%%%%%%%%%%%%%%%%%%%equation2.2
\begin{equation}
 T_{d} (z_1, \dots, z_{n})  = (z_1^{(d)}, \dots, z_{n}^{(d)}).   
\end{equation} 
See \cite{U2, U3}  for the case that  \(n = 2\) or \(3\).

Note that 
\[z_{n+1-k} = \sum t_{i_1} \cdot \cdot \cdot t_{i_{n+1-k}}  = \sum 1/(t_{j_1} \cdot \cdot \cdot t_{j_{k}}),\] \ where \ \(\{i_1, \dots , i_{n+1-k},  j_1, \dots , j_{k}\} =\{1, 2, \dots , n+1\}.\)\\
Considering the correspondence of  \(t_i\) \ and \ \(1/t_i,  i = 1,  \dots,  n+1\), 
we have 
\[ z_{n+1-j}^{(d)}(z_{1}, \dots , z_{n}) =  z_{j}^{(d)}(z_{n}, \dots , z_{1}) .\]
If \ \(z_{j} = \bar{z}_{n+1-j}\), \ for any \ \(j = 1, \dots , n\),\ then 
\[\overline{z_{j}^{(d)}(z_{1}, \dots , z_{n})} =  z_{j}^{(d)}(z_{n}, \dots , z_{1}) .\]
Hence if \ \(z_j = \bar{z}_{n+1-j}\) for any \(j = 1 \dots , n\), \ then \ \({  z_{n+1-j}^{(d)}} = \overline{ z_{j}^{(d)}},\) \ for any  \  \(j = 1 \dots , n.\)\\
Then we consider a subspace  \(R_{n}\) \ of  \( {\mathbb C}^{n}\) given by 
\[R_{n} : =  \{(z_{1}, \dots , z_{n}) \in   {\mathbb C}^{n} : z_j = \bar{z}_{n+1-j}, \mbox{for any} \ j = 1, \dots , n\}.\]
The maps  \(T_{d}\) admit the invariant subspace  \(R_{n}\). The subspace \(R_n\)  appears also in \cite{B, U1}.

The set of bounded orbits of \(T_{d}\), denoted by  \(K(T_{d})\), is defined by 
\[K(T_{d}) : = \{z \in  {\mathbb C}^{n} : \{T_{d}^m(z) : m = 1, 2,  \dots  \} \ \mbox{is bounded}\},\]
where  \(T_{d}^m\)  denotes the \(m\)-th iterate of  \(T_{d}\).  It is known that \ \(K(T_d) \subset R_n\)  in \cite{U3}.  Since 
\[\Phi(t_1^{-1}, \dots, t_{n+1}^{-1}) = (z_n,  z_{n-1} \dots, z_{1}),\]
we can define  \(T_{-d}(z_1, \dots, z_n)\) \enskip and have that 
\[T_{-d}(z_1, \dots, z_n) = T_d (z_n,  z_{n-1} \dots, z_{1}).\]
The endomorphisms \(T_{-d}\)  admit  also the subspace  \(R_n\).
Note that
%%%%%%%%%%%%%%%%%%%%%%%%%%%%%%%equation2.3
\begin{equation}
T_j \circ T_k = T_{jk} \quad \mbox{for any } \quad j, k \in {\mathbb Z} \setminus \{0\}.
\end{equation}

Next we define actions  \(R\)  and  \(C\) on \( {\mathbb C}^{n}\).  
Set  \(\zeta = e^{2\pi i/(n+1)}\).   We define  
\[R(t_j) = \zeta t_j, \enskip t_j \in {\mathbb C}^*, \enskip j = 1,\enskip \dots , n+1.\]
Note that\quad \(R(t_1)R(t_1) \cdots R(t_{n+1})  = 1\).  Since
\[z_{k} = \sum t_{j_1} \cdots t_{j_{k}}   \ \mbox{where} \ \{j_1, \dots , j_{k}\} \subset \{
 1, 2, \dots , n+1\},\]
we have
\[R(z_k) = \zeta ^kz_k, \quad k = 1, \dots , n.\]
We set, for  \((z_1, \dots , z_n) \in {\mathbb C}^n,\)
\[ R(z_{1}, \dots , z_{n})  :  = ( R(z_1), \dots , R(z_n)) = (\zeta z_{1}, \ \zeta^2 z_{2}, \ \dots , \ \zeta^n z_n) .\]

%%%%pro2.1
 \begin{pro} \label{pro2.1}
 For any \ \(d \in \mathbb Z \setminus \{0\},\) we have
%%%%%%%%%%%%%%%%%%%%%%%%%%%%%%%equation2.4
\begin{equation}
R^d \circ T_d = T_d \circ R .
\end{equation}
 \end{pro}
%%%%%%%%%%%%%%%%%%%%%%%%%%%%%proof
\begin{proof}
We have
\[R^d \circ T_d(z_{1}, \dots , z_{n})  = (\zeta^d z_{1}^{(d)}, \ \zeta^{2d} z_{2}^{(d)}, \ \dots , \ \zeta^{nd} z_n^{(d)}), \ (z_1, \dots , z_n) \in {\mathbb C}^n .\] 
Any element \ \((z_1, \dots , z_n) \) \  determines a set \(\{t_1, \dots , t_{n+1}\}\) \  satisfying (2.1) as roots of the following equation ;
\[t^{n+1} + (-1)z_1t^n + \  . . . \ + (-1)^nz_nt + (-1)^{n+1} = 0.\] 
  Since
\[z_{k}^{(d)} = \sum t_{j_1}^d \cdots t_{j_{k}}^d   \ \mbox{where} \ \{j_1, \dots , j_{k}\} \subset \{
 1, 2, \dots , n+1\},\]
we have
\[\zeta^{kd}z_k^{(d)} = \sum(\zeta t_{j_1})^d  \cdots (\zeta t_{j_{k}})^d.   \]

On the other hand, the element \ \(R(z_{1}, \dots , z_{n})\) \ determines a set \(\{\zeta t_1, \dots , \zeta t_{n+1}\}\) \  satisfying (2.1) .  Then the \(k\)-th element of  \ \(T_d \circ R(z_{1}, \dots , z_{n})\) \ is equal to
\[\sum(\zeta t_{j_1})^d  \cdots (\zeta t_{j_{k}})^d.   \]
\end{proof}

Next we define an action  \(C\)  on  \({\mathbb C}^n\).  We define 
\[C(t_j) = \bar{t}_j, \  t_j \in {\mathbb C}^*, \ j=1,  \dots , n+1.\]
Note that 
\[C(t_1) \ \cdots \ C(t_{n+1}) = 1.\]
Clearly \   \(C(z_k) = \bar{z}_k\).  Since \ \(z_k^{(d)}\)  is a polynomial in  \ \(z_1, \ \dots\ , z_n\) \  over \({\mathbb Z} \),   it follows that 
\[\overline{T_d(z_1, \dots z_n)} =  T_{d}(\bar{z}_{1}, \dots , \bar{z}_{n}) .\]
Set
\[ C(z_{1}, \dots , z_{n})  :  = ( C(z_1), \dots , C(z_n)) = (\bar {z}_{1}, \ \dots , \ \bar{ z}_n) .\]
Then we have the following.
%%%%pro2.2
 \begin{pro} \label{pro2.2}
 For any \ \(d \in \mathbb Z \setminus \{0\},\) we have
%%%%%%%%%%%%%%%%%%%%%%%%%%%%%%%equation2.5
\begin{equation}
C \circ T_d = T_d \circ C .
\end{equation}
 \end{pro}

Clearly
\[CRC(z_{1}, \dots , z_{n}) = (\bar{\zeta}z_1, \bar{\zeta}^2z_2, \dots , \bar{\zeta}^nz_n) = ( {\zeta}^{-1}z_1,   {\zeta}^{-2}z_2, \dots ,  {\zeta}^{-n}z_n)\]
\[=( {\zeta}^{n}z_1,   {\zeta}^{2n}z_2, \dots ,  {\zeta}^{n^2}z_n) = R^{n}(z_1, \dots\ , z_n) = R^{-1}(z_1, \ \dots\ , z_n).\]
Hence  \(CRC = R^{-1}\).
Then the group generated by the two elements  \(R\)  and  \(C\)  is the dihedral group  \  \(D_{n+1} = <R, C>\) \ of order  \  \(2n+2\).

We consider the\ \(D_{n+1}\)-orbit of \ \({\bf z} \in{\mathbb C}^n\) \ defined by
\[D_{n+1}\cdot{\bf z} : = \{\pi {\bf z} : \pi \in D_{n+1}, \ {\bf z} \in {\mathbb C}^n\}.\]
Then by Propositions \ref{pro2.1}  and \ref{pro2.2},  we  see that  \(T_d\) maps an orbit \  \(D_{n+1}\cdot{\bf z}\)  \ to another orbit 
\(D_{n+1}\cdot{T_d}({\bf z})\).  Hence we can define an induced map of \(T_d\) on the orbit space 
\({\mathbb C}^n/D_{n+1}\).

Recall that \(R_n\) is the subspace of \ \({\mathbb C}^n\).  Clearly,
\[R(R_n) \subset R_n \ \mbox{and}\ C(R_n) \subset R_n .\]
Hence the maps \  \(T_d, \ R\) \ and \(C\) admit the subspace \ \(R_{n}\) .    Then from (2.4) and (2.5) we have  
%%%%%%%%%%%%%%%%%%%%%%%%%%%%%%%equation2.6
\begin{equation}
(R \mid _{R_n})^d \circ (T_d \mid _{R_n}) = (T_d \mid _{R_n}) \circ (R \mid _{R_n}),
\end{equation}
%%%%%%%%%%%%%%%%%%%%%%%%%%%%%%%equation2.7
\begin{equation}
(C \mid _{R_n}) \circ (T_d \mid _{R_n}) = (T_d \mid _{R_n}) \circ (C \mid _{R_n}).
\end{equation}
Note that \ \(R_{n} \simeq  {\mathbb R}^{n} \).   Sometimes we regard \(R_{n}\)  as   \({\mathbb R}^{n} \).  We set
\[z_1 = x_1 + ix_2, \  z_2 = x_3 + ix_4, \dots  ,\ \   z_{n/2} = x_{n-1} + ix_n \ \mbox{ if}  \ n \ \mbox{ is\enskip even},\]
and
\[z_1 = x_1 + ix_2, \dots , \    z_{(n-1)/2} = x_{n-2} + ix_{n-1},  \  z_{(n+1)/2} = x_{n}\ \mbox{ if} \ n \ \mbox {is\enskip odd}.\]

Then  we have
%%%%%%%%%%%%%%%%%%%%%%%%%%%%%%%equation2.8
\begin{equation}
 ^t(z_1, \dots , z_n)  = M_n\ ^t(x_1, \dots , x_n),
\end{equation}
where \(M_n\)  is a non-singular matrix.

We denote the map  \(T_d\)  restricted to  \(R_n\) by  \ \(f_d(x_1, \dots , x_n)\).  Note that \ \(f_d(x_1, \dots , x_n)\) \ is a polynomial map defined over  \({\mathbb Z}\).   Properties of \ \(f_d(x_1, x_2)\) \ are studied in \cite{U1} and \cite{U2}.  The generators  \(R\) and  \(C\) restricted to  \(R_n\)  can be represented by matrices in \ \(GL(n,  {\mathbb R}) \).  Those are written in the following block-diagonal matrices.  The generator \ \(R \mid_{R_n}\) \ is represented in the form:

\[R \mid_{R_n}  = \left(
  \begin{array}{cccccc}
      R_{11} &           &        &          &   &   \\
               & R_{22}  &        &         &   $\mbox{\Huge0}$ &   \\
               &    &    &     &    &   \\
                 &   &   \ddots  & & \\
               &    &    &     &    &   \\
               &   $\mbox{\Huge0}$   &    &     &    &   \\
     &   &    &     &  & R_{mm}  
    \end{array}
  \right) .\]
If \(n\) is even,  \(m = n/2\) \  and \  \(R_{jj} \)   is a block of (2, 2) type  defined by
\[R_{jj} = \left(
  \begin{array}{ccc}
      \cos\frac{2\pi j}{n+1} &           &  - \sin\frac{2\pi j}{n+1}      \\
      \sin\frac{2\pi j}{n+1}   &   & \cos\frac{2\pi j}{n+1}   \\
               \end{array}
  \right) , \  \mbox{for} \  \ 1 \le j \le m.\]
If \(n\) is odd,  \(m = (n+1)/2\) \  and \  \(R_{jj} \)   is  the same as above for  \ \(1 \le j \le (n-1)/2\) \ and \ \(R_{mm}  = (-1)\)  \  is a block of (1, 1) type.   The generator \ \(C \mid_{R_n}\) \ is represented in the form:

\[C \mid _{R_n} = \left(
  \begin{array}{cccccc}
      C_{11} &           &        &          &   &   \\
               & C_{22}  &        &         &   $\mbox{\Huge0}$ &   \\
               &    &    &     &    &   \\
                 &   &   \ddots  & & \\
               &    &    &     &    &   \\
               &   $\mbox{\Huge0}$   &    &     &    &   \\
     &   &    &     &  & C_{mm}  
    \end{array}
  \right) .\]
If \(n\) is even,  \(m = n/2\) \  and \  \(C_{jj} \)   is a block of (2,2) type  defined by
\[C_{jj} = \left(
  \begin{array}{cc}
      1  &    0     \\
      0   &  -1   \\
               \end{array}
  \right) , \  \mbox{for} \  \ 1 \le j \le m.\]

If \(n\) is odd,  \(m = (n+1)/2\) \  and \  \(C_{jj} \)   is  the same as above for  \ \(1 \le j \le (n-1)/2\) \ and \ \(C_{mm}  = (1)\)  \  is a block of (1, 1) type.
%%%%%%%%Subsection2.2
\subsection{\rm{Morphisms on  orbit varieties}}\hspace{0cm}

From (2.6) and (2,7) we can define an induced map of \(f_d\) on \ \({\mathbb R}^n/D_{n+1}\), \ where \(D_{n+1}\) is the finite matrix group in \ \(GL(n,  {\mathbb R})\).   Let \(k\) be the field  \({\mathbb R}\) \ or the field\ \({\mathbb C}\) .  The \  \(D_{n+1}\)-orbit  of \ \({\bf a} \in k^n\) \  is a set 
\[D_{n+1}\cdot {\bf a} = \{A\cdot{\bf a}  : A \in  D_{n+1}\}. \] 
The set of  all  \(D_{n+1}\)-orbits  in  \(k^n\)  is denoted by \ \(k^n/D_{n+1}\) \ and is called the orbit space.  

%Now we consider complexification.  We will construct a morphism  \ \(g_d\) \ on the orbit variety \ \({\mathbb C}^n/D_{n+1} \simeq V_F\).

We begin with the construction of the maps on \ \({\mathbb R}^n/D_{n+1}\).  We use invariant  theory.  See \cite{CLS}, \cite{DK} and \cite{St}.

Let  \ \(k[x_1, \cdots , x_n]^{D_{n+1}}\) \ is the set of invariant polynomials  \ \(f({\bf x}) \in k[x_1, \cdots , x_n]\)  \ such that  \ \(f({\bf x}) = f(A\cdot{\bf x})\), \ for all  \ \(A \in D_{n+1}\).  By Emmy Noether's theorem (Theorem 5 in \cite[Chapter 7, \S3]{CLS}),  we know that \ \(k[x_1, \dots , x_n]^{D_{n+1}}\) \ is generated by finitely many homogeneous invariants in \  \(k[x_1, \dots , x_n]\).   We denote these homogeneous invariants by \ \(p_1({\bf x}), \dots , p_m({\bf x})\) .  
Since \  \(D_{n+1} \subset GL(n,  {\mathbb R})\),  \ by the proof of Theorem 5 in \cite[Chapter 7, \S3]{CLS},  
we can select a basis \ \(\{p_1({\bf x}), \dots , p_m({\bf x})\}\)\ 
%with \ \(p_j({\bf x})  \in  { \mathbb R}[x_1, \dots , x_n] \ (j = 1, \dots , m) \) 
\ such that 
%%%%%%%%%%%%%%%%%%%%%%%%%%%%%%%equation2.9
\begin{equation}
\begin{split}
 { \mathbb R}[x_1, \dots , x_n]^{D_{n+1}} = { \mathbb R}[p_1, \dots , p_m],\\
\mbox { and} \hspace{5cm}\\  
 { \mathbb C}[x_1, \dots , x_n]^{D_{n+1}} = { \mathbb C}[p_1, \dots , p_m].
\end{split}
\end{equation}
%Note that \ \( F = (p_1, \dots , p_m).\)  Hence we can define the map \( \tilde{F}\) similarly in two cases.
%Note that \  \(D_{n+1} \subset GL(n,  {\mathbb R})\).  Then by the proof of the above theorem, we can see that \ \(p_j({\bf x})\) \ is a polynomial defined over \ \({\mathbb R}, \ j = 1, \dots , m\).

We need some more results on invariant theory.  See \cite[Chapter 7, \S4]{CLS}.  Using a basis \ \(\{p_1, . . . , p_m\}\) \ satisfying (2.9),  we have  that \ \(k[x_1, \dots , x_n]^{D_{n+1}} = k[p_1, \dots , p_m] \).  We let the ideal  \ \(F = (p_1, \dots , p_m) \).  Then we define the syzygy ideal \(I_F\) for \(F\) :
\[I_F : = \{h \in k[y_1, \dots , y_m] : h(p_1, \dots , p_m) = 0 \ \mbox{in} \  k[x_1, \dots , x_n]\}.\]
Then we have the affine  variety  \ \(V_F : = V(I_F)\).  Let  \(k[V_F]\)  be the coordinate ring of  \(V_F\).  Then there is a ring isomorphism
 \ \(k[V_F] \simeq k[x_1, \dots , x_n]^{D_{n+1}}\).
%Theorem2.3
\begin{theorem} \label{theorem1}
(\cite[Chapter 7, \S4, Theorem10]{CLS})

We assume that \ \(k = {\mathbb C}\).  Then :\\
(1)  The polynomial map  \(\tilde{F} :{ \mathbb C}^n \to  V_F\) \ defined by \ \(\tilde{F} ({\bf a}) = (p_1({\bf a}), \dots , p_m({\bf a}))\) \ is subjective.\\
(2)  The map sending the\ \(D_{n+1}\)-orbit \  \(D_{n+1}\cdot {\bf a} \subset { \mathbb C}^n\) \ to a point \ \(\tilde{F} ({\bf a}) \in V_F\)  \ includes a one-to-one correspondence \ 
\( { \mathbb C}^n / D_{n+1} \simeq V_F\).
\end{theorem}

We recall that the map \ \(f_d({\bf x})\) \ is the map \(T_d\) restricted to \ \(R_n \simeq { \mathbb R}^n \).  Then \(f_d\) is the map from \  \( { \mathbb R}^n \) \ to  \( { \mathbb R}^n \).  We  know from (2.9)  that  \ \(p_j({\bf x})\) \ is a polynomial over  \( { \mathbb R}\).  Let 
\[P_j({\bf x}) : = p_j(f_d({\bf x})), \  j = 1, \dots , m.\] 
Then we define a map \ \(\tilde{f}_d\) from \ \(\tilde{F}( { \mathbb R}^n) \) \ to  \({ \mathbb R}^m \) \ by 
\[\tilde{f}_d(p_1({\bf x}), \dots , p_m({\bf x})) = (P_1({\bf x}), \dots , P_m({\bf x})).\]
We will verify that \ \(\tilde{f}_d\) \ is well-defined.  We suppose \ \(\tilde{F}({\bf x}) = \tilde{F}({\bf x}^\prime) \) \ for\  \({\bf x} , {\bf x}^\prime \in   {\mathbb R}^n\).
From Theorem \ref{theorem1}(2), we see that \ \( {\bf x}^\prime \)  belongs to the \ \(D_{n+1}\)-orbit \(D_{n+1}\cdot{\bf x}\).  Then there exists an element \ \(A \in D_{n+1}\) \ such \  \({\bf x}^\prime =  A\cdot{\bf x}\).  Then by (2.6) and (2.7), there exists an element \ \(B \in D_{n+1}\) \ such that \  \(f_d(A\cdot{\bf x}) = B\cdot f_d({\bf x})\).  Hence
\[P_j(A\cdot{\bf x})  = p_j(f_d(A\cdot{\bf x})) =  p_j(B\cdot f_d({\bf x})) = p_j(f_d({\bf x}))  = P_j({\bf x}) .\] 
Then \ \(\tilde{f}_d\) \ is well-defined.  And we also show that 
\[P_j(A\cdot{\bf x})  = P_j({\bf x}), \ \mbox{for any} \ A \in D_{n+1}. \]
Then 
\[P_j({\bf x})  \in  { \mathbb R}[x_1, \dots , x_n]^{D_{n+1}} = { \mathbb R}[(p_1({\bf x}), \dots , p_m({\bf x})].\]
Hence\ \(P_j({\bf x}) \) \ is a polynomial \(P_j(p_1({\bf x}), \dots , p_m({\bf x}))\)\ in \(p_1({\bf x}), \dots , p_m({\bf x})\)\ over\ \({ \mathbb R}\).  Here we use the same symbol  \(P_j\)  for a polynomial of  \(p_1, \dots , p_m\).
We set 
\[f({\bf x}):  = P_j(p_1({\bf x}), \dots , p_m({\bf x})) - (p_j(f_d({\bf x})).\]
Then 
\(f \in  { \mathbb R}[x_1, \dots , x_n] \ \mbox{and}\   f :   { \mathbb R}^n \to  { \mathbb R}\) is the zero function.  Hence  by Proposition 5 in \cite[Chapter 1, \S1]{CLS},  we have \ \(f = 0\)  in  \( { \mathbb R}[x_1, \dots , x_n] \).
 That is,
 %%%%%%%%%%%%%%%%%%%%%%%%%%%%%%%equation
\begin{equation*}
P_j(p_1({\bf x}), \dots , p_m({\bf x})) \equiv  p_j(f_d({\bf x})).
\end{equation*}
Then
%%%%%%%%%%%%%%%%%%%%%%%%%%%%%%%equation2.10
\begin{equation}
\begin{split}
\tilde{f}_d(p_1({\bf x}), \dots , p_m({\bf x})) = (P_1(p_1, \dots , p_m)({\bf x}), \dots , P_m(p_1, \dots , p_m)({\bf x})) ,\\
\mbox { where} \  P_j(p_1, \dots , p_m)({\bf x}) : = P_j(p_1({\bf x}), \dots , p_m({\bf x}) ),  j = 1, \dots , m.
\end{split}
\end{equation}
And so, for\  \({\bf x} \in { \mathbb R}^n\),  we have
%%%%%%%%%%%%%%%%%%%%%%%%%%%%%%%equation2.11
\begin{equation}
\tilde{f}_d \circ \tilde{F}({\bf x}) = \tilde{F} \circ f_d({\bf x}) .
\end{equation}
We have the following commutative diagram.
\def\spmapright#1{\smash{
  \mathop{\hbox to 1.3cm{\rightarrowfill}}
   \limits^{#1}}}
 \def\sbmapright#1{\smash{
  \mathop{\hbox to 1.3cm{\rightarrowfill}}
   \limits^{#1}} } 
\def\lmapdown#1{\Big\downarrow \llap{$\vcenter{\hbox{$\scriptstyle#1\,$}}$}}
\def\rmapdown#1{\Big\downarrow \rlap{$\vcenter{\hbox{$\scriptstyle#1\,$}}$}}
\def\lmapup#1{\Big\uparrow \llap{$\vcenter{\hbox{$\scriptstyle#1\,$}}$}}
\def\rmapup#1{\Big\uparrow \rlap{$\vcenter{\hbox{$\scriptstyle#1\,$}}$}}
\def\lmapupdown#1{\Big\updownarrow  \llap{$\vcenter{\hbox{$\scriptstyle#1\,$}}$}}
%%%%%%equation 2.12
\begin{equation}
\begin {array}{ccc}
  \{t_1,  \dots ,  t_{n+1}\} & \spmapright{ } & \{t_1^d, \dots , t_{n+1}^d\}\\
  \lmapupdown{} &  & \lmapupdown{} \\
  (z_1, \dots , z_n) & \sbmapright{T_d} & (z_1^{(d)},  \dots , z_n^{(d)}) \\
 \rmapup{M_n} &  & \rmapup{M_n} \\
  (x_1, \dots , x_n) & \sbmapright{f_d} & (X_1, \dots , X_n) \\
\lmapdown{} &  & \lmapdown{} \\
  (p_1, \dots ,  p_m) & \sbmapright{} & (P_1, \dots , P_m),
 \end{array}
\end{equation}
where \ \(T_d :  { \mathbb C}^n  \to   { \mathbb C}^n\) \ and \ \(f_d :  { \mathbb R}^n  \to  { \mathbb R}^n\).

To use the correspondence \ \( { \mathbb C}^n / D_{n+1} \simeq V_F\)  in Theorem \ref{theorem1}(2), we consider complexification.  We regard \ \(x_1, \dots , x_n\) \ as complex variables and \(f_d\) as a polynomial map from \  \( { \mathbb C}^n \) \ to  \( { \mathbb C}^n \).
The map  \(M_n\) in (2.8)  is an invertible linear map from \( { \mathbb C}^n\) \ to \( { \mathbb C}^n\).
We consider also complexification of  \( \tilde{f_d}\).  Recall that the bijection caused by \( \tilde{F}\) gives \ \({ \mathbb C}^n / D_{n+1}\) \ the structure of the affine algebraic variety  \ \(V_F\).   Then from (2.10), we can define a map  \(g_d\) from \(V_F\) to \(V_F\).

Then we have the following diagram. 
%%%%%%equation 2.13
\begin{equation}
\begin {array}{ccc}
  \{t_1,  \dots ,  t_{n+1}\} & \spmapright{ } & \{t_1^d, \dots , t_{n+1}^d\}\\
  \lmapupdown{} &  & \lmapupdown{} \\
  (z_1, \dots , z_n) & \sbmapright{T_d} & (z_1^{(d)},  \dots , z_n^{(d)}) \\
 \rmapdown{M_n^{-1}} &  & \rmapdown{M_n^{-1}} \\
  (x_1, \dots , x_n) & \sbmapright{f_d} & (X_1, \dots , X_n) \\
\rmapdown{\tilde{F}} &  & \rmapdown{\tilde{F}} \\
  (p_1, \dots ,  p_m) & \sbmapright{g_d} & (P_1, \dots , P_m),
 \end{array}
\end{equation}
where
 \(t_j \in  { \mathbb C}^*\),   
\(T_d :  { \mathbb C}^n  \to   { \mathbb C}^n,\)\ \(f_d :  { \mathbb C}^n  \to   { \mathbb C}^n\) \ and \ \(g_d :  V_F  \to  V_F\).

Then the element \((t_1, \dots , t_{n+1})\)  may take any value of \ \(({\mathbb C}^*)^{n+1}\)  with \(t_1 \cdots  t_{n+1} = 1.\)
Note that  \(T_d\) and \(f_d\) are polynomial maps and the diagram (2.13) is commutable if \ \((x_1, \dots , x_n) \in { \mathbb R}^n.\)  Then the diagram (2.13) is also commutable for any  \((x_1, \dots , x_n) \in { \mathbb C}^n.\)

%Definition2.4
\begin{definition} \label{definition1}
We define a map  \(g_d\) from \(V_F\) to \(V_F\)  by 
\[g_d (p_1, \dots , p_m) = (P_1 (p_1, \dots , p_m), \dots , P_m (p_1, \dots , p_m)).\]
\end{definition}

Note that \(P_j\) is a polynomial in \ \(p_1, \dots , p_m\) \ over  \(\mathbb R \).
Then we have the following.
%%%%pro2.5
 \begin{pro} \label{pro2.5}
 Let \(g_d\) be a map in Definition 2.4.  The map  \(g_d\)  is a morphism from \ \(V_F\)  to \(V_F\)  defined  over \(\mathbb R \)
 and is surjective.
\end{pro}
%proof
\begin{proof}
It is enough  to prove that \(g_d\) is surjective.    Note that the vertical arrows in (2.13) are surjective and that  the morphism \(\{t_1, \dots , t_{n+1}\} \mapsto \{t_1^d, \dots , t_{n+1}^{d}\}\)  is surjective.  Then the morphism \(g_d\) is surjective.
\end{proof}
Let\ \( \{p_1({\bf x}), \dots , p_m({\bf x})\}\) \ be a basis of \ \( { \mathbb C}[x_1, \dots , x_n]^{D_{n+1}}\) satisfying (2.9).  Then by (2.10) and (2.11) we have 
%%%%%%%%%%%%%%%%%%%%%%%%%%%%%%%equation2.14
\begin{equation}
\begin{split}
g_d(p_1({\bf x}), \dots , p_m({\bf x})) = (P_1(p_1, \dots , p_m)({\bf x}), \dots , P_m(p_1, \dots , p_m)({\bf x})) \\
 =  {\hat f}_d(p_1({\bf x}), \dots , p_m({\bf x})),\qquad \qquad \qquad\qquad
\end{split}
\end{equation}
\hspace{2cm}           for any \ \({\bf x} \in  { \mathbb C}^n\).

In  Sections 3 and 4, we study the maps  \(f_d\)  on   \({\mathbb C}^2 \)  or  \({\mathbb C}^3 \)  and their induced maps \(g_d\).  In the cases, we can select a basis  \ \(\{p_1({\bf x}), \dots , p_m({\bf x})\}\)  satisfying the following properties : \\
\({\mbox (1)} \ p_j({\bf x}) \in  { \mathbb Z}[x_1, \dots , x_n], \  (j = 1, \dots , m), \ (n = 2\  {\mbox or}\  3),\)\\
(2) \ \(g_d\) is a morphism  defined over\ \( { \mathbb Z}\).

We consider the set \ \(\{g_d : d \in  { \mathbb Z}\setminus \{0\}\}\) of morphisms. 

Since \  \( f_1({\bf x}) = {\bf x}, \  {\tilde f}_1(p_1({\bf x}), \ \dots \ , \ p_m({\bf x})) =   (p_1({\bf x}),\ \dots \ , \ p_m({\bf x}))\).
Then  \(g_1\)  is an identity map.
  We show the set \ \(\{g_d : d \in  { \mathbb Z}\setminus \{0\}\}\) is a commutative monoid.
%pro2.6
\begin{pro} \label{pro2.6}
We have \ \(g_e \circ g_d = g_{de}, \) \ for any \ \( d, e \in  { \mathbb Z}\setminus \{0\}.\)
\end{pro}
%%%%%%%%%%%%%%%%%%%%%%%%%%%%%proof
\begin{proof}
By (2.3),  we have\ \(T_e \circ T_d = T_{de}. \)  If  \(f_d\) is the map  \(T_d\) restricted to  \(R_n\),  then \ \(f_e \circ f_d = f_{de} \).  If  we view \ \(f_d\)  as a map from  \({ \mathbb C}^n  \to   { \mathbb C}^n\),  then the same formula holds.  We set
%%%%%%%%%%%%%%%%%%%%%%%%%%%%%%%equation2.15
\begin{equation}
g_e \circ g_d(p_1, \dots , p_m) - g_{de}(p_1, \dots , p_m) = h(p_1, \dots , p_m). 
\end{equation}
Note that  \ \(h(p_1, \dots , p_m) \in  { \mathbb R}[p_1, \dots , p_m].\)

We substitute  \(p_j({\bf x})\) \ for \ \(p_j, \ (j = 1, \dots , m)\) \ in (2.15).
By (2.10), (2.11) and (2.14), we deduce
\[g_d(p_1({\bf x}), \dots , p_m({\bf x})) = (P_1(p_1, \dots , p_m)({\bf x}), \dots , P_m(p_1, \dots , p_m)({\bf x})) \]
\[= {\hat f}_d(p_1({\bf x}), \dots , p_m({\bf x})) = \tilde{f}_d \circ \tilde{F}({\bf x}) = \tilde{F} \circ f_d({\bf x}) , \]
\[\mbox { where} \  P_j(p_1, \dots , p_m)({\bf x}) : = P_j(p_1({\bf x}), \dots , p_m({\bf x}) ),  j = 1, \dots , m.\]
And we have
\[g_e (P_1(p_1, \dots , p_m)({\bf x}), \dots , P_m(p_1, \dots , p_m)({\bf x}))  = {\tilde f}_e(P_1(p_1, \dots , p_m)({\bf x}), \dots , P_m(p_1, \dots , p_m({\bf x}))\]
\[= {\tilde f}_e \circ \tilde{F} \circ f_d({\bf x}) = \tilde{F} \circ f_e \circ f_d({\bf x}) = \tilde{F} \circ f_{ed}({\bf x}). \]
Hence \ \(g_e \circ g_d(p_1({\bf x}), \dots , p_m({\bf x})) =  g_{de}(p_1({\bf x}), \dots , p_m({\bf x})).\)\\
Then \ \(h(p_1({\bf x}), \dots , p_m({\bf x})) = 0\), \  for any \ \({\bf x} \in  { \mathbb C}^n\).
Therefore \ \(h(y_1, \dots , y_m) \) \ belongs to the syzygy ideal  \(I_F\)  for \(F\).  Thus, for any \ \((p_1, \dots , p_m) \in V_F = V(I_F)\),\ we have  \ \(g_e \circ g_d(p_1, \dots , p_m) =  g_{de}(p_1, \dots , p_m).\)
\end{proof}
%%%%%%%%%%%%%%%%%%%%%%%%%%%%%%%%%%%%%%Corollary 2.7
\begin{cor} \label{cor1}
For any  \ \(d \in  { \mathbb N}\), we have \ \(g_{-d} = g_d\).
\end{cor}
%%%%%%%%%%%%%%%%%%%%%%%%%%%%proof
\begin{proof}
 It suffices to show\ \(g_{-1} = g_1\).    We see in Section 2.1 that
\[ T_{-1}(z_{1}, \dots , z_{n})  =T_1(\bar {z}_{1}, \ \dots , \ \bar{ z}_n) \ {\mbox in} \ R_n .\]
Then \ \(f_{-1}(x_1, x_2, \dots , x_n) = (x_1,  -x_2, x_3, \dots, (-1)^{n-1}x_n)\).

Recall that  \ \(f_1(x_1, \dots , x_n) = (x_1,   \dots , x_n)\).
Then   \ \((C\mid _{R_n})(f_1(x_1, \dots, x_n) = f_{-1}(x_1,   \dots,  x_n)\).
Hence\ \(p_j(f_{-1}({\bf x}))  = p_j( {\bf x}),  \ j = 1, \dots, m \).  Therefore 
\[\tilde{f}_{-1}(p_1({\bf x}), \dots , p_m({\bf x})) = (p_1({f}_{-1}({\bf x}), \dots , p_m({f}_{-1}({\bf x})) =  (p_1({\bf x}),  \dots , p_m({\bf x})) .\] 
Thus we have \ \(g_{-1} = g_1\).
\end{proof}

From now on we consider only morphisms \ \(g_d\) \ with \ \(d \in  { \mathbb N}\).

We consider the properties of the morphisms \(g_d\).
We use the following theorem.
%%%%%%%%%%%%%%%%%%%%%%%%theorem 2.8
\begin{theorem} \label{theorem6} (\cite[II, \S5, 3, Theorems 6 and 7 ]{Sh})

 If a morphism \ \(f : X \to Y\) \ is finite, dominant and separable, the varieties \(X\) and \(Y\) are irreducible of equal dimension, and \(Y\) is normal, then any point of some non-empty subsets \ \(U \subset Y\) \ has \ deg \(f\)  \ distinct inverse images, and any point of the complement has fewer inverse images.
\end{theorem}

We have the following proposition.
%pro2.9
\begin{pro} \label{pro2.9}
(1)  The affine algebraic variety \(V_F \simeq { \mathbb C}^n / D_{n+1} \)  is normal.\\
(2)  We assume that \(g_d\)  is a finite morphism.  Then \ deg  \( g_d = d^n\).
\end{pro}
%%%%%%%%%%%%%%%%%%%%%%%%%%%%%proof
\begin{proof}
(1)  Clearly \( { \mathbb C}[x_1, \dots, x_n]\) is integrally closed.  Then by Propsition 2.3.11 in \cite{DK} we know that \( { \mathbb C}[x_1, \dots, x_n]^{D_{n+1}}\)  is also integrally closed. 
Recall that \( { \mathbb C}[V_F]  \simeq  { \mathbb C}[x_1, \dots, x_n]^{D_{n+1}}\).
Then by \cite[II, \S5 ]{Sh},  the affine variety \(V_F\) is normal.\\
(2)  Note that the vertical arrrows in the diagram (2.13) are at most finite surjective morphisms.  Then it is easy to see that the degree of \(g_d\) is the same as the morphism \(\{t_1, \dots , t_{n+1}\} \mapsto \{t_1^d, \dots , t_{n+1}^{d}\}\).  Then \ deg \( g_d = d^n\).
\end{proof}

%%%%%%%%%%%%%%%%%%%%%%%%%%%%%%%%%%%%%%%%Section 3
\section{Morphisms on \ \( { \mathbb C}^2 / D_{3}\)}
%%%%%%%%Subsection3.1
\subsection{\rm{Morphisms and branch locus}}\hspace{0cm}

We begin to study the orbit variety \ \( { \mathbb C}^2 / D_{3} \simeq V_F\).  We set \ \(x : = x_1\) \ and \  \( y : = x_2\).   By Algorithm 2.5.14 in \cite{St},  we know that \ \( { \mathbb C}[x, y]^{D_3} =  { \mathbb C}[x^2+y^2,  x^3-3xy^2]\).  See also Exercises(3) in  \cite[\S2.2]{St}.  Clearly \ \(x^2+y^2\) \ and \ \(x^3-3xy^2\) \ are algebraically independent.  Set \ \(p = x^2+y^2,   q = x^3-3xy^2\) \ and \ \(F = (p, q)\).  Then the syzygy ideal \ \(I_F\) for \(F\) is equal to the zero ideal.  Hence \ \(V_F = V(I_F) \simeq { \mathbb C}^2 \).    We use a coordinate \ \((p, q)\) of \ \( { \mathbb C}^2 \). 

We consider another fundamental system \(\{p^\prime, q^\prime \}\) of invariants of  \ \( { \mathbb C}[x, y]^{D_3} \).  By Molien' theorem (\cite[Theorem 2.2.1]{St}) , the Hilbert series \  \(\Phi_{D_3}(t)\)\  of the invariant ring \ \( { \mathbb C}[x, y]^{D_3} \) \ is written as 
\[\Phi_{D_3}(t) = 1 + t^2 + t^3 +  t^4 + t^5 +  2t^6 + t^7 +  2t^8 + 2t^9  + \it{O}(t^{10}).\]
Then \ \(dim( { \mathbb C}[x, y]_2^{D_3}) =  dim( { \mathbb C}[x, y]_3^{D_3}) = 1\).
Since  \( x^2 + y^2\)  and  \(x^3 - 3xy^2 \in { \mathbb C}[x, y]^{D_3} \), \ then it follows that  \ \(p^\prime  = a( x^2 + y^2)\) and \ \(q^\prime  = b(x^3 - 3xy^2)\), where \(a\) and  \(b\) are non-zero complex numbers.  Hence the fundamental system
\ \(\{p, q\}\) \ of invariants of 
\ \( { \mathbb C}[x, y]^{D_3} \) is unique up to constants.

By \cite[p.996]{U2}, we have \ \(f_2(x, y) = (x^2 - y^2 -2x, \ 2xy+ 2y)\), \ \(f_3(x, y) = (x^3 - 3xy^2 - 3x^2 - 3y^2 +3, \ 3x^2y - y^3 ).\)
We compute two morphisms \ \(g_2\) and \ \(g_3\) on \(V_F\) : 
%%%%%%%%%%%%%%%%%%%%%%%%equation 3.1
\begin{equation}
g_2(p, q) = (4 p + p^2 - 4 q, 12 p^2 - 8 q - 6 p q + 2 q^2 - p^3),
\end{equation}
%%%%%%%%%%%%%%%%%%%%%%%%equation 3.2
\begin{equation}
\begin{split}
g_3(p, q) = 
(9 - 18 p + 9 p^2 + p^3 + 6 q - 6 p q,  \qquad \qquad \qquad \qquad \qquad \qquad\\
9 p^4-3 p^3 q-36 p^3+27 p^2 q+81 p^2-18 p q^2-54 p q-81 p+4 q^3+18 q^2+27 q+27).
\end{split}
\end{equation}

These are morphisms defined over  \( { \mathbb Z}\).  This fact holds for general  morphisms \ \(g_d\) on \ \( { \mathbb C}^2 / D_{3}\).
%%%%%%%%%%%%%%%%%%%%%%%%theorem 3.1
\begin{theorem} \label{theorem2}\hspace{0cm}\\
(1) Let \ \(p = x^2+y^2 \  {\mbox and} \  q = x^3-3xy^2\).  Then we have \({ \mathbb R}[x, y]^{D_3} \cap  { \mathbb Z}[x, y] = { \mathbb Z}[p, q].\)  \\
(2) Let  \(g_d\) \ be a morphism on \ \( { \mathbb C}^2 / D_{3}\)  \ in Definition 2.4.  
Then \ \(g_d\)  is  defined over  \( { \mathbb Z}\)  for any \(d \in  { \mathbb N}. \)
\end{theorem}
%%%%%%%%%%%%proof
\begin{proof}\quad (1) Suppose  \(P(p, q)\) be a polynomial such that
%%%%%%%%%%%%%%%%%%%%%%%%equation 3.3
\begin{equation}
P(p({\bf x}),  q({\bf x})) \in  { \mathbb R}[p({\bf x}),  q({\bf x})] \cap  { \mathbb Z}[x, y] .
\end{equation}

Then we will show that 
%%%%%%%%%%%%%%%%%%%%%%%%equation 
\begin{equation*}
P(p,  q) \in  { \mathbb Z}[p, q].
\end{equation*}

We define the xy-degree of \ \(p^iq^j\) to be \ \(2i + 3j\).   Let \ \(h_k(p, q)\)  denote the sum of terms \ \(a_{ij}p^iq^j\) \ in \(P(p,  q) \) \ with the xy-degree  \(k\).  Then 
\[P(p,  q) =   \sum_kh_k(p, q).\]
It is enough to prove that \ \(h_k(p,  q) \in  { \mathbb Z}[p, q]\) \ for any  \(k\).  Clearly
%%%%%%%%%%%%%%%%%%%%%%%%equation 3.4
\begin{equation}
h_k(p({\bf x}),  q({\bf x})) \in  { \mathbb R}[p({\bf x}),  q({\bf x})] \cap  { \mathbb Z}[x, y] .
\end{equation}
Let
\[h_k(p,  q)  = \alpha p^j + qH(p, q), \ {\mbox where }\ \alpha \in  { \mathbb R} \ {\mbox and} \ H(p, q) \in  { \mathbb R}[p, q] ,\]
and the xy-degree of any term in  \(H(p, q)\) \ is \ \(k-3\).\\
If \ \(\alpha \ne 0\), \ then \ \(k = 2j\).  We consider the monomial \ \(y^{2j}\) \ in \ \(\alpha p({\bf x})^j +  q({\bf x})H(p({\bf x}),  q({\bf x})).\)  The monomial \ \(y^{2j}\)  occurs only in the term of \ \(\alpha p({\bf x})^j\).  Since \ \(h_k(p({\bf x}),  q({\bf x})) \in  { \mathbb Z}[x, y] ,\)  we have \(\alpha \in { \mathbb Z}\)  and so 

%%%%%%%%%%%%%%%%%%%%%%%%equation 3.5
\begin{equation}
q({\bf x})H(p({\bf x}),  q({\bf x})) \in   { \mathbb Z}[x, y] .
\end{equation}
We consider the case \ \(k\) \ is odd.  Since \ \(p(x, y), \ q(x, y) \in   { \mathbb Z}[x, y^2], \)  we may set 
\[H(p({\bf x}),  q({\bf x}))  =   \sum_{j=0}^{(k-3)/2}a_jx^{k-3-2j}y^{2j}, \ a_j \in  { \mathbb R}.\]
Then
\[ q({\bf x})H(p({\bf x}),  q({\bf x}))  =  a_0x^k  + (\sum_{j=1}^{(k-3)/2}(a_j-3a_{j-1})x^{k-2j}y^{2j}) - 3a_{(k-3)/2}xy^{k-1} .\]
By (3.5), we see that
\[a_0, a_1-3a_0, a_2-3a_1, \dots , a_{(k-3)/2}-3a_{(k-5)/2} \in   { \mathbb Z}.\]
Hence
\[a_0, a_1, \dots , a_{(k-3)/2} \in   { \mathbb Z}.\]
Therefore
%%%%%%%%%%%%%%%%%%%%%%%%equation 3.6
\begin{equation}
H(p({\bf x}),  q({\bf x})) \in  { \mathbb R}[p({\bf x}),  q({\bf x})] \cap  { \mathbb Z}[x, y] .
\end{equation}
For the case \(k\) is even, we can prove (3.6) similarly.

We get (3.6) from (3.4).  Repeating this procedure, eventually we have a polynomial \ \(H_*(p({\bf x}),  q({\bf x}))\) satisfying\\
(i) the xy-degree of any monomial in \ \(H_*(p,  q)\)  is 2 or 0,\\
(ii) \(H_*(p({\bf x}),  q({\bf x})) \in  { \mathbb R}[p({\bf x}),  q({\bf x})] \cap  {\mathbb Z}[x, y] .\)

Then \ \(H_*(p,  q) = bp\)  or  \(b, \ b \in   { \mathbb Z}.\)  The assertions (1) follows.\\
 (2)  We see in (2.10) that
\[\tilde{f}_d(p({\bf x}),  q({\bf x})) = (P_1({\bf x}),  P_2({\bf x})) ,
\mbox { with} \ {\bf x}  = (x,y),  \]
\[\mbox{and} \ P_j({\bf x}) =  P_j(p({\bf x}),  q({\bf x})) \in  { \mathbb R}[p({\bf x}),  q({\bf x})], \  j = 1, 2 .\]
Since \ \(T_d(z_1, z_2) \) \ is a polynomial map over  \( {\mathbb Z}\), \ \(f_d({\bf x})\) \ is a polynomial map defined over  \( { \mathbb Z}\).  Then \(P_j(p({\bf x}),  q({\bf x})), (j = 1, 2),\) satisfies (3.3).  Then  \(P_j(p, q) \in  { \mathbb Z}[p, q]\).
%(3) We have to check that  \ \(f_d : { \mathbb C}^2  \to  { \mathbb C}^2 \) \ is surjective, then (2.10) and Theorem \ref{theorem1}(1)  assert that \ \(g_d : V_F \to V_F\) \ is surjective.   To check  that  \ \(f_d : { \mathbb C}^2  \to  { \mathbb C}^2 \) \ is surjective, we will show that  \(f_d\)  extends to a morphism  \({\bar f}_d\)  of \( { \mathbb P}^2\).  The fiber \  \(({\bar f})_d^{-1}(a)\) \ contains \(d^2\) points, counted with multiplicity, for every \ \(a \in { \mathbb P}^2\)   (See \cite[Proposition 1.2]{DS2}).

%Then the dominant homogeneous part \(f_d^+\)   of  \(f_d\)  is written in the form : \\
%\(f_d^+(x, y) = (h_1(x, y),  \ h_2(x, y))\).  To see that  \(f_d\)  extends to a morphism of \( { \mathbb P}^2\),  it suffices to show that \ \((f_d^+)^{-1}(0, 0) = \{(0, 0)\}.\)  
%We suppose that\ \(h_1 = 0\) \ and  \ \(h_2 = 0\).  Then \ \(0 = h_1 + ih_2 = (x + iy)^d\).  
 %Then \ \(x = y = 0\).  For the case \(d\) is odd, we can prove the same result similarly.
\end{proof}
We note that since \ \(I_F = 0,\) \ the expression of \ \(P_j(p, q)\) \ is unique.

We study the dynamics of  \(g_d\)  on \ \( {\mathbb C}^2/D_3\).  We have seen that \ \( {\mathbb C}^2/D_3 \simeq {\mathbb C}^2\).  To start,  we introduce a new parametrization of \( {\mathbb C}^2\) \ which is convenient for understanding the dynamics of \(g_d\).  From (2.13), we have
%%%%%%equation 3.7
\begin{equation}
\begin {array}{ccc}
  \{t_1, t_2, t_3\} & \spmapright{ }& \{t_1^d, t_2^d, t_3^d\}\\
  \lmapupdown{} &  & \lmapupdown{} \\
  (z_1,z_2) & \sbmapright{T_d} & (z_1^{(d)}, z_2^{(d)}) \\
 \rmapdown{M_2^{-1}} &  & \rmapdown{M_2^{-1}} \\
  (x, y) & \sbmapright{f_d} & (X, Y) \\
\rmapdown{\tilde{F}} &  & \rmapdown{\tilde{F}} \\
  (p, q) & \sbmapright{g_d} & (P_d, Q_d) .
 \end{array}
\end{equation}
Set \quad \(t_1 =  e^{i(\alpha+\beta)}, \ t_2 = e^{i(\alpha-\beta)}, \ t_3 =  e^{-2i\alpha}, \) \ for  \(\alpha, \beta \in   {\mathbb C}.\)   Then
%We consider the points These points can be written as 
%%%%%%%%%%%%%%%%%%%%%%%%equation 3.8
\begin{equation}
z_1 =  e^{i(\alpha+\beta)} +  e^{i(\alpha-\beta)} +  e^{-2i\alpha} \ \mbox{and} \ z_2 = e^{-i(\alpha+\beta)} +  e^{-i(\alpha-\beta)} +  e^{2i\alpha}.
\end{equation}
%The point \(z_1\)  
%lies  inside or on the deltoid in \({\mathbb R}^2\)  (See \cite[Proposition 2.1]{U2}). 
%Set \ \(x = Re\ z_1\) \ and  \ \(y = Im\ z_1\).   
Hence by (2.8) we have
\[x = \cos(\alpha + \beta) +  \cos(\alpha - \beta) +  \cos2\alpha, \quad  y = \sin(\alpha + \beta) +  \sin(\alpha - \beta) -  \sin2\alpha.\]
Since \ \(p(x, y) = x^2 + y^2 \) \ and  \ \(q(x, y) = x^3 - 3xy^2 \) , it follows that 
%%%%%%%%%%%%%%%%%%%%%%%%equation 
\[p(x, y) = 3 + 2\cos(3\alpha - \beta) + 2 \cos 2\beta + 2 \cos(3\alpha + \beta) ,\]
\[q(x, y) = \cos6 \alpha + 2\cos\beta(5 \cos 3 \alpha + \cos(3 \alpha - 2 \beta) + 6 \cos\beta + \cos(3 \alpha + 2 \beta)).\]
Using trigonometrical identities, we can rewrite \ \(p(x, y)\) \ and \  \(q(x, y)\) \ as 
%%%%%equation 3.9
\begin{equation}
\begin{split}
p = 1 + 4\cos 3\alpha\cos \beta + 4(\cos \beta)^2\qquad \qquad\\
q = \frac 12(-2 + 4(\cos 3\alpha)^2 +  24(\cos \beta)^2 + 4\cos 3\alpha\cos \beta(3 + 4(\cos \beta)^2).
\end{split}
\end{equation}
Note that all the identities which we use to get (3.9) are valid for complex variables \(\alpha\) and \(\beta\).
Then setting \ \(u = 2\cos 3\alpha\) \ and \ \(v = 2\cos\beta\), we have
%%%%%equation 3.10
\begin{equation}
\begin{split}
p(u, v)  = 1 + uv + v^2,\\
q(u, v) = \frac 12(-2 + u^2 + 6v^2 + uv(3+v^2)).
\end{split}
\end{equation}
We can think of this as the mapping \ \(G_ {{\mathbb C}^2} :  {\mathbb C}^2 \to {\mathbb C}^2\) \ defined by  
\[G_{ {\mathbb C}^2}(u, v) = (1 + uv + v^2, \ \frac 12(-2 + u^2 + 6v^2 + uv(3+v^2)).\]
%%%%%%%%%%%%%%%%%%%%%%%%%%%%%%%%%%%%%%Proposition 3.2
\begin{pro} \label{pro3.1}
Under the above notations, we have 
\(G_{{\mathbb C}^2}({\mathbb C}^2) = {\mathbb C}^2\).   
\end{pro}
%%%%%%%%%%%%%%%%%%%%%%%%%%%proof
\begin{proof}
It suffices to prove that the polynomial parametrization \(G_{{\mathbb C}^2}({\mathbb C}^2) \)   covers all points of \({\mathbb C}^2\).   Since \(\tilde{F}\)  is surjective, for any 
\((p, q) \in {\mathbb C}^2\),  there exists an element  \((x, y) \in {\mathbb C}^2\) satisfying  \(\tilde{F}(x, y) = (p, q)\).  Then there exists an element  \(^t(z_1, z_2) = M_2\ ^t(x, y)\).  Since \ \(\alpha, \beta \in  {\mathbb C}\),  the parametrization \ \(t_1 = e^{i(\alpha+\beta)}, \ t_2 = e^{i(\alpha-\beta)}, \ t_3 = e^{-2\alpha i}\) \ covers all points of \ \(\{t_1, t_2, t_3\}\) \ with \(t_1t_2t_3 = 1\).  Hence we have an element \ \((\alpha, \beta) \in  {\mathbb C}^2\) such that the two elements \((x, y)\) and \((\alpha, \beta)\) satisfy  
\[x = \cos(\alpha + \beta) +  \cos(\alpha - \beta) +  \cos2\alpha, \quad  y = \sin(\alpha + \beta) +  \sin(\alpha - \beta) -  \sin2\alpha.\]
Then \ \(G_{{\mathbb C}^2}(2\cos 3\alpha, 2\cos\beta) = (p, q).\)
\end{proof}

Through we have a simpler proof in this case \(n = 2\),  we note this proof because  this proof is valid for \ \(n = 2, 3.\) See Section 4.
%We can show that  \ \(G_{{\mathbb C}^2}({\mathbb C}^2) =  {\mathbb C}^2\).  Then we have a polynomial parametrization  \(G_{ {\mathbb C}^2}\) of  an affine algebraic variety \( {\mathbb C}^2\)  that covers all points of \( {\mathbb C}^2\) .

 For any invariant  affine algebraic variety \(V\)  in \({\mathbb C}^m\) under \(g_d\),   we define the set of bounded orbits of \  \(g_d\mid_{V}\) \ by
\[K(g_d\mid_{V}) = \{w \in V : \{g_d^k(w) : k = 1, 2, \cdots \} \  \mbox{is bounded in }\ {\mathbb C}^m\}.\]
Let \(T_d(z)\) \ be the \(d^{th}\) Chebyshev polynomial (See \cite[\S6.2]{S1}) satisfying \ \(T_d(s + s^{-1}) = s^d + s^{-d}\), \ with \ \(z = s + s^{-1}\).
%%%%%%%%%%%%%%%%%%%%%%%%%%%theorem 3.3
\begin{theorem} \label{theorem2b}
Let \(g_d\) \ be a morphism on \ \( { \mathbb C}^2 / D_{3}\) in Definition 2.4.
\[(1) \ g_d (G_{ {\mathbb C}^2}(u, v)) = G_{ {\mathbb C}^2}(T_d(u), T_d(v)), \quad \mbox{for any} \quad (u, v ) \in  {\mathbb C}^2.\qquad \qquad \qquad \qquad \qquad\]
%%%%%equation 
\[(2) \ K(g_d\mid_ {\mathbb C^2}) = \{ G_{ {\mathbb C}^2}(2\cos 3\alpha, \ 2\cos \beta) : 0 \le 3\alpha, \beta < 2\pi\}, \ \mbox{for} \ d \ge 2. \qquad \qquad \qquad\]
\end{theorem}
%%%%%%%%%%%%%%%%%%%%%%%%%%%proof
\begin{proof}
(1)  We consider the parametrization \ \(p(u, v)\) \ and \ \(q(u, v)\) \ in (3.10).  Set
\[g_d(p(u, v), \ q(u, v)) = (P_d(u, v), \ Q_d(u, v)).\] 
Setting \ \(s = e^{i3\alpha}\),  we have \ \(s + s^{-1} = 2\cos 3\alpha.\)  Hence \ \(T_d(2\cos 3\alpha) = 2\cos 3d\alpha\).   Also we have \ \(T_d(2\cos \beta) = 2\cos d\beta\).  

Note that \ \(T_{d} (z_1,  z_{2})  = (z_1^{(d)},   z_{2}^{(d)})\) \ in (3.7).  The map \(T_d\) gives rise to a map : \ \((t_1, t_2, t_3) \mapsto  (t_1^d, t_2^d, t_3^d)\) \  such that \ \(\Phi (t_1^d, t_2^d, t_3^d) = (z_1^{(d)},   z_{2}^{(d)}).\) \ Since \(z_1\) and \(z_2\) are  written as (3.8), we have 
\[z_1^{(d)} =  e^{id(\alpha+\beta)} +  e^{id(\alpha-\beta)} +  e^{-2di\alpha},\]
\[z_2^{(d)} =  e^{-id(\alpha+\beta)} +  e^{-id(\alpha-\beta)} +  e^{2di\alpha}.\]
Set \ \(f_d(x, y) = (X, Y).\)  So, \ \(^t(X, Y) = M_2^{-1}\ ^t(z_1^{(d)}, z_2^{(d)}).\)
%That is, \(X = Re\ z_1^{(d)}\) \ and \(Y = Im\ z_1^{(d)}\). 
\ Then we get \ \((X, Y)\) by substituting  \ \(d\alpha\) for \(\alpha\) and \(d\beta\) for \(\beta\) in \((x, y)\). 

And also we get \ \(P_d\) and \(Q_d\) by substituting \(d\alpha\) for \(\alpha\) and \(d\beta\) for \(\beta\)  in \ \((p, q)\) in (3.9).
 Therefore we have
%%%%%equation 3.11
\begin{equation}
\begin{split}
P_d(u, v) = 1 + T_d(u)T_d(v) + T_d(v)^2, \\
Q_d(u, v) = \frac 12(-2 + T_d(u)^2 + 6T_d(v)^2 + T_d(u)T_d(v)(3+T_d(v)^2)),
\end{split}
\end{equation}
for every \ \((u, v) \in { \mathbb C}^2.\)  \\
%Clearly \ \([-2, 2]\) \ is an infinite subset of  \({ \mathbb C}\) and \ \(P_d(u, v)\) \ and \  \(Q_d(u, v)\) \ are polynomials.  Then by Corollary 4.6 in \cite[Chapter V, \S4]{Ls},  we see that (3.11) holds for any \ \((u, v) \in { \mathbb C}^2\).\\
(2)  Let \ \(u = t_1 + t_1^{-1}\) \ and \ \(v = t_2 + t_2^{-1}\), \ with \ \(t_j = r_je^{i\theta_j}, \ r_j \geq 1, \ \theta_j \in { \mathbb R},\) \ for \ \(j = 1, 2.\)  Then \ \(T_d(u) = t_1^d + t_1^{-d}\).  Let  \ \(T_d^n(u)\) \ denote the \(n\)-th iterate of \(T_d\).  Then  \ \(T_d^n(u) = t_1^{d^n} + t_1^{{-d}^n}\).  Set
\[g_d^n(p(u, v), \ q(u, v)) = (p_n(u, v), \ q_n(u, v)).\] 
We assume that \ \(\{p_n(u, v) : n = 1, 2, \dots\}\) \ is bounded.  We suppose that \ \(r_1 \ne r_2\).  We assume \ \(r_1 > r_2  \geq 1\).  Then the dominant term of \ \(p_n(u, v)\) \ is \ \(T_d^n(u)T_d^n(v)\) \ and \ \(\lim_{n\to\infty}T_d^n(u)T_d^n(v) = \infty.\)  A contradiction.  Also from the assumption \  \(r_2 > r_1  \geq 1\),  we have a contradiction.  Then \ \(r_1 = r_2  \geq 1\).  We suppose that \ \(r_1 = r_2 > 1\).  Then 
 \(\lim_{n\to\infty}q_n(u, v) = \infty.\)   Hence if \ \(g_d^n(p(u, v), \ q(u, v)) : n = 1, 2, \dots\}\) \ is bounded, then \ \(r_1 = r_2 = 1\).
\end{proof}
The  set \ \(K(g_d\mid_{{\mathbb C}^2})\) \ is a closed region lying on the real slice \ \({\mathbb R}^2\) \ in \ \({\mathbb C}^2\).  It is invariant under \(g_d\).  It is depicted in Figure 1.
%%%%%%%%%%%%%%%%%%%%%%%%%%%%figure 1
\begin{figure}
\includegraphics[scale=0.7]{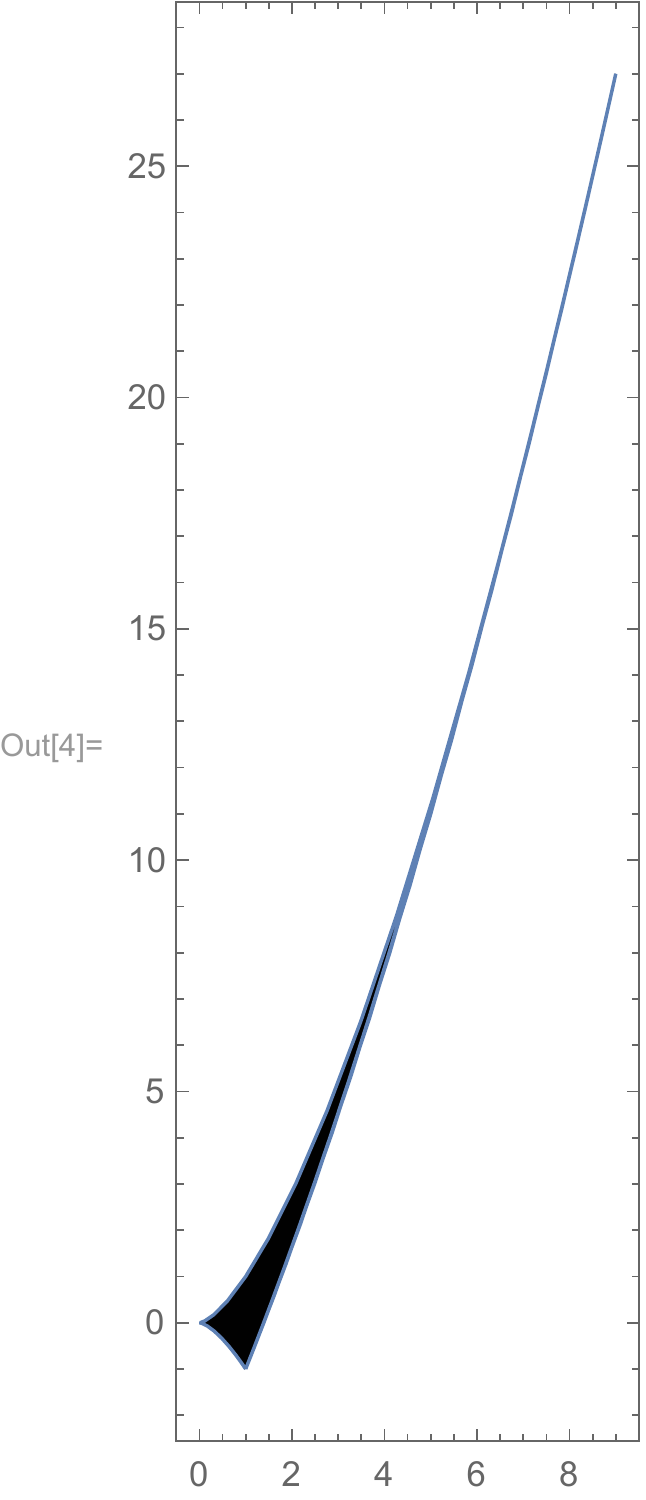}\hspace{2.5cm} \\
\caption{ An invariant closed domain.}
\label{Fig.1}
\end{figure}
It is known that Chebyshev endomorphism  \( P_{A_n}^d\)  on  \({\mathbb C}^n\) \ is postcritically finite.  More precisely,  the set of critical values of \( P_{A_n}^d\)  is invariant under \( P_{A_n}^d\)  (See \cite{U2} and \cite{U3}).
 
Inspired by the above result, we consider the branch points of the morphisms \(g_d\).  

We recall Theorem 2.8.  We use the notations in the theorem.
A point \ \(y \in Y\) \ is called a branch point of \(f\) if the number of inverse images of \(y\) under  \(f\) is less than the degree of  \(f\).  In this paper we use "a branch point" in place of   "a ramification point" in \cite[II, \S5]{Sh}.

We consider the inverse images of \((a, b)\) under the morphism  \(g_2\)  in (3.1).  Let \ \((P, Q)\) \  denote the right hand side of (3.1).  We consider the resultants
\[\mbox{ Res}(P-a, Q-b, q) \ \mbox{ and \ Res}(P-a, Q-b, p) \ \mbox{of} \  P-a \ \mbox{ and }\ Q-b\]
with respect to  \(q\) or  \(p\).\\
By direct computations, we have
\[\mbox{ Res}(P-a, Q-b, q) = 32 a + 2 a^2 - 16 b - 128 p + 8 a p + 96 p^2 - 4 a p^2 - 24 p^3 + 
 2 p^4\]
and 
\[\mbox{ Res}(P-a, Q-b, p) = 192 a^2 - a^3 - 256 b - 36 a b + b^2 - 2048 q + 864 a q - 24 a^2 q\]
\[ -  152 b q + 768 q^2 - 108 a q^2 - 4 b q^2 - 96 q^3 + 4 q^4.\]
Then \ \(g_2\)  is finite (See \cite[p.48]{Sh}).  By Proposition 2.9(2), we have  \ deg \ \(g_2 = 4.\)
To get the branch points of \ \(g_2\),  we compute the discriminant of\  Res\((P-a, Q-b, q)\) \ with respect to \(p\).  The discriminant is written in the form : \ \(-1048576(-27 + 18a + a^2 - 8b)(a^3 - b^2)\).
We have two affine varieties \ \(V(-27 + 18a + a^2 - 8b)\) \ and 
 \ \(V(a^3 - b^2)\).
We   denote the  cuspidal cubic \ \(V(p^3 - q^2)\) \ by  \(C_C\) \ and denote the parabola \ \(V(-27 + 18p + p^2 - 8q)\) \ by  \(C_D\). \ We will show that two curves \(C_C\)  and  \(C_D\) are branch locus of \(g_2\).  Let  \(O\) denote the origin \ \((0, 0)\)  and \(S\) denote the points of intersection \ \(C_C \cap C_D\) \ in \ \( {\mathbb C}^2\).  That is, \ \(S = \{(9, 27), (1, -1)\}\).
%%%%pro3.4
 \begin{pro} \label{pro3.2}
Under the above notations we have the followings : \\
(1) If \ \((a, b)  \in C_C\setminus\{O \cup S\}, \ \mbox{then } \   \sharp \ g_2^{-1}(a, b)= 3.\)  Two points lie on \(C_C\)  and one point with multiplicity 2 lies on \ \(V(3p - q -4)\).\\
(2)  If \ \((a, b)  \in C_D\setminus S, \ \mbox{then } \   \sharp \ g_2^{-1}(a, b)= 3.\)    Two points lie on \(C_D\)  and one point with multiplicity 2 lie on \ \(V(p - 1)\).\\
(3)  \( \sharp \ g_2^{-1}(p, q)  = 2,\)  \ for any point \ \((p, q)\) in \ \(O \cup S\).
 \end{pro}
%%%%%%%%%%%%%%%%%%%%%%%%%%%proof
\begin{proof}
 By direct computations, we can prove this proposition.
\end{proof}

Next we consider the inverse images of \ \((a, b)\) \ under the morphism  \(g_3\) in (3.2).  Let  \((P, Q)\) \ denote the right hand side of (3.2).   The resultant \ \(\mbox{ Res}(P-a, Q-b, q)\) \ is a polynomial in  \(p\) of degree 9 over \ \( { \mathbb Z}[x, y] .\)  Its leading term is \ \(-4p^9\).  The resultant \ \(\mbox{Res}(P-a, Q-b, p)\) \ is also a polynomial in \(q\) whose leading term is \ \(64q^9\).
Then   \(g_3\)  is a finite map.  By Proposition 2.9(2), we have   \ deg \ \(g_3 = 9.\)   The discriminant of 
\ \(\mbox{Res}(P-a, Q-b, q)\) \ is written in the form \ \(n(a-1)^6(-27 + 18a + a^2 - 8b)^3(a^3 - b^2)^3\),  where \(n\) is a negative integer. 

We consider the case \ \(a = 1\).  If \ \(a = 1\),  the resultant \ \(\mbox{ Res}(P-a, Q-b, q)\) \ has a zero  \ \(p = 1\) \ with multiplicity 3.  But if  \ \(a = 1\) \ and \(b \ne \pm1\),  then \ \( \sharp \ g_3^{-1}(1, b)  = 9.\)   We compute \ \( \sharp \ g_3^{-1}(1, 1)  = 6\) \  and \ \( \sharp \ g_3^{-1}(1, -1)  = 4.\)   Clearly, the points \ \((1, \pm 1)\) \ belong to \ \(V(a^3 - b^2)\).
%%%%pro3.5
 \begin{pro} \label{pro3.4}
 We have the followings :\\
(1)  If \  \( (a, b) \in C_C\setminus\{O \cup S\} ,\)   then  \ \( \sharp \ g_3^{-1} (a, b) = 6\).  Three  points lie on \(C_C\)  and other three points with multiplicity 2 lie on \ \(V(-27 + 54 p - 27 p^2 + p^3 - 18 q + 18 p q - 4 q^2)\).\\
(2)  If \  \( (a, b) \in C_D\setminus  S\) ,   then  \ \( \sharp \ g_3^{-1}(a, b) = 6\).  Three  points lie on \(C_D\)  and other three points with multiplicity 2 lie on \ \(V(p^2 - 2q)\).\\
(3)  \( \sharp \ g_3^{-1}(0, 0)  = 3,\) \ \( \sharp \ g_3^{-1}(9, 27)  = 3,\) \ and \ \( \sharp \ g_3^{-1}(1, -1)  = 4.\) 
 \end{pro}
%%%%%%%%%%%%%%%%%%%%%%%%%%%proof
\begin{proof}
 By direct computations, we can prove this proposition.
\end{proof}
%%%%pro3.6
 \begin{pro} \label{pro3.5}
(1)  The branch locus of  \ \(g_d : {\mathbb C}^2 \to {\mathbb C}^2\) \  consists of \ \(C_C\) \ and \ \(C_D\), for \ \(d = 2, 3.\)\\
(2)  The branch locus of \ \(\tilde{F} : {\mathbb C}^2 \to V_F\)  is \(C_{C}\).
 \end{pro}
%%%%%%%%%%%%%%%%%%%%%%%%%%%proof
\begin{proof}
 (1)  This assertion follows from the above arguments.\\
(2)  Recall that  \ \(p = x^2 + y^2\) \ and \ \(q = x^3 - 3xy^2.\)  Then we have a equation \ \(4x^3 - 3px - q = 0.\)  Its discriminant with respect to \(x\) is \(432(p^3 - q^2).\)  Then \(C_C\) is the branch divisor of \(\tilde{F}\).
\end{proof}
In the next subsections we study the dynamics of \ \(g_d\)  on these branch divisors.

%%%%%%%%Subsection3.2
\subsection{\rm{The dynamics of \ \(g_d\) \ on the cuspidal cubic \(C_C\).}}\hspace{0cm}

In the previous subsection, we see the relation between \ \(g_d^{-1}(C_C)\) \ and \(C_C\)  for \ \(d = 2, 3\).  We will prove that \ \(g_d(C_C) = C_C\) \ for any \ \(d \in {\mathbb N}\) \ and show the dynamics of 
\ \(g_d\) \ on  \(C_C\).  A polynomial parametric representation of \ \(C_C = V(p^3 - q^2)\) consists of \ \(p = t^2\) \ and \ \(q = t^3\).
We consider a mapping \ \(G_C :  {\mathbb C} \to {\mathbb C}^2\) \ defined by  
\[G_C(z) = ((1+z)^2, (1+z)^3).\]
%%%%%%%%%%%%%%%%%%%%%%%%theorem 3.7
\begin{theorem} \label{theorem3} 
Let \ \(g_d\) be a morphism on \ \( {\mathbb C}^2/D_3\) in Definition 2.4.\\
(1) \(\ g_d(G_C(z)) = G_C(T_d(z)),\) \ for any \ \(z \in  { \mathbb C} \quad and \quad d \in  { \mathbb N}.\)  \\
 (2) \(g_d(C_C) = C_C, \) \ for any \  \( d \in  { \mathbb N}\).
\end{theorem}
%%%%%%%%%%%%%%%%%%%%%%%%%%%theorem 3.4 proof begin
\begin{proof}
(1)The proof is similar to the proof of Theorem 3.3.  

Using (3.9) we compute
\[p(x, y)^3 - q(x, y)^2 = 4(\cos 3\alpha - \cos 3\beta)^2(\sin 3\alpha)^2.\]
Since \ \(p^3 = q^2\), \ we assume that \ \(\alpha = 0.\) 
Clearly  \ \(d\alpha = 0.\)  Hence we put \ \(u = 2\) \ in (3.10).  Then from \ \(G_{ {\mathbb C}^2}\) \ we get \(G_C\).
Then  we have the assertion (1).\\
(2)  We consider a polynomial parametrization of \(C_C\) given by \  \(p = t^2\) \ and\  \(q = t^3\).
Equivalently, we consider the mapping \ \(G : {\mathbb C}  \to {\mathbb C}^2 \) \ defined by  
\ \(G(t) = (t^2, t^3)\).  We know that the polynomial parametrization  \ \(G({\mathbb C})\) \ fills up all of \ 
\(V(p^3 - q^2).\) \\
%%%%%%%%%%%%%%%%%%%%%%%%%%%theorem 3.4 proof end
Then the assertion (2) follows from (1).
\end{proof} 
We remark that if we assume that \ \(\cos3\alpha = \cos3\beta\), \ then \ \(p(x, y) = (1 + 2\cos 2\beta)^2, \  q(x, y) = (1 +2\cos 2\beta)^3 \).  Then the similar result holds.

We may view \ \(V(p^3-q^2)\) \ as an affine curve in \({\mathbb C}^2 \).    Then 
by the argument similar to that in the proof of Theorem 3.3 (2), we have
\[K(g_d\mid_{C_C}) =  \{((1+2\cos\theta)^2, \ (1+2\cos\theta)^3) : 0 \le \theta <2\pi\}.\]
The set \ \(K(g_d\mid_{C_C})\) \ is a part of the cuspidal cubic \ \(V(p^3-q^2)\) \ in   \({\mathbb R}^2 \)  containing the cusp point \((0, 0)\).  It is a Jordan arc,  denoted by \(\gamma_1\).  Its end points are \((1, -1)\)  and  \((9, 27)\).  Recall that the set of these two points coincides with the intersection \ \(C_C \cap C_P \) \ in  \ \({\mathbb C}^2 \). 

The dynamics of \ \(g_d\mid_{C_C}\) can be deduced by the dynamics of the Chebyshev map \(T_d : {\mathbb C} \to {\mathbb C}\).  It is known (e. g. in \cite[\S7]{M})  that \ \(K(T_d) = [-2, 2], \ T_d([-2, 2]) = [-2, 2]\), 
and Julia set of \ \(T_d\) is also \ \( [-2, 2]\).  The Jordan arc \(\gamma_1\) corresponds to the interval \([-2, 2]\).  Then all preperiodic points of \ \(g_d\mid_{C_C}\)\ 
lie on \(\gamma_1\).  It is know that \ \(T_2\)  restricted to \ \( [-2, 2]\) \ is conjugate to a typical chaotic map \ \(L(x) = 4x(1-x)\) \ on \ \([0, 1]\) which is studied by Ulam and von Newman \cite{UN}. 
In this meaning, we say that \ \(g_d\mid_{C_C}\)\ is chaotic on \(\gamma_1\).  
%%%%%%%%%%%%%%%%%%%%%%%%%%%%%%%%%%%%%%Corollary 3.8
\begin{cor} \label{cor2}
Under the above notations we have the following :\\
(1) \(K(g_d\mid_{C_C}) = \gamma_1, \ g_d(\gamma_1) =  \gamma_1\) \ and \ \(g_d\mid_{C_C}\)\ is chaotic on \(\gamma_1\).  \\
(2) If \ \(3\not|\ d,\) \  then the cusp point \(O\) in  \(\gamma_1\) is a fixed point of \(g_d\). \\
  If \ \(3\mid d\), \ then \ \(g_d(O) = (9, 27)\).  And \ \(g_d(9, 27) = (9, 27)\) \ for any \  \(d \geq 1\). 
\end{cor}
 %%%%%%%%%%%%%%%%%%%%%%%%%%proof
\begin{proof}
(1) By the above argument, we have this assertion.\\
 (2) Note that the  point \ \(O = (0, 0).\)   We set \ \(z = s + s^{-1}\).  If \ \(1 + s + s^{-1} = 0,\)  \ then \(s\) is a primitive cubic root of unity.  Then the assertion (2) is obvious.  
\end{proof}

Next we will show that \ \(g_d\mid_{C_C}\)\ extends to a morphism on  \({\mathbb P}^2 \).  Let \ \(g_d(p, q) = (P_d(p, q), Q_d(p, q))\).
Then \ \(g_d\mid_{C_C}(p, q) = ([P_d(p, q)],\ [Q_d(p, q)])\) where \ \([ P_d(p, q)]\) \ and \  \([ Q_d(p, q)]\)\ are equivalent classes in the coordinate ring \ \({\mathbb C}[p, q]/(p^3-q^2) \). 
%%%%%%%%%%%lemma3.9
\begin{lemma} \label{lemma1}
We can represent the equivalent classes in the form :\\
(1) \ \(P_d(p, q) = p^d + H_1(p, q) \),\\
(2) \ \(Q_d(p, q) \equiv q^d + pH_2(p, q) + H_3(p, q) \quad mod \ (p^3 - p^2)\),\\
where \ \(H_j(p, q) \in  { \mathbb Z}[p, q], \ j = 1, 2, 3,\)\ and the degrees of \ \( H_1(p, q) \) \ and \ \( H_3(p, q) \) \ are less than or equal to \ \(d-1\) \ and \ \( H_2(p, q) \)\ is a homogeneous polynomial of degree \ \(d-1\).
\end{lemma}
%%%%%%%%%%%%%%%%%%%%%%%%%%%proof
\begin{proof}
 (1)  We consider the diagram (3.7) in the case \({z_2} =\overline{z_1}\).  Then \ \(T_d(z_1, \overline{z_1}) = (z_1^{(d)}, \overline{z_1^{(d)}})\).      Let \ \(z_1 = x + iy\),  where \(x\) and \(y\) are real variables.  Then \ \(X(x, y) = Re\ z_1^{(d)}\) \ and \ \(Y(x, y) = Im\ z_1^{(d)}\).  Then \ \(f_d(x,y) = (X(x, y), \ Y(x, y)).\)  We recall that \ \(T_d(z_1, z_2) = (z_1^{(d)}, z_2^{(d)})\).  By \cite{L} and \cite[p.996]{U2}, we have a recurrence equation
\[z_1^{(d)} = z_1z_1^{(d-1)} -z_2z_1^{(d-2)} + z_1^{(d-3)}, \ z_1^{(1)} = z_1,\  z_1^{(2)} = z_1^{2} - 2z_2,\ z_1^{(3)} = z_1^3 -3z_1z_2 + 3.\]
Then \ \(z_1^{(d)}(z_1, z_2)\) \ is a polynomial of degree \(d\) and its leading term  is \ \(z_1^d\).  Then the leading term of \(z_1^{(d)}(z_1, \overline{z_1})\) \ is \ \(z_1^d\).  
 Set \ \((x + iy)^d = X_1 + iY_1\).   Set \ \(X = X_1 + X_2\) \ and \ \(Y = Y_1 + Y_2\), where   degrees of \(X_2\) and \(Y_2\) are less than or equal to \ \(d-1\).   

  If \(d\) is even, the term \ \(y^d\) occurs in \(X_1\).   If \(d\) is odd, the term \ \(y^d\) occurs in \(Y_1\).  Then the term \ \(y^{2d}\) occurs in \(X^2 + Y^2\).  Since \ \(P_d(p({\bf x}),  q({\bf x})) = X^2 + Y^2\), it follows that the term \ \((x^2 + y^{2})^d\) \ must  occur in \ \(P_d(p({\bf x}),  q({\bf x}))\).
Hence the term \(p^d\)  occurs  in \ \(P_d(p, q)\).  The xy-degree of any term in \ \(P_d(p, q)\) is less than or equal to \(2d\).  If there is a term \ \(p^iq^j\) \ in \ \(P_d(p, q)\) \ with \ \(2i + 3j = 2d\) \ and \ \(j  > 0\), \ then \ \(i + j < d\). \\
(2)  We consider  a monomial \ \(p^{3n+j}q^k\), \ where \ \(j, k, n \in{ \mathbb N} \cup \{0\}\) \ and \ \(0 \le j \le 2\).
We define a degree minimization function \(m\) by \ \(m(p^{3n+j}q^k) = p^jq^{2n+k}\).  Clearly, \(m(p^{3n+j}q^k) \equiv p^{3n+j}q^{k} \ mod \ (p^3-q^2)\).  Let \ \(Q_d(p, q) = \sum a_{ij}p^iq^j\).  We replace any term \ \(a_{ij}p^iq^j\) \ with \ \(a_{ij}m(p^iq^j)\).  Since the leading term of \ \(z_1^{(d)}\) \ is \ \(z_1^{d}\),  the xy-degree of any  term  in \ \(Q_d(p, q)\) is less than or equal to \(3d\). \\
(i)  The monomials with the xy-degree  \(3d\) :
\(q^d, q^{d-2}p^3, \dots \dots, q^{d-2k}p^{3k},\)  where \ \(k = d/2\) \ or \ \((d-1)/2\) \ according to whether \(d\) is even or odd. \\
 The images of these monomials under the function \(m\) are \(q^d\).\\
(ii)  The monomials with the xy-degree  \(3d-1\) :
\(q^{d-1}p,  q^{d-3}p^4, \dots \dots, p^{3(d-1)/2} \  \mbox{or} \ p^{3d/2-2}q. \)  
The images of these monomials under the function \(m\) are \(q^{d-1}p\).\\
(iii)  The monomials with the xy-degree  \(3d-2\) :
\(q^{d-2}p^2,  q^{d-4}p^5,  \dots \dots, p^{3d/2-1} \  \mbox{or} \ p^{3d/2-5/2}q. \)  
The images of these monomials under the function \(m\) are \(q^{d-2}p^2\).

Then to prove that the  polynomial \ \(\sum a_{ij}m(p^iq^j)\) \ satisfies the conditions of the right hand side of (2),  it suffices to prove the following.  Let \ \(\sum_{h=0}^ka_hq^{d-2h}p^{3h}\)\ be the sum of terms in \ \(Q_d(p, q)\) \ with the xy-degree  \(3d\).  Then it is enough to prove
\ \(\sum_{h=0}^ka_h = 1\).  Note that 
 \[X_1^3 - 3X_1Y_1^2 = \sum_{h=0}^ka_hq({\bf x})^{d-2h}p({\bf x})^{3h}.\] 
%\[Q_d(p({\bf x}),  q({\bf x})) = X(x, y)^3 - 3X(x, y) Y(x, y)^2  = X_1(x, y)^3 - 3X_1(x, y) Y_1(x, y)^2 + l(x, y),\]
%where \ \(l(x, y)\) \ is a polynomial in \(x, y\) whose degree is less than or equal to \(3d-1\).  
Since \ \(X_1^3 - 3X_1Y_1^2 = Re(X_1+iY_1)^3 = Re(x+iy)^{3d},\)  the term \(x^{3d}\) occurs in \ \(X_1^3 - 3X_1Y_1^2\) \ and so it  occurs in \ \(\sum_{h=0}^ka_h q({\bf x})^{d-2h}p({\bf x})^{3h} \).  Hence \ \(\sum_{h=0}^ka_h = 1\).
\end{proof}
We will show that the morphism  \(g_d\) restricted to \ \(V(p^3-q^2) \) \ can extend to a morphism \ \(\bar{g}_d\) on \  \({\mathbb P}^2({\mathbb C}) \).   Let \ \((p : q : r)\) \ be homogeneous coordinates of \  \({\mathbb P}^2({\mathbb C}) \).  The projective closure of  \ \(V(p^3-q^2) \) \ is the variety  \ \(V(p^3-q^2r) \).  Based on Lemma \ref{lemma1}, we define morphisms on \  \({\mathbb P}^2({\mathbb C}) \).   We set
\[\bar{P}_d(p, q, r) = p^d + r^dH_1(p/r, q/r) = p^d + r\bar{H}_1(p, q, r) ,\]
and
\[\bar{Q}_d(p, q, r) = q^d  +  pH_2(p, q) + r^dH_3(p/r, q/r) = q^d + pH_2(p, q) + r\bar{H}_3(p, q, r) .\]
Here \ \(\bar{H}_1(p, q, r)\) \ and \ \(\bar{H}_3(p, q, r)\) \ are homogeneous polynomials.  Then we define mappings \ \(\bar{g}_d\) on \  \({\mathbb P}^2({\mathbb C}) \)  by 
%%%%%%%%%%%%%%%%%%%%%%%%equation 3.12
\begin{equation}
\bar{g}_d(p : q : r) = (\bar{P}_d : \bar{Q}_d : r^d).
\end{equation}
%%%%pro3.10
 \begin{pro} \label{pro3.9}
Let \ \(\bar{g}_d\)  be a mapping in (3.12).\\
(1)  The mapping \  \(\bar{g}_d : {\mathbb P}^2({\mathbb C}) -\to  {\mathbb P}^2({\mathbb C}) \) \ is a morphism.\\
(2)  \(\bar{g}_d(V(p^3-q^2r)) = V(p^3-q^2r)\).
 \end{pro}
%%%%%%%%%%%%%%%%%%%%%%%%%%%proof
\begin{proof}
(1) We suppose that \ \(\bar{P}_d = 0, \  \bar{Q}_d = 0\) \ and\  \(r^d = 0\).  Then we have \ \(p = q = r = 0.\)\\
(2) If \ \(r = 1\), then by Theorem \ref{theorem3}(2), we have \  \(\bar{g}_d(V(p^3-q^2)) = V(p^3-q^2)\).  Clearly \ 
 \(V(p^2-q^2r) \cap \{(p : q : 0)\} = (0 : 1 :0), \) \ and \ 
 \((\bar{g}_d)^{-1}(0 : 1 : 0) = \{(0 : 1 : 0)\}\).
\end{proof}
By \cite[II2]{S2},  we know that \  \(\bar{g}_d : V(p^3-q^2r) \to V(p^3-q^2r)\) \ is  finite.  Diller \cite{D}  and Uehara \cite{Ue} studied surface automorphisms preserving the cuspidal cubic. 
%%%%%%%%%%%%%%%%%%%%%%%%Subsection3.3
\subsection{\rm{The dynamics of \ \(g_d\) \ on the affine curve \(C_D\) and \(K\) sets.}}\hspace{0cm}

We first show the relation between the curve \(C_D\) and a relative orbit variety of a deltoid which  has a connection with the set of critical values of \ \(f_d\) on \(R_2\).  The critical set of \ \(P_{A_2}^d \ (d \ge 2)\) \ on \({\mathbb C}^2\) \ is described in Lemma 2.1 in \cite{U2} : 
\[z_1 = t_0 + \varepsilon t_0 + \frac 1{ \varepsilon t_0^2},\quad z_2 = \frac 1t_0 + \frac 1{\varepsilon t_0} + { \varepsilon t_0^2}, \ where \ \varepsilon = e^{2\pi\sqrt{-1}/d},  \quad t_0 \in {\mathbb C}^*, \ j \in \mathbb{N}.\] 
We see that \ \(z_1 = \bar{z}_2\) \ if and only if \ \(\mid t_0\mid = 1.\)  The critical values are written in the form \ \(z_1^{(d)} = t_0^d +  t_0^d +  t_0^{-2d}.\) \ Note that the set of critical values of \(f_d\) is the same set for any \ \(d \ge 2\).   We set \ \(t = t_0^d\).  Then \ \(z_1^{(d)} = t + t + t^{-2}.\)  We set \ \(t = e^{i\alpha}\).  
Then the set of critical values of \(f_d\) on \(R_2\)  can be parametrized as
%%%%%%%%%%%%%%%%%%%%%%%%equation 3.13
\begin{equation}
x = 2\cos\alpha + \cos2\alpha, \quad y = 2\sin\alpha - \sin2\alpha. 
\end{equation}
The real variety parametrized by the above equations for \ \(0 \le \alpha < 2\pi\) is a deltoid. Its implicit representation is \  \(h(x, y) : = (x^2 + y^2 + 9)^2 + 8(-x^3 + 3xy^2)-108 = 0\). (See \cite[p. 998]{U2}).  

We regard \(\alpha\)  as a complex variable in (3.13). Hence we view  \ \(h(x, y) \) as a polynomial in \({\mathbb C}[x, y]\).   Since \  
\(h(x, y) \) \ is invariant under the action of \(D_3\), we can construct the relative orbit variety \ \(V(h(x, y))/D_3 \)\ whose points are \(D_3\)-orbits of zeros of  \ \(h(x, y) \).  See \cite[\S2.6]{St}.   By Algorithm 2.6.2 in \cite{St},  we see  that \ \(V(h(x, y))/D_3 = V(-27 + 18p + p^2 -8q)\).  Hence the \ \(C_D = V(-27 + 18p + p^2 -8q)\) \ is the relative orbit variety \ \(V(h(x, y))/D_3 \).  

Note that by setting \ \(\beta = 0\) \ in (3.8), we have (3.13).  Hence,  we put \ \(v = 2\) \ in (3.10).
Then  we have
\[p( u) = 5 + 2u, \quad q( u) = 11 + 7u + \frac{ u^2}2.\]
Hence from  \(G_{{ \mathbb C}^2}\) we get a mapping from \ \({\mathbb C}\) to \({\mathbb C}^2\) defined by \ \(G_D( u) = (5 + 2u, 11+ 7u +  u^2/2).\)
We prove in the next theorem that the mapping \ \(G_D\) \  parametrizes the curve \ \(C_D\).
The morphism  \(g_d\) restricted to \(C_D\) is represented in the following form.
%%%%%%%%%%%%%%%%%%%%%%%%theorem 3.11
\begin{theorem}\label{theorem4}  
Let \ \(g_d\) be a morphism on \ \( {\mathbb C}^2/D_3\) in Definition 2.4.\\
(1) \ \(g_d(G_D( u)) = G_D(T_d( u)), \)  \quad for any \ \( u \in  { \mathbb C} \ and \ d \in  { \mathbb N}.\)\\
 (2) \(g_d(C_D) = C_D, \) \ for any \  \( d \in  { \mathbb N}\).
\end{theorem}
%%%%%%%%%%%%%%%%%%%%%%%%%%%proof
\begin{proof}
(1) By the argument similar to that used in the proof of Theorem 3.3(1), we can prove this assertion.  \\
(2)  We can show that by Theorem 1 in  \cite[Chapter 3, \S3]{CLS},  \(V(p^2 + 18p - 8q - 27)\)\ is the smallest algebraic set containing   
\ \(G_D ( { \mathbb C})\).  We know from \cite[Chapter 3 \S3, p.131]{CLS}  that \(G_D ({ \mathbb C})\) covers all points of  \(V(p^2 + 18p - 8q - 27)\). Then the assertion (2) follows.
\end{proof}

By the similar way to the case of \  \(g_d\mid_{ C_C}\),  we have 
\[K(g_d\mid_{C_D}) = \{(5 + 4\cos\alpha, 11 + 14\cos\alpha + 2\cos^2\alpha)  : 0 \le \alpha \le 2\pi\}, \ \mbox{for} \ d \ge 2.\] 
Set \ \(\gamma_2 = K(g_d\mid_{C_D})\).
The Jordan arc \(\gamma_2\)  corresponds to the interval \([-2, 2]\) for \(T_d\).  Then \ \(g_d(\gamma_2) = \gamma_2\) \ and \  \(g_d\mid_{C_D} \) \ is chaotic on \(\gamma_2\). The end points of \(\gamma_2\)  are \((1, -1)\) and (9, 27).  Then we can connect two Jordan arcs \(\gamma_1\)  and  \(\gamma_2\)  and get a Jordan curve
%%%%%%%%%%%%%%%%%%%%%%%%equation 3.14
\begin{equation}
\gamma = \gamma_1 \cup \gamma_2.
\end{equation}
The Jordan curve \(\gamma\) lies on the real slice \ \({\mathbb R}^2 \)  in \ \({\mathbb C}^2 \). See Figure 2.  We will show that it is the boundary of \ \(K(g_d\mid_{{\mathbb C}^2})\) \ in \ \({\mathbb R}^2 \).  See Figure 1.  Clearly \(g_d\) maps \ \({\mathbb R}^2 \)  to \   \({\mathbb R}^2 \) \ and \(g_d(\gamma) = \gamma\) and  \(g_d\mid_{ {\mathbb R}^2} \) \ is chaotic on \(\gamma\)  for \ \(d \ge 2\).
%%%%%%%%%%%%%%%%%%%%%%%%%%%%figure 
\begin{figure}[htbp]
\begin{tabular}{lcr}
\begin{minipage}{0.39\hsize}
\begin{center}
\includegraphics[scale=0.45]{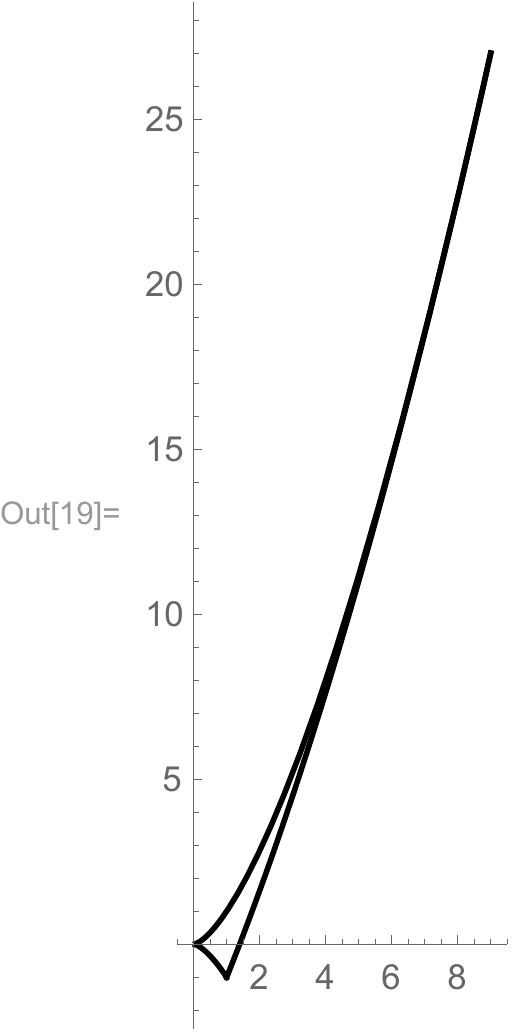} \\
\caption{The  Jordan curve \(\gamma\).}
\label{Figure2}
\end{center}
\end{minipage}
\begin{minipage}{0.4\hsize}
\begin{center}
\includegraphics[scale=0.45]{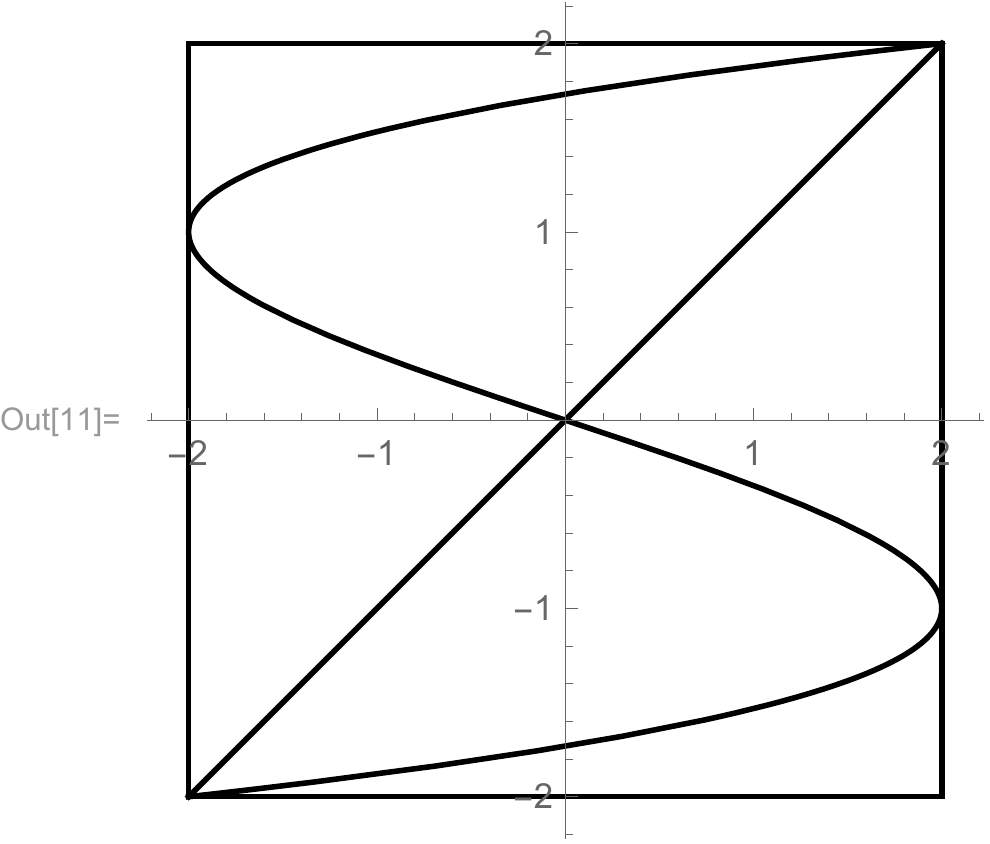}\\
\caption{The zeros of det[J].}
\label{Figure3}
\end{center}
\end{minipage} 
\begin{minipage}{0.4\hsize}
\begin{center}
\includegraphics[scale=0.45]{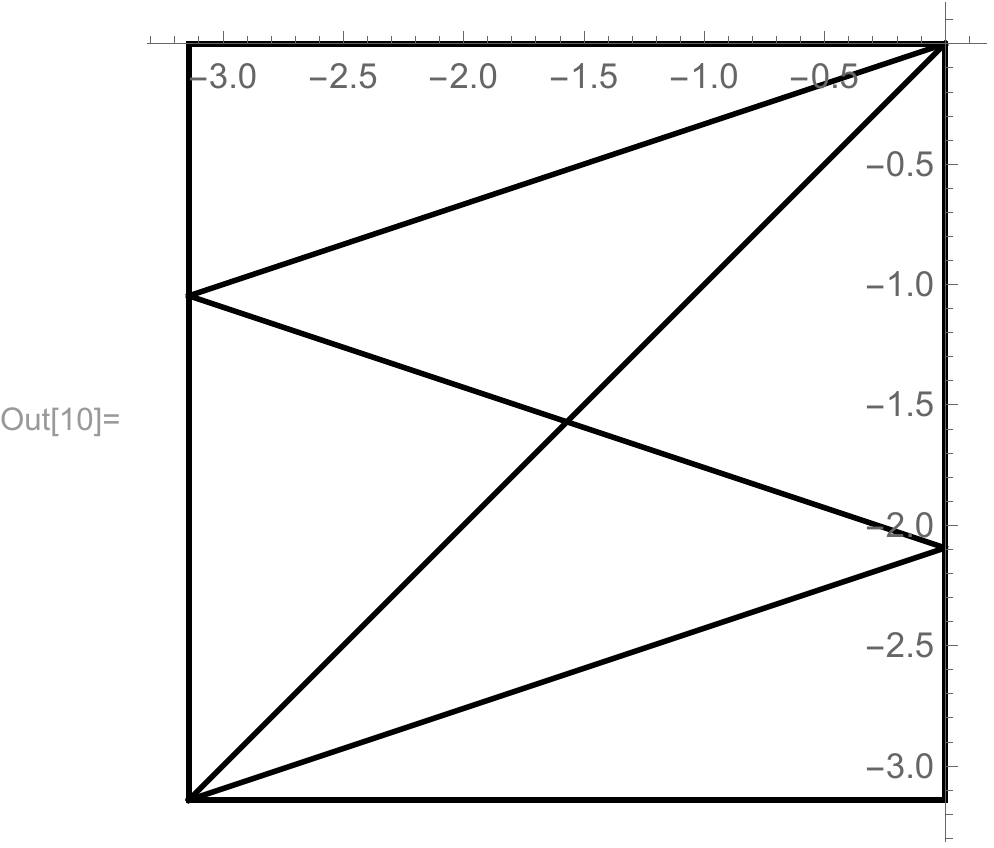} \\
\caption{A partition in the (\(\alpha^\prime, \beta\)) space.}
\label{Figure4}
\end{center}
\end{minipage}
\end{tabular}
\end{figure} 

%%%%%%%%%%%%%%%%%%%%%%%%%%%%%%%%%%%%%%Theorem 3.12
\begin{theorem} \label{theorem3.12} 
Let \(\gamma\) be the the Jordan curve defined in (3.14).  
Then \  \(\gamma\)  is the boundary of \ \(K(g_d\mid_{{\mathbb C}^2})\) \ in \({\mathbb R}^2\).   
\end{theorem}
%%%%%%%%%%%%%%%%%%%%%%%%%%%proof
\begin{proof}
Set \ \(I^2 = [-2, 2]\times[-2, 2]\).  Let \ \(h(u, v)\) \ denote the map \ 
\(G_{ {\mathbb C}^2}\mid_ {I^2} , \ i. e. ,\)  
\[h(u, v) = (1 + uv + v^2, \ \frac 12(-2 + u^2 + 6v^2 + uv(3+v^2)), \ -2 \le u, v\le 2.\]
By Theorem 3.3(2), \ we have \ \(h(I^2) = K(g_d\mid_{{\mathbb C}^2})\).   We consider \(p(\alpha, \beta)\) and \(q(\alpha, \beta)\) in (3.9) for \(0 \le \alpha, \beta < 2\pi\).  Then 
%%%%%%%%%%%%%%%%%%%%%%%%equation 3.15
\begin{equation}
\begin{split}
p^3 - q^2 = 4(\sin 3\alpha)^2(\cos 3\alpha - \cos 3\beta)^2 \ge 0,  \\
\mbox{and} \qquad \qquad \qquad \qquad \qquad \qquad \qquad \qquad \qquad \qquad \qquad \qquad\\
27+8q-18p-p^2 = 16(\sin\beta)^2(\cos 3\alpha - \cos\beta)^2 \ge 0.
\end{split}
\end{equation}
Let \(D_1\) denote the closed domain in \({\mathbb R}^2\) bounded by \(\gamma\).  Then \ \(h(I^2) \subset D_1\).  We will show that \(h\) maps \(I^2\) onto \(D_1\).

We compute the Jacobian matrix \ \(J = \partial(p,q)/\partial(u,v)\) \ of \ \(h(u, v)\) \ and its determinant.  Then
\[\det[J] = -(u-v)(u+3v-v^3).\]
The set of zeros of \(\det[J]\) in the \ \((u, v)\) space is depicted in Figure 3. 
The square \(I^2\) is partitioned into six domains.  Set 
\[D_2 = \{(u, v) : u \ge v^3-3v, \ v \ge u, \ -2 \le u, v \le 2\}.\]
We will show \(D_2\) is mapped onto \(D_1\)  under the map \(h\).  Let \(\Gamma\) be the boundary of \(D_2\).  Then \ \(h(\Gamma) = \gamma.\)  As a point on \(\Gamma\)  moves anticlockwise round \(\Gamma\),  its image under \(h\) moves round \(\gamma\)  once  anticlockwise.  Note that \ \(J > 0\)\ in the interior of \(D_2\).  We use a topological argument principle (\cite[Chapter 7, IV, 4]{N}).  It states that \ \(\nu(h(\Gamma), w) = \Sigma\nu(a_j),\)  \   where \ \(\nu(h(\Gamma), w)\) \ is the winding number of a closed loop \ \(h(\Gamma) = \gamma\)  \ about a point \(w\) in the interior of \ \(D_1\) \ and the sum is over \(w\)-points \(a_j\) lying inside  \(\Gamma\)   and  \ \(\nu(a_j) =\) the sign of det\([J[a_j]]\).  Then we have that the mapping  \(h\) maps the interior of \(\Gamma\)  onto the interior of 
\(\gamma\)   in one-to-one fashion.  Then \(h(D_2) = D_1.\)  Any of other five domains in \(I^2\) is also mapped onto \(D_1\) under \(h\).
\end{proof}

Set \ \(\alpha^\prime = 3\alpha\).  Set \ \(\tilde{I}^2 : = \{(\alpha^\prime, \beta) : -\pi \le \alpha^\prime, \beta \le 0\}.\)  The partion of the square \(\tilde{I}^2\)  depicted in Figure 4 corresponds to the partition of \(I^2\) depicted in Figure 3.

Next we will show that \ \(g_d\mid_{C_D}\) \ extends to a morphism on \ \({\mathbb P}^2 \) \  as in the case of \(C_C\).    We use the same notations used in Lemma \ref{lemma1}.   
%%%%%%%%%%%lemma3.13
\begin{lemma} \label{lemma2}
Let \ \(g_d(p, q) = (P_d(p, q), Q_d(p, q))\).\\
(1) \ \(P_d(p, q) = p^d + H_1(p, q) \),\\
(2) \ \(Q_d(p, q) \equiv 2^{d-1}q^d + pH_2(p, q) + H_3(p, q) \quad mod \ (p^2 +18p -8q -27)\),\\
where \ \(H_j(p, q) \in  { \mathbb Z}[p, q], \ j = 1, 2, 3,\)\ and the degrees of \ \( H_1(p, q) \) \ and \ \( H_3(p, q) \) \ are less than or equal to \ \(d-1\) \ and \ \( H_2(p, q) \)\ is a homogeneous polynomial of degree \ \(d-1\).
\end{lemma}
%%%%%%%%%%%%%%%%%%%%%%%%%%%proof
\begin{proof}
(1) The proof is given in Lemma  \ref{lemma1}(1). \ 
(2)  We consider   a monomial  \ \(p^{2n+j}q^k\),  where \ \(j, k, n \in {\mathbb N} \cup\{0\}\) \ and \ \(0 \le j \le 1\).  We define a function \(m\) by \ \(m(p^{2n+j}q^k) = (27 - 18q + 8q)^np^jq^k\).   We replace any term \ \(a_{ij}p^{i}q^j\) \ in \ \(Q_d(p, q)\) \ with \ \(a_{ij}m(p^{i}q^j)\). \\
(i)  The monomials with the xy-degree \(3d\) : \ \(q^d, q^{d-2}p^3, \dots , q^{d-2k}p^{3k}.\)  Then
\[ m(q^{d-2}p^3) = (27 - 18p + 8q)pq^{d-2}, \ \dots \ .\]
(ii) The monomials with the xy-degree \(3d-1\) : \(q^{d-1}p, q^{d-3}p^4, q^{d-5}p^{7}, \dots .\)  Then
\[m(q^{d-3}p^4) = (27 - 18p + 8q)^2q^{d-3}, \ \dots \ .\]
Then to prove that the replaced polynomial satisfies the properties of the right hand side of (2),  it suffices to prove the term \(2^{d-1}q^d\)  occurs in \ \(Q_d(p, q)\).  
We see in the proof of  Lemma  \ref{lemma1} that 
 \[X_1^3 - 3X_1Y_1^2 = \sum_{h=0}^ka_hq({\bf x})^{d-2h}p({\bf x})^{3h}.\]
We substitute \ \(x^2\) \ for \ \(-y^2\) in the both sides.
Let \(u\) and \(v^2\) denote the polynomials obtained by substituting \(x^2\) for \ \((iy)^2\) \ in \(X_1\) and \ \(Y_1^2\),  respectively.  Then \ \(u = mx^d\) \ and \ \(v^2 = -n^2x^{2d}\),  \ where
\[m = \sum_{j=0}^{[d/2]}{}_dC_{2j}, \quad  n = \sum_{j=0}^{[(d-1)/2]}{}_dC_{2j+1}.\]
By the identity \ \({}_dC_{k} = {}_{d-1}C_k +  {}_{d-1}C_{k-1},\) \ we know \ \(m = n = 2^d/2 = 2^{d-1}\).   \ Hence  \ \(a_0 = (m^3 + 3mn^2)/4^d = 2^{d-1}.\)
\end{proof}
The projective closure of \ \(V(p^2 + 18p - 8q -27)\) \ is \ \(V(p^2 + 18pr - 8qr -27r^2)\).   Hence we define a morphism \ \(\bar{g}_d\) \ on \ \({\mathbb P}^2({\mathbb C}) \) \ by the similar way to (3.12).
%%%%pro3.14
\begin{pro} \label{pro3.13}
There is a morphism \ \(\bar{g}_d : {\mathbb P}^2({\mathbb C}) \to  {\mathbb P}^2({\mathbb C})\) \ such that
\[\bar{g}_d(V(p^2 + 18pr - 8qr -27r^2)) = V(p^2 + 18pr - 8qr -27r^2).\]
 \end{pro}
%%%%%%%%%%%%%%%%%%%%%%%%%%%proof
\begin{proof}
The proof is similar to the argument given in Proposition \ref{pro3.9}.
\end{proof}
\S 3.4.  An infinite number of invariant curves.

In Sections 3.2 and 3.3, we see two invariant curves \(C_C\) and  \(C_D\) under \(g_d\).  The invariability of  \(C_C\) is based on the fact that  \(C_C\) corresponds to \ \(\alpha = 0\)  and \  \(g_d(C_C)\) \ corresponds to \ \(d\alpha = 0\).  Also the invariability of  \(C_D\) is based on the fact that  \(C_D\) corresponds to \ \(\beta = 0\)  and \  \(g_d(C_D)\) \ corresponds to \ \(d\beta = 0\).   These properties can be generalized.  Let \(C\) be the curve corresponding to \ \(\beta - \frac nm \alpha = 0,  \  (m, n  \in { \mathbb N)}\). Then \ \(g_d(C)\) \ corresponds to \ \(d\beta - \frac nm d\alpha = 0\).  Then we will show that there is an infinite number of invariant affine algebraic curves under \(g_d\) for any \ \(\alpha \in { \mathbb N}\).

To prove this we need some preparations.  We begin to show the following lemma.
%%%%%%%%%%%lemma3.15
\begin{lemma} \label{lemma315}
Assume that an affine algebraic set \(C\) in \ \( { \mathbb C}^2\) has a polynomial parametric representation \ \(x = p(t), \ y = q(t), \) where \ \(p(t), \ q(t) \in { \mathbb C}[t]\).  Then the  polynomial parametric representation covers all points of \(C\).
\end{lemma}

We remark that by Proposition 5 in \cite[Chapter 4 \S5]{CLS},  we know that \(C\) is irreducible.  Then if \ \(p(t)\) \ and \  \(q(t)\) \ are non-constant polynomials, then  \(C\) is an affine algebraic curve. 
%%%%%%%%%%%%%%%%%%%%%%%%%%%proof
\begin{proof}
Let \ \(f_1 = x - p(t)\) \ and \  \(f_2 = y - q(t)\).  Set \ \(I = (f_1, f_2)\).

Let \ \(I_1 = I \cap  { \mathbb C}[x, y]\).  Then \ \(C = V(I_1)\).

Let \ \(p(t) = a_nt^n + a_{n-1}t^{n-1} +  \cdot\cdot \cdot + a_0, \enskip a_n \ne 0\)

\qquad \(q(t) = b_mt^m + b_{m-1}t^{m-1} +  \cdot\cdot \cdot + b_0, \enskip b_n \ne 0\).

We assume the lex order with \ \(t > x > y\).  Then by Corollary 4 in \cite[Chapter 3 \S1]{CLS},  we see that if \ \((x_0, y_0) \in V(I_1)\), then there is a \(t_0 \in  { \mathbb C}\) \ such that \ \((t_0, x_0, y_0) \in V(I)\).
\end{proof}

We consider the case that \ \(\beta = \frac nm \alpha, \  (m, n  \in { \mathbb N)}\) \ in (3.9).  Set \ \(t = 2\cos\frac \beta n\).  Then
\ \(2\cos 3\alpha = T_{3m}(t)\) \ and \ \(2\cos \beta = T_{n}(t)\).  Then we define an affine algebraic curve \ \(C_{m,n}\) \ by the following polynomial parametric representation :
%%%%%%%%%%%%%%%%%%%%%%%%equation 3.16
\begin{equation}
\begin{split}
p(t) = 1 + T_{3m}(t)T_n(t) + T_n(t)^2,  \qquad \qquad \qquad \qquad\\
q(t) = \frac 12(-2 + T_{3m}(t)^2 + 6T_n(t)^2 + T_{3m}(t)T_n(t)(3 +  T_n(t)^2).
\end{split}
\end{equation}
%%%%%%%%%%%lemma3.16
\begin{lemma} \label{lemma316}
Let \(C_{m,n}\) \ be the affine algebraic curve defined in (3.16).  Then \ \(g_d(C_{m,n}) = C_{m,n}\).
\end{lemma}
%%%%%%%%%%%%%%%%%%%%%%%%%%%proof
\begin{proof}
By Lemma 3.15, we may assume that any point of \ \(C_{m,n}\) \ is written as \ \((p(t), \ q(t))\) in (3.16).  Set  
\ \(g_d(p(t), \ q(t)) = (P(t), \ Q(t))\).   Then by Theorem 3.3(1), \ \( (P(t), \ Q(t)) = (p(T_d(t), \ q(T_d(t))\).  Then this lemma follows.
\end{proof}
%%%%pro3.17
\begin{pro} \label{pro317}
There is an infinite number of invariant affine algebraic curves under \ \(g_d\) \ for any \ \(d \in { \mathbb N}\).
 \end{pro}
%%%%%%%%%%%%%%%%%%%%%%%%%%%proof
\begin{proof}
We consider a family of curves \ \(\{C_{2k,3} : k = 1, 2, \dots\}\).  By Lemma 3.16, each curve \ \(C_{2k, 3}\) \ is invariant under \(g_d\).  Then, to prove this proposition it suffices to show that all curves \ \(C_{2k, 3}\) \ are different from each other.  To prove this, we consider continuous arcs \ 
\(\gamma_{2k} : = C_{2k, 3} \cap   K(g_d\mid_{{\mathbb C}^2})\).  Then we will show that the point \ \(P = (1, -1)\) \ is an end point of each \ \(\gamma_{2k}\)  \ and that the slopes of the tangent lines of \ \(\gamma_{2k}\) \ at \(P\) in \ \( K(g_d\mid_{{\mathbb C}^2})\) \ are different from each other. The point \ \(P = (1, -1)\) \ is a fixed point of \(g_d\).

To prove the above facts we need some preparations.  Let \ \(T_d(z)\) \ be the \(d^{th}\) Chebyshev polynomial in one variable.  Then
%%%%%%%%%%%%%%%%%%%%%%%%equation 3.17 
\begin{equation}
T_{d+2}(z) = zT_{d+1}(z) - T_{d}(z), \enskip T_{0}(z) = 2,   \enskip T_{1}(z) = z.
\end{equation}
By induction, we have
%%%%%%%%%%%%%%%%%%%%%%%%equation 3.18 
\begin{equation}
\begin{split}
T_{2k}(-2) = 2, \enskip T_{2k+1}(-2) = -2, \enskip  T_{2k}(0) = (-1)^k\times 2, \\
 T_{2k+1}(0) = 0, \enskip  T_{2k}(2) = T_{2k+1}(2) =2, \enskip \mbox{for} \enskip k \in {\mathbb N}.
\end{split}
\end{equation}
Let \ \(T'_d(z)\)\  denote the derivative of \ \(T_d(z)\).  Then by (3.17) and (3.18), we have
%%%%%%%%%%%%%%%%%%%%%%%%equation 3.19 
\begin{equation}
T_{2k}'(-2) = -(2k)^2,\ \mbox{and} \enskip T_{2k+1}'(-2) = (2k+1)^2,  \enskip \mbox{for} \enskip k \in {\mathbb N}.
\end{equation}

We consider a polynomial parametric representation of \ \(C_{2k,3}\).  We set \ \(v = 2\cos \beta\).  Then \ \(2\cos 3\alpha = T_{2k}(v)\).  Then we have a polynomial parametric representation :
%%%%%%%%%%%%%%%%%%%%%%%%equation 3.20 
\begin{equation}
\begin{split}
p(v) = 1 + vT_{2k}(v) + v^2, \qquad \qquad \qquad \qquad\\
q(v) = \frac 12(-2 + T_{2k}(v)^2 + 6v^2 + T_{2k}(v)(3v + v^3)).
\end{split}
\end{equation}
By Theorem 3.3,
\[ C_{2k, 3} \cap   K(g_d\mid_{{\mathbb C}^2}) = \{(p(v), \ q(v)) : -2 \le v \le 2\}.\] 
Then we set 
\[\gamma_{2k}(v) = \{(p(v), \ q(v)) : -2 \le v \le 2\}.\]

We claim that if \ \(p(v) = 1\) \ and \ \(q(v) = -1,\) \ in \ \(C_{2k,3}\) \  for \ \(v \in  {\mathbb C}\), \ then \ \(v = -2\).\\
 Indeed.  Suppose  \ \(p(v) = 1\)  in (3.20).  Then \ \(v = 0\)\ \ or \ \(T_{2k}(v) = -v\).  If \ \(v = 0\), \ then by (3.18),  \ \(q(0) = 1 \neq -1\).  If \  \(T_{2k}(v) = -v\), \ then \ \(q(v) = -1 + 2v^2 - \frac 12 v^4\).   We assume that \ \(T_{2k}(v) = -v\), \ \(q(v) = -1\) \ and \ \(v \neq 0\).  Then \ \(v = \pm2.\) \ If \ \(v = 2, \enskip T_{2k}(2) = 2 \neq -2 \).  If \ \(v = -2\), \ then we have \ \(p(-2) = 1\) \ and \ \(q(-2) = -1.\)  This completes the proof the claim.

Hence if \ \(\gamma_{2k}(v) = (1, -1),\)  \ then \ \(v = -2\) \ and an end point of any \ \(\gamma_{2k}(v) \) \ is the point \ \(P = (1, -1)\).

We compute the slope of the tangent line of \ \(\gamma_{2k}(v) \)  \ at the point \(P\).  Clearly, \\
\[\frac{dp}{dv} = vT_{2k}'(v) + T_{2k}(v) +2v.\]  
Then by (3.18) and (3.19),
 \[\left. \frac{dp}{dv}\right |_{v = -2} = 2(2k)^2 -2 > 0.\]
Hence 
\[\left. \frac{dp}{dv}\right |_{v} > 0, \ \mbox{for} \ v \in (-2-\varepsilon, \ -2+\varepsilon), \ \mbox{with  a small positive number}\   \varepsilon.\]
Similarly, we have that \ \(dg/dv\) \ is a polynomial in \(v\) and 
\[\left. \frac{dq}{dv}\right |_{v = -2} = 5(2k)^2 + 3.\]
Then 
\[\left. \frac{dq}{dp}\right |_{(p,q) = (1,  -1)}  = \ \frac{\left. \frac{dq}{dv}\right |_{v = -2}} { \left. \frac{dp}{dv}\right |_{v = -2}} \ = \frac52 + \frac{4}{4k^2-1}.\]
Hence
\[\lim_{v \to -2+0}\left. \frac{dq}{dp}\right |_{\gamma_{2k}(v)}  =  \frac52 + \frac{4}{4k^2-1}.\]
The slopes of the tangent lines of the continuous arcs \({\gamma_{2k}(v)}, \ (k \in  { \mathbb N})\), \ at \ \(v = -2 + 0\) \ are different from each other.
\end{proof}
%%%%%%%%%%%%%%%%%%%%%%%%%%%%%%%%%%%%%%%%%%%%%%%%%%%%%%%%%%%%%%%%%%%%%%%%%%%%%%%%%%%%%%%%%%%%
%%%%%%%%%%%%%%%%%%%%%%%%%%%%%%%%%%%%%%%%Section 4
\section{Morphisms on \ \( { \mathbb C}^3 / D_{4}\)}
%%%%%%%%Subsection4.1
\subsection{\rm{The orbit  variety\ \( { \mathbb C}^3 / D_{4}\) \ and morphisms defined over \({ \mathbb Z}\)}}\hspace{0cm}

We begin with the study of the orbit variety \ \( { \mathbb C}^3 / D_{4}\).  We set \ \(x : = x_1, \quad y : = x_2\) \and \ \(z : = x_3\).  By Algorithm 2.5.14 in \cite{St},  we know that \ \(\{x^2 + y^2, \ z^2, \ z(x^2-y^2), \ x^2y^2\}\) \ is a  fundamental system of invariants  of \ \( { \mathbb C}[x, y, z]^{ D_{4}}\).  But, to simplify the \ syzygy \ ideal we select another fundamental system of invariants \ \(\{x^2 + y^2, \ z^2, \ z(x^2-y^2), \ (x^2-y^2)^2\}\).  Set
%%%%%%%%%%%%%%%%%%%%%%%%equation 4.1
\begin{equation}
 p = x^2 + y^2, \ q = z^2, \ r = z(x^2-y^2), \ s = (x^2-y^2)^2,
\end{equation}
and \ \(F = (p, q, r, s).\)

Then by Proposition 3 in \cite[Chapter 7 \S4]{CLS},  we  see that the  \ syzygy \ ideal \(I_F\) for \(F\) is the ideal \ \((r^2 - qs)\).  
By Theorem \ref{theorem1}, we have \  \({ \mathbb C}^3 / D_{4} \simeq V_F = V(I_F)\).
Then the orbit variety \  \( V_F\) \ can be embedded into \ \({ \mathbb C}^4\)  as an affine subvariety \ \(V(r^2-qs)\).  The subvariety \ \(V(r^2-qs)\)\ is a quadric hypersurface in  \ \({ \mathbb C}^4\), denoted by \(X_Q\).  The singular locus of \(X_Q\) is the set \ \(\{(p, 0, 0, 0) : p \in { \mathbb C}\}\).

Now we consider a morphism \(g_2\) on \(X_Q\).  By \cite[p.198]{U3}, we have 
\[f_2(x, y, z) = (x^2 - y^2 -2z, \ 2xy, \ z^2 - 2x^2 - 2y^2 +2).\]
Using the generators \ \(\{p, q, r, s\}\) in (4.1),  we can compute \ \(g_2\) :
%%%%%%%%%%%%%%%%%%%%%%%%equation 4.2
\begin{equation}
g_2(p, q, r, s) = (p^2 + 4q - 4r, (q-2p+2)^2,  (q-2p+2)(2s-p^2-4r+4q), (2s-p^2-4r+4q)^2). 
\end{equation}
The morphism \(g_2\) is defined  over\ \( { \mathbb Z}\).  This fact holds for any morphisms \(g_d\) on \(X_Q\).  

%%%%%%%%%%%%%%%%%%%%%%%%theorem 4.1
\begin{theorem} \label{theorem5}   
Let \(g_d\) be a morphism on \ \(X_Q\) as in Definition \ref{definition1}. Then  
 any morphism  \(g_d\) from \(X_Q\)  to   \(X_Q\)  is defined over\ \( { \mathbb Z}\).\\
\end{theorem}
%%%%%%%%%%%%%%%%%%%%%%%%%%%begin proof theorem4.1
\begin{proof}
 We use recurrence formulas of \(z_1^{(d)}\)  and  \(z_2^{(d)}\)  (See \cite{L} and  \cite[p.201]{U3}) :

 %%%%%equation 4.3
\begin{equation}
\begin{split}
z_1^{(k)} = z_1\cdot z_1^{(k-1)} - z_2\cdot z_1^{(k-2)} + z_3\cdot z_1^{(k-3)} - z_1^{(k-4)} ,\\
z_1^{(0)} = 4, \ z_1^{(1)} = z_1, \ z_1^{(2)} = z_1^2 -2z_2, \ z_1^{(3)} = z_1^3 -3z_1z_2 + z_3,  
\end{split}
\end{equation}

%%%%%equation 4.4
\begin{equation}
\begin{split}
z_2^{(k)} = z_2\cdot z_2^{(k-1)} - (z_1z_3-1)\cdot z_2^{(k-2)} + (z_1^2-2z_2 + z_3^2)\cdot z_2^{(k-3)}\\
 - (z_1z_3-1)\cdot z_2^{(k-4)} +  z_2\cdot z_2^{(k-5)} -  z_2^{(k-6)} ,\\
z_{2}^{(-2)} = z_2^2-2z_1z_2+2,\ z_{2}^{(-1)} = z_2, \ z_{2}^{(0)} = 6, \ z_{2}^{(1)} = z_2, \\
 z_{2}^{(2)} = z_{2}^{(-2)}, \ z_{2}^{(3)} = z_2^3-3z_1z_2z_3+3z_3^2+3z_1^2-3z_2.
\end{split}
\end{equation}
We assume that \ \(z_1 = \bar{z}_3\) \ and \ \(z_2 \in {\mathbb R}\).  Let  \ \(z_1 = x + iy\) and \(\ z_2 = z\)  where \(x, y, z\) are real variables.   We set 
%%%%%equation 4.5
\begin{equation}
\begin{split}
z_1^{(d)} (x + iy, \ z, \ x-iy) = X_d(x, \ y, \ z) + iY_d(x, \ y, \ z),\\
  \quad z_2^{(d)} (x + iy, \ z, \ x-iy) = Z_d(x, \ y, \ z), \\ 
 \mbox{where}  \ X_d,\ Y_d, \ Z_d \in {\mathbb R}[x, \ y, \ z].
\end{split}
\end{equation}
By Proposition \ref{pro2.6} and (4.2), to prove this theorem it suffices to consider only  the case \(d\) is odd.  We assume \(d\) is a positive old integer.  Then it is enough to prove the following lemma.  Set \ \(R =  { \mathbb Z}[x^2+y^2, \ z^2, \ z(x^2-y^2), \ (x^2-y^2)^2]\).
%%%%%%%%%%%lemma4.2.
\begin{lemma} \label{lemma3}
Let \(d\) be a positive odd integer.\\
(a)  \ \(X_d^2 + Y_d^2  \in R\). \ (b) \  \(Z_d^2 \in R\).\  (c) \  \(Z_d(X_d^2 - Y_d^2) \in R \). 
\  (d)\  \((X_d^2 - Y_d^2)^2 \in R \).
\end{lemma}
%%%%%%%%%%%%%%%%%%%%%%%%%%%begin proof of Lemma 4.2
\begin{proof}[Proof of (b)] Using (4.4), we can prove the following properties (4.6) by induction:
%%%%%equation 4.6
\begin{equation}
\begin{split}
 \mbox{if} \ k \ \mbox{is odd}, \quad Z_k = c_1(x^2-y^2) + c_2z,\\
\mbox{if} \ k \ \mbox{is even}, \quad Z_k = c_3z(x^2-y^2) + c_4,\\
\mbox{where} \  c_1, c_2, c_3, c_4 \in { \mathbb Z}[x^2+y^2, z^2,  (x^2-y^2)^2].
\end{split}
\end{equation}
Then the assertion (b) follows.
\end{proof}
To prove the assertions (a), (c) and (d), it suffices to show the following lemma.
%%%%%%%%%%%lemma4.3.
\begin{lemma} \label{lemma4}
Let \(d\) be a positive odd integer.\\
(a)  \ \(X_d^2 + Y_d^2  \in R\).  \ (e)\  \(X_d^2 - Y_d^2 \in (x^2-y^2)R + zR \).
\end{lemma}
The assertion (c) follows from (4.6) and (e).   The assertion (d) follows from  (e).   The proof of Lemma \ref{lemma4} needs some preparations and will be given after Lemma \ref{lemma8}.

We begin to show the following lemma.
%%%%%%%%%%%lemma4.4.
\begin{lemma} \label{lemma5}
Let \(d\) be a positive odd integer.  (1) Let \ \(\alpha_{abc}z_1^az_2^bz_3^c\) \ be a term in \ \(z_1^{(d)}(z_1, z_2, z_3)\).  Then \ \(a + 2b + 3c \equiv d\) \  (mod 4) \ and  \ \(\alpha_{abc} \in  { \mathbb Z}\).\\
(2) \ The term   \(z_1^{d}\) occurs in  \(z_1^{(d)}\) and it is the only term with degree \(d\).
\end{lemma}
%%%%%%%%%%%%%%%%%%%%%%%%%%%proof
\begin{proof}
Using (4.3) ,  we can prove this lemma by induction.
\end{proof}
We define the weighted degree of \ \(z_1^az_2^bz_3^c\) \ to be  \ \(a + 2b + 3c\).
Set \ \(R_0 =  { \mathbb Z}[x^2+y^2, z^2]\).  Since \ \(z_1z_3 = x^2+y^2\) \ and  \(z_2^2 = z^2\),  we pick out monomials \ \(z_1z_3 \) \ and \(z_2^2\)  \ in \(z_1^{(d)}\)  as coefficients belonging to \(R_0\).  Then we may view the polynomial \(z_1^{(d)}\) as a polynomial  \(\tilde{z}_1^{(d)}\)   whose terms are written as 
%%%%%%%%%%%%%%%%%%%%%%%%equation 4.7
\begin{equation}
\alpha z_1^{n_1} \ \mbox{or} \ \beta z_1^{n_2}z_2 \ \mbox{or} \ \gamma z_3^{n_3} \ \mbox{or} \ \delta z_3^{n_4}z_2, \ \mbox{where} \ \alpha, \beta, \gamma, \delta\  \in  R_0, \ n_j \in {\mathbb N}, \ (j = 1, 2, 3, 4).
\end{equation}
The weighted degrees of \ \(z_1z_3\) \ and \(z_2^2\) are 4.   Then by Lemma \ref{lemma5}(1), we see that if \(n_j^,\)s are the degrees of monomials in (4.7), then 
%%%%%%%%%%%%%%%%%%%%%%%%equation 4.8
\begin{equation}
n_1 \equiv n_2 + 2 \equiv 3{n_3} \equiv 3{n_4} + 2 \equiv d \ (\mbox{mod} \ 4). 
\end{equation}
\[\mbox{Let}\quad \tilde{z}_1^{(d)} = \sum a(1, n_1)z_1^{n_1} +  \sum a(2, n_2)z_1^{n_2}z_2 + \sum a(3, n_3)z_3^{n_3} + \sum a(4, n_4)z_1^{n_4}z_2, \] 
where the sums are over \ \(n_j \in I_j \quad j = 1, 2, 3, 4.\)  Here \ \(a(j, n_j) \in R_0, \ j = 1, 2, 3, 4.\)

Now we assume that \ \(z_1 = x + iy, \ z_2 = z, \  z_3 = x - iy\).  Let 
\ \(z_1^m = U_m + iV_m, \ U_m, \ V_m \in  { \mathbb Z}[x, y]\).  Then \ \(z_3^m = U_m - iV_m\).  Hence 
\[\tilde{z}_1^{(d)}(x+iy, z, x-iy) = \sum a(1, n_1)(U_{n_1}+iV_{n_1}) +  \sum a(2, n_2)(U_{n_2}+iV_{n_2})z\] 
\[ + \sum a(3, n_3)(U_{n_3}-iV_{n_3}) + \sum a(4, n_4)(U_{n_4}-iV_{n_4})z. \] 
Then by (4.5),
\[X_d = \sum a(1, n_1)U_{n_1} +  \sum a(2, n_2)U_{n_2}z + \sum a(3, n_3)U_{n_3} + \sum a(4, n_4)U_{n_4}z, \] 
\[Y_d = \sum a(1, n_1)V_{n_1} +  \sum a(2, n_2)V_{n_2}z - \sum a(3, n_3)V_{n_3} - \sum a(4, n_4)V_{n_4}z. \] 
Hence
\[X_d^2 = \sum_{1\le j, k \le 4}(\sum a(j, m_j)U_{m_j}z^{\lambda(j)})(\sum a(k, n_k)U_{n_k}z^{\lambda(k)}) ,\] 
\[Y_d^2 = \sum_{1\le j, k \le 4}\varepsilon_{jk}(\sum a(j, m_j)V_{m_j}z^{\lambda(j)})(\sum a(k, n_k)V_{n_k}z^{\lambda(k)}), \] 
where if \ \(j = 2\) \ or \(4\), then \ \(\lambda(j) = 1\), otherwise \ \(\lambda(j) = 0\),  and if \ \(\{j, k\} = \{1, 3\}\) \ or  \ \(\{1, 4\}\) \ or \ \(\{2, 3\}\) \ or \ \(\{2, 4\}\),  \ then \ \(\varepsilon_{jk} = -1\), otherwise \ \(\varepsilon_{jk} = 1\).  Set
%%%%%%%%%%%%%%%%%%%%%%%%equation 4.9
\begin{equation}
s(\pm, j, k, m_j, n_k) = U_{m_j}U_{n_k} \pm \varepsilon_{jk}V_{m_j}V_{n_k},
\end{equation}
and
%%%%%%%%%%%%%%%%%%%%%%%%equation 4.10
\begin{equation}
s(\pm, j, k) = \sum a(j, m_j) a(k, n_k)s(\pm, j, k, m_j, n_k)z^{\lambda(j)+\lambda(k)}, 
\end{equation}
Then
%%%%%%%%%%%%%%%%%%%%%%%%equation 4.11
\begin{equation}
X_d^2 \pm Y_d^2 = \sum_{1\le j, k \le 4}s(\pm, j, k) .
\end{equation}

First, we compute \ \(s(\pm, j, k, m_j, n_k)\).  
Note that
%%%%%%%%%%%%%%%%%%%%%%%%equation 4.12
\begin{equation}
Re (z_1^mz_3^n) =  U_mU_n + V_mV_n, \quad Re(z_1^{m+n}) =  U_mU_n - V_mV_n.
\end{equation}

By (4.9) and (4.12), we have
%%%%%%%%%%%%%%%%%%%%%%%%equation 4.13
\begin{equation}
s(+, j, k, m_j, n_k) = \left\{
\begin{array}{ll}
\displaystyle{Re (z_1^{m_j}z_3^{n_k})}& \quad \ \mbox{if}\ \varepsilon_{jk} = 1\\
\displaystyle{Re (z_1^{m_j+n_k})}& \quad \ \mbox{if}\ \varepsilon_{jk} = -1
\end{array} \right.
\end{equation}
%%%%%%%%%%%%%%%%%%%%%%%%equation 4.14
\begin{equation}
s(-, j, k, m_j, n_k) = \left\{
\begin{array}{ll}
\displaystyle{Re (z_1^{m_j+n_k})}& \quad \ \mbox{if}\ \varepsilon_{jk} = 1\\
\displaystyle{Re (z_1^{m_j}z_3^{n_k})}& \quad \ \mbox{if}\ \varepsilon_{jk} = -1
\end{array} \right.
\end{equation}
Note that  \ \(Re(z_1^mz_3^n) = Re(z_1^nz_3^m).\) \   By (4.13) and (4.14),  we have 
\[s(+, j, k, m_j, n_k) = c_1 Re (z_1^{m_j - \varepsilon_{j_k}n_k})\ \ \mbox{or} \ c_2 Re (z_1^{n_k - \varepsilon_{j_k}m_k}), \]
\[s(-, j, k, m_j, n_k) = c_3 Re (z_1^{m_j + \varepsilon_{j_k}n_k})\ \ \mbox{or} \ c_4 Re (z_1^{n_k + \varepsilon_{j_k}m_k}), \]
where \ \(c_j \in  { \mathbb Z}[x^2+y^2], \ j= 1, 2, 3, 4.\)
%%%%%%%%%%%lemma4.5.
\begin{lemma} \label{lemma6}
Let \(d\) be a positive odd integer.  Then :\\
(1)  If \ \(\mid j-k \mid\) \ is even, then \ \(m_j + \varepsilon_{jk}n_k \equiv 2\), \ \(m_j - \varepsilon_{jk}n_k \equiv 0\) \quad (mod 4),\\
\qquad \(n_k + \varepsilon_{jk}m_j \equiv 2\), \ \(n_k - \varepsilon_{jk}m_j \equiv 0\) \quad (mod 4).\\
(2)  If \ \(\mid j-k \mid\) \ is odd, then \ \(m_j + \varepsilon_{jk}n_k \equiv 0\), \ \(m_j - \varepsilon_{jk}n_k \equiv 2\) \quad (mod 4),\\
\qquad \(n_k + \varepsilon_{jk}m_j \equiv 0\), \ \(n_k - \varepsilon_{jk}m_j \equiv 2\) \quad (mod 4).\\
\end{lemma}
%%%%%%%%%%%%%%%%%%%%%%%%%%%proof
\begin{proof}
From (4.8) and the definition of \ \(\varepsilon_{jk}\) \ we have this lemma.
\end{proof}
%%%%%%%%%%%lemma4.6.
\begin{lemma} \label{lemma7}
Let \ \(R_1 =  { \mathbb Z}[x^2+y^2, (x^2-y^2)^2]\).\\
(1)  If \ \(m \equiv 0\) \quad (mod 4),\  then \ \(Re(z_1^m) \in R_1\).\\
(2)  If \ \(m \equiv 2\) \quad (mod 4),\  then \ \(Re(z_1^m) \in (x^2-y^2)R_1\).
\end{lemma}
%%%%%%%%%%%%%%%%%%%%%%%%%%%proof
\begin{proof}
(1)  Clearly, \ \(z_1^{4h} = (z_1^2)^{2h} = (x^2 - y^2 + 2xyi)^{2h}, \ h \in {\mathbb N}\).  Hence
\[Re(z_1^{4h}) = \sum_{j=0}^h {}_{2h}C_{2j}(x^2-y^2)^{2h-2j}(2xyi)^{2j}\]
\[ = \sum_{j=0}^h {}_{2h}C_{2j}((x^2-y^2)^2)^{h-j}(-(x^2+y^2)^2+ (x^2-y^2)^2)^{j} \in R_1.\]
(2)  Clearly, \ \(z_1^{4h+2} = (x^2 - y^2 + 2xyi)^{2h+1}\).  Hence
\[Re(z_1^{4h+2}) = \sum_{j=0}^h {}_{2h+1}C_{2j}(x^2-y^2)^{2h+1-2j}(2xyi)^{2j}\]
\[ = \sum_{j=0}^h {}_{2h+1}C_{2j}(x^2-y^2)(x^2-y^2)^{2(h-j)} (-4x^2y^2)^{j} \in (x^2-y^2)R_1.\]
\end{proof}
Hence from Lemmas \ref{lemma6} and \ref{lemma7}, we have the following.
%%%%%%%%%%%lemma4.7.
\begin{lemma} \label{lemma8}
(1)  If \ \(\mid j - k\mid\) \ is even, then 
\[s(+, j, k, m_j, n_k) \in R_1 \ \mbox{and} \ s(-, j, k, m_j, n_k) \in (x^2 - y^2)R_1.\]
(2)  If \ \(\mid j - k\mid\) \ is odd, then 
\[s(+, j, k, m_j, n_k) \in (x^2 - y^2)R_1 \ \mbox{and} \ s(-, j, k, m_j, n_k) \in R_1.\]
\end{lemma}
Now we can prove Lemma \ref{lemma4}.
%%%%%%%%%%%%%%%%%%%%%%%%%%%begin proof of Lemma 4.3
\begin{proof}[proof of Lemma \ref{lemma4}.]
  Note that  if \ \(\mid j - k\mid\) \ is even, then  \ \(\lambda(j) + \lambda(k) = 0 \ \mbox{or}\ 2,\)  and  if \ \(\mid j - k\mid\) \ is odd, then  \ \(\lambda(j) + \lambda(k) = 1\).  Then by (4.10) and Lemma \ref{lemma8}, we have 
\begin{equation*}
\begin{split}
s(+, j, k) \in R, \hspace{5.75cm} \\
s(-, j, k) \in \left\{
\begin{array}{ll}
\displaystyle{(x^2 - y^2)R}& \quad \ \mbox{if}\ \mid j - k \mid \ \mbox{is even},\\
\displaystyle{\quad zR}& \quad \ \mbox{if}\ \mid j - k \mid \ \mbox{is odd}.
\end{array} \right.
\end{split}
\end{equation*}
\[\mbox{Hence} \quad X_d^2 + Y_d^2 \in R \ \mbox{and} \ X_d^2 - Y_d^2 \in zR + (x^2-y^2)R.\qquad \qquad \qquad\]
%%%%%%%%%%%%%%%%%%%%%%%%%%%end proof of Lemma 4.3
\end{proof}
%%%%%%%%%%%%%%%%%%%%%%%%%%%end proof theorem4.1
Then Theorem \ref{theorem5} is proved.
\end{proof}

We remark that the original basis \(\{p, q, r, s^\prime\}\) with \(s^\prime = x^2y^2\)  satisfies 
\[{ \mathbb R}[x, y, z]^{D_4} \cap   { \mathbb Z}[x, y, z] = { \mathbb Z}[p, q, r, s^\prime],\]
while \ \(\{p, q, r, s\}\) does  not. 

Apart from the fundamental system \ \(\{p, q, r, s\}\), there are many fundamental systems of invariants of \ \( {\mathbb C}[x, y, z]^{D_4}\).  Let \ \(F'\) be the ideal generated by such a fundamental system.  Then there are the corresponding affine algebraic varieties \ \(V_{F'}\) \ and morphisms \ \(g'_d\) \ on\ \(V_{F'}\).  We will show in Section 4.7 that there is an isomorphism \ 
\(\varphi : X_Q \to V_{F'}\ \mbox{such that}\ \varphi \circ g_d = g'_d \circ \varphi.\)

We consider the dynamics of \(g_d\) on \(X_Q\).  As in Section 3.1, we introduce a parametrization of  \(X_Q\) which is convenient to understand the dynamics of \(g_d\).  From (2.13), we have 
%%%%%%equation 4.15
\begin{equation}
\begin {array}{ccc}
  \{t_1, t_2, t_3, t_4\} & \spmapright{ }& \{t_1^d, t_2^d, t_3^d, t_4^d\}\\
  \lmapupdown{} &  & \lmapupdown{} \\
  (z_1,z_2,z_3) & \sbmapright{T_d} & (z_1^{(d)}, z_2^{(d)}, z_3^{(d)}) \\
 \rmapdown{M_3^{-1}} &  & \rmapdown{M_3^{-1}} \\
  (x, y, z) & \sbmapright{f_d} & (X, Y, Z) \\
\rmapdown{\tilde{F}} &  & \rmapdown{\tilde{F}} \\
  (p, q, r, s) & \sbmapright{g_d} & (P_d, Q_d, R_d, S_d). 
 \end{array}
\end{equation} 
Set 
 %%%%%equation 4.16
\begin{equation}
t_1 = e^{i(\alpha + \gamma)}, \ t_2 = e^{i(\alpha - \gamma)}, \ t_3 = e^{i\beta}, \ t_4 = e^{-i(2\alpha + \beta)}, \quad \mbox{for} \  \alpha, \beta, \gamma \in {\mathbb C}.
 \end{equation}
 Then
\[z_1 = e^{i(\alpha + \gamma)} +  e^{i(\alpha - \gamma)} + e^{i\beta} + e^{-i(2\alpha + \beta)}, \ z_2 = 2(\cos2\alpha + \cos(\alpha + \beta + \gamma ) + \cos(\alpha + \beta -\gamma)), \]
\[z_3 = e^{-i(\alpha + \gamma)} +  e^{-i(\alpha - \gamma)} + e^{-i\beta} + e^{i(2\alpha + \beta)}.\] 
Hence by (2.8) we have
%%%%%equation 4.17
\begin{equation}
\begin{split}
x =  \cos(\alpha + \gamma) +  \cos(\alpha -\gamma) +  \cos \beta +  \cos(2\alpha + \beta), \qquad\qquad\\
y =  \sin(\alpha + \gamma) +  \sin(\alpha -\gamma) +  \sin \beta -  \sin(2\alpha + \beta), \qquad\qquad \\
z =  2\cos 2\alpha  +  (2\cos(\alpha + \beta))(2\cos \gamma). \qquad\qquad\qquad\qquad\qquad
\end{split}
\end{equation}
Therefore, using trigonometrical identities we have
%%%%%equation 4.18
\begin{equation}
\begin{split}
p(\alpha, \beta, \gamma) =  (2\cos\gamma)^2 +   (2\cos\gamma)(2\cos 2\alpha)(2\cos(\alpha + \beta)) + (2\cos(\alpha + \beta))^2,  \qquad \\
q(\alpha, \beta, \gamma) =  ((2\cos(\alpha + \beta))(2\cos\gamma) +  (2\cos 2\alpha))^2, \qquad\qquad\qquad\qquad  \qquad\qquad\\
r(\alpha, \beta, \gamma) =  ((2\cos(\alpha + \beta))(2\cos\gamma) +  (2\cos 2\alpha))\times \quad\qquad\qquad\qquad  \qquad\qquad\quad \\
 (2(2\cos(\alpha + \beta))(2\cos\gamma) + \frac 12(2\cos 2\alpha)((2\cos(\alpha + \beta))^2 + (2\cos\gamma)^2)),\\
s(\alpha, \beta, \gamma) =  (2(2\cos(\alpha + \beta))(2\cos\gamma) + \frac 12(2\cos 2\alpha)((2\cos(\alpha + \beta))^2 + (2\cos\gamma)^2))^2.
\end{split}
\end{equation}
Note that all the identities which we use to get (4.18) are valid for complex variables \ \(\alpha, \beta\) and \(\gamma\).
Then setting \ \(u = 2\cos 2\alpha, \ v = 2\cos(\alpha + \beta)\) \ and \ \(w = 2\cos\gamma\), \ we have 
%%%%%%%%%%%%%%%%%%%%%%%%%equation4.19
\begin{equation}
\begin{split}
p(u, v, w) = v^2 + uvw + w^2,\qquad\qquad\\
q(u, v, w) = (u + vw)^2,\qquad\qquad\\
r(u, v, w) = (u + vw)(2vw + \frac 12u(v^2 + w^2)),\\
s(u, v, w) = (2vw + \frac 12u(v^2 + w^2))^2. \qquad
\end{split}
\end{equation}
Let \(G_X\) be a mapping from   \({\mathbb C}^3\) to \({\mathbb C}^4\)  defined by 
\[G_X(u, v, w) =  (v^2 + uvw + w^2, \  (u + vw)^2, \  (u + vw)(2vw + \frac 12u(v^2 + w^2)), \  (2vw + \frac 12u(v^2 + w^2))^2).\]

We show that the mapping \ \(G_X(u, v, w)\) \ is a polynomial parametric representation of \(X_Q\).   Indeed. 
By the definition of \(G_X\)  and (4.19), we have  \ \(r(u, v, w)^2 = q(u, v, w) s(u, v, w).\)  Then for any \ \((u, v, w) \in {\mathbb C}^3,\  G_X(u, v, w) \in V(r^2-qs) = X_Q.\)
By an argument similar to that use in the proof of Proposition 3.2,  then we can show that  \(G_X({\mathbb C}^3)\)  covers all points of \(X_Q\).

Using this parametrization we describe the dynamics of  \(g_d\)  on  \(X_Q\). 
%%%%%%%%%%%%%%%%%%%%%%%%theorem 4.8
\begin{theorem}\label{theorem4.8}  
Let \ \(g_d\) be a morphism on \ \(X_Q\)  in Definition 2.4.
 \[(1) \ g_d(G_X(u, v, w)) = G_X(T_d(u), T_d(v), T_d(w)) \ \mbox{for any} \quad u, v, w  \in {\mathbb C}. \]
\[ (2) \ K(g_d \mid _{X_Q}) = \{G_X(2\cos2\alpha,  2\cos(\alpha+\beta), 2\cos\gamma)  : \ 0 \le \alpha,  \beta, \gamma < 2\pi\} \ \mbox{for} \quad d \ge 2.\]
\end{theorem}
%%%%%%%%%%%%%%%%%%%%%%%%%%%proof
\begin{proof}
(1)  The proof of the assertion (1) is similar to the proof of Theorem 3.3 (1).  Then it is omitted.\\
(2)  The proof of the assertion (2) is similar to the proof of Theorem 3.3 (2).    Let \ \(u = t_1 + t_1^{-1}\), \ \(v = t_2 + t_2^{-1}\) \ and \ \(w = t_3 + t_3^{-1}\), \ with \ \(t_j = r_je^{i\theta_j}, \ r_j \geq 1, \ \theta_j \in { \mathbb R},\) \ for \ \(j = 1, 2, 3.\)   Then  \ \(T_d(u) = t_1^{d} + t_1^{{-d}}\) \ and \ \(T_d^n(u) = t_1^{d^n} + t_1^{{-d}^n}\).  Let \ \(G_X(u, v, w) = (p(u, v, w), \ q(u, v, w), \ r(u, v, w), \ s(u, v, w))\).   Set
\[g_d^n(p(u, v, w), \ q(u, v, w), \ r(u, v, w), \ s(u, v, w)) = (p_n(u, v, w), \ q_n(u, v, w), \ r_n(u, v, w), \ s_n(u, v, w)).\] 
We assume that \ \(\{q_n(u, v, w) : n = 1, 2, \dots\}\) \ is bounded.  Then we have
%%%%%%%%%%%%%%%%%%%%%%%equation 4.20
\begin{equation}
r_1 = r_2r_3.
\end{equation}
Besides, we assume that \ \(\{p_n(u, v, w) : n = 1, 2, \dots\}\) \ is bounded. 
We consider the set \ \(Dom(p) = max\{r_2^2, \ r_1r_2r_3, \ r_3^2\}.\)  If  \ \(Dom(p)\) \ is a singleton,  then \ \(\lim_{n\to\infty}p_n = \infty.\)  If \ \(r_2 = r_3\),  then we have \ \(r_1 = r_3^2\) \ by (4.20).  If \ \(r_2 = r_3 > 1\),  then \ \(r_1r_2r_3 > r_2^2 = r_3^2\).
We suppose that \ \(r_1r_2r_3  = r_2^2\).  Then by (4.20) we have \ \(r_1 = r_2\) \ and \ \(r_3 = 1\).  If \ \(r_1 = r_2 > 1\), then  \ \(\lim_{n\to\infty}s_n = \infty.\)   Hence if \ \(\{g_d^n(G_X(u, v, w)) : n = 1, 2, \dots\}\) \ is bounded, then \ \(r_1 = r_2 = r_3 = 1\).  In the case \ \(r_1r_2r_3  = r_3^2\), we get also \ \(r_1 = r_2 = r_3 = 1\).
\end{proof}
%%%%%%%%%%%%%%%%%%%%%%%%%%%%%%%%%%%%%%%%%%%%%%%%%%%%%%%%%%%%%%%%%%%%%%%%%%%%%%%%%%%%%%%%%
%%%%%%%%Subsection4.2
\subsection{\rm{The branch loci of \  \(g_2(p, q, r, s)\)  and  \(\tilde{F}(p, q, r, s)\).}}\hspace{1cm}

As in Section 3.1, we consider the branch locus  of \ \(g_2(p, q, r, s)\).  We will show that the branch locus  of \ \(g_2\) \ consists of three affine algebraic varieties \(L_1\),   \(S_P\) and \(S_A\);
%%%%%%%%%%%%%%%%%%%%%%%%equation 4.21
\begin{equation}
\begin{split}
L_1 : = V(q, r, s), \ S_P : = V(p^2-s, qs-r^2), \  S_A : = V(qs-r^2, A_h), \ \mbox{where} \hspace{3.5cm}\\
A_h(p, q, r, s) :  = 256 - 192 p + 48 p^2 - 4 p^3 - 128 q - 80 p q + p^2 q + 16 q^2 + 
 288 r + 36 p r - 8 q r - 108 s.
\end{split}
\end{equation}
The algebraic sets \(S_P\) and \(S_A\)  will be proved to be affine algebraic surfaces.

Recall that \(g_2\) is a morphism from \ \(X_Q\)  to \ \(X_Q\).  We may view it as a morphism from \ \(X\)  to \(Y\), where the coordinate rings \ \({\mathbb C}[X]\) \ and  \ \({\mathbb C}[Y]\) \ are given by \ 
\({\mathbb C}[p, q, r, s]/(qs-r^2)\) \  and \ 
\({\mathbb C}[a, b, c, d]/(bd-c^2)\), respectively.
We consider the inverse images of \ \((a, b, c, d) \in Y\) \ under \(g_2\).  We suppose that \ \(g_2(p, q, r, s) = (a, b, c, d)\).  A polynomial parametric representation of \ \(V(bd - c^2)\) \ is given by \ \(b = u^2, \ c = uv\) \ and \ \(d = v^2\).  So by (4.2), we consider equations : 
%%%%%equation 4.22
\begin{equation}\left\{
\begin{split}
p^2 + 4q -4r - a= 0,\qquad\qquad\\
q - 2p +2 - u = 0,\qquad\qquad\\
(q - 2p +2)(2s -p^2 +4q -4r)  - uv = 0,\\
 2s - p^2 + 4q - 4r - v = 0,\qquad
\end{split}
\right .
\end{equation}
%%%%%equation 4.23
\begin{equation}\left\{
\begin{split}
p^2 + 4q - 4r - a= 0,\qquad\qquad\\
q - 2p +2 + u = 0,\qquad\qquad\\
(q - 2p +2)(2s -p^2 +4q -4r)  - uv = 0,\\
 2s - p^2 + 4q - 4r + v = 0.\qquad
\end{split}
\right .
\end{equation}
Let \(I_+\) \ be the ideal generated by the four polynomials in the left hand sides of (4.22).  Substituting -u and -v for u and v in \(I_+\), we have another ideal \(I_-\) for (4.23).
We compute Gr\(\ddot{\mbox o}\)bner basis \ \(\mathcal{G}_{\pm}\)\ for the ideals \ \(I_{\pm}+ (qs - r^2)\),  using lex order with \ \(q > r > s > p > a > u > v :\)
\[\mathcal{G}_+ = \{64 + a^2 - 128 p + 80 p^2 - 2 a p^2 - 16 p^3 + p^4 - 64 u + 64 p u - 
  8 p^2 u + 16 u^2 + 16 v - 16 p v - 8 u v,\]
\[ a - 2 p^2 + 2 s - v,  2 - 2 p + q - u, 8 + a - 8 p - p^2 + 4 r - 4 u\},\]
\[\mathcal{G}_- = \{64 + a^2 - 128 p + 80 p^2 - 2 a p^2 - 16 p^3 + p^4 + 64 u - 64 p u + 
  8 p^2 u + 16 u^2 - 16 v + 16 p v - 8 u v, \]
\[a - 2 p^2 + 2 s + v, 2 - 2 p + q + u, 8 + a - 8 p - p^2 + 4 r + 4 u\}.\]
Let
%%%%%equation 4.24
\begin{equation} \left\{
\begin{split}
a - 2p^2 + 2s - v = 0,\qquad\qquad\\
2 - 2p + q - u = 0,\qquad\qquad\\
8 + a -8p - p^2 + 4r - 4u = 0,\qquad
\end{split}
\right .
\end{equation}
%%%%%equation 4.25
\begin{equation} \left\{
\begin{split}
a - 2p^2 + 2s + v = 0,\qquad\qquad\\
2 - 2p + q + u = 0,\qquad\qquad\\
8 + a -8p - p^2 + 4r + 4u = 0.\qquad
\end{split}
\right .
\end{equation}
We set
\[h_{\pm} = \mathcal{G}_{\pm} \cap {\mathbb C}[a, p, u, v].\]
Then
\[h_+ = 64 + a^2 - 128 p + 80 p^2 - 2 a p^2 - 16 p^3 + p^4 - 64 u + 64 p u - 
  8 p^2 u + 16 u^2 + 16 v - 16 p v - 8 u v,\]
\[h_- = 64 + a^2 - 128 p + 80 p^2 - 2 a p^2 - 16 p^3 + p^4 + 64 u - 64 p u + 
  8 p^2 u + 16 u^2 - 16 v + 16 p v - 8 u v. \]

The product   \(h_+h_-\) \ is invariant under  the action
\(\sigma : u \mapsto -u, \ v \mapsto -v\).   
Hence, if a monomial \(u^iv^j\) occurs in  \(h_+h_-\), then \(i+j\) is even.  Then  \(h_+h_-\) is reduced to a polynomial  in \({\mathbb Z}[a, b, c, d, p]\).  Thus  \(h_+h_-\) \ is a monic polynomial in \(p\)  of degree 8 with coefficients in \( {\mathbb C}(Y)\),   denoted by \(H(p)\).

%%%%pro4.9
 \begin{pro} \label{pro4.8}
The morphism \ \(g_2\)  is a finite morphism.
 \end{pro}  
%%%%%%%%%%%%%%%%%%%%%%%%%%%proof
\begin{proof}
Since \ \(g_2 : X \to Y\) \ is a surjective morphism, we may regard \ \({\mathbb C}[Y]\) \ as a subring of \ \({\mathbb C}[X]\) .  Then it is enough to verify that \ \(p, q, r\) \ and \ \(s\) are integral over
\ \({\mathbb C}[Y]\).  Considering \ \(H(p)\), we know that \(p\) is integral over\ \({\mathbb C}[Y]\).  The function \(q\) is integral over \ \({\mathbb C}[Y][p]\) \  and \(r\) is integral over 
\ \({\mathbb C}[Y][p,q]\) \  and \(s\) is integral over 
\ \({\mathbb C}[Y][p,q,r]\).  Thus  \ \(p, q, r\) \ and \ \(s\) are integral over
\ \({\mathbb C}[Y]\) by the transitivity property of integral dependence (See \cite[Corollary 5.4]{AM}). 
 \end{proof}

Then by Proposition 2.9, we have the following.
%Proposition 4.10
 \begin{pro} \label{pro4.9}
Let \(g_2 : X_Q \to X_Q\)  be the morphism in (4.2).   Then \ deg \(g_2 = 8\). 
 \end{pro}  

We consider the ramification points of  \(g_2\).  We study multiple points of the inverse images \ \((p, q, r, s)\) \ of \ \((a, b, c, d)\) \ under \(g_2\).  We assume that the values of \ \(a, b, c, d\)  and \ \(u = \sqrt b\), \ \(v = \sqrt d\) \ are given.  Let \(p_+\) and \(p_-\) are zeros of \(h_+\) and \(h_{-}\), respectively.  Then for \(p_+\),  there is the unique solution 
\((p_+,  q_+,  r_+,  s_+)\) \ of (4.24) and for \(p_-\),  there is the unique solution  \((p_-,  q_-,  r_-,  s_-)\)  of (4.25).
If there are multiple points of the inverse image \ \((p, q, r, s)\),  then  \(H(p) \) must have a multiple zero.   Then there are three cases :(i) \(h_+\) has a multiple zero, \quad (ii) \(h_-\) has a multiple zero, \quad (iii) \(h_+\) and \(h_-\) \ have a common zero. 

Cases \ (i) \ and \ (ii).  We compute the discriminants of \(h_+\) and \(h_{-}\) \ with respect to \(p\) :
\[disc(h_+) = -16384 (a + v)^2 (-256 + 192 a - 48 a^2 + 4 a^3 + 128 u^2 \]
\[\qquad+ 80 a u^2 - a^2 u^2 - 16 u^4 - 288 u v - 36 a u v + 8 u^3 v + 108 v^2),\]

\[disc(h_-) =-16384 (a - v)^2 (-256 + 192 a - 48 a^2 + 4 a^3 + 128 u^2 \]
\[\qquad + 80 a u^2 - a^2 u^2 - 16 u^4 - 288 u v - 36 a u v + 8 u^3 v + 108 v^2).\]
Clearly, 
\[V(v+a) \cup V(v-a) = V(a^2-d),\] and
\[V(-256 + 192 a - 48 a^2 + 4 a^3 + 128 u^2 + 80 a u^2 -
    a^2 u^2 - 16 u^4 - 288 u v - 36 a u v + 8 u^3 v + 108 v^2)\]
\[= V(256 - 192 a + 48 a^2 - 4 a^3 - 128 b - 80 a b + a^2 b + 16 b^2 + 
 288 c + 36 a c - 8 b c - 108 d).\]

Then we have two algebraic sets \ \(V(a^2-d, bd-c^2)\) \ and 
\[ V(256 - 192 a + 48 a^2 - 4 a^3 - 128 b - 80 a b + a^2 b + 16 b^2 +  288 c + 36 a c - 8 b c - 108 d, bd-c^2).\]
These are  \ \(S_P\) \ and \ \(S_A\) in (4.21).
These will be proved to be irreducible and 
2-dimensional varieties in Sections 4.4 and 4.5.

Case (iii). \ 
We assume that  \ \((p_+, q_+, r_+, s_+) = (p_-, q_-, r_-, s_-)\).  Then from (4.24) and (4.25),  we have \ \(u = v = 0\).  The equality \ \(u = v = 0\) \  is equivalent to \ \(b = c = d = 0\).
If \ \(u = v = 0\), then \  \(h_+ = h_-\).  Then  \(H(p) \)  has four multiple zeros.  Then we have a variety \ \(L_1 = V(b, c, d)\).  This variety is the singular locus of  \(X_Q\). 

As in Section 3.1,  we  consider the ramification points of the varieties \ \(L_1, \ S_P\) \ and \ \(S_A\).  In 
Sections 4.4 and 4.6,  we will show that the singular loci of \ \(S_P\) \ and \ \(S_A\)  are 
\ \(V(a, c, d)\) \ and \ \(V(-12+3a-b, c^2-bd, -54d+108c+bc, 108b+b^2-54c)\), \ which are denoted by \(L_2\) and \(C_A\), respectively.  

By direct computations we have the following.
%%%%pro4.11
 \begin{pro} \label{pro4.11}
Under the above notations we have the followings :
\[\mbox{(1)} \ g_2^{-1}(L_1) = V(r^2 + s - 2ps, \ -8 + 8 p - p^2 - 4 r + 
 2 s, \ 2 - 2 p + q).\]

If \ \(y \in L_1\), \ then \  \(g_2^{-1}(y)\) \ consists of four points with multiplicity 2.
\[\mbox{(2)} \ g_2^{-1}(S_P \setminus L_2) = V(qs-r^2, \ 4q - 4 r +  s) \cup S_P.\qquad \qquad \qquad \qquad \qquad \qquad \qquad\]

If \ \(y \in S_P \setminus L_2\), \ then \  \(\sharp g_2^{-1}(y) = 6\).  Two points with multiplicity 2 lie on \ 
\(V(qs-r^2, \ 4q - 4 r +  s)\) \ and four points with multiplicity 1 lie on \(S_P\).
\[\mbox{(3)} \ g_2^{-1}(L_2) = V(2r-s, \ 4q-s, \ p^2- s) .\qquad \qquad \qquad \qquad \qquad \qquad \qquad\qquad\]

If \ \(y \in L_2\), \ then \  \(g_2^{-1}(y)\) \  consists of four points with multiplicity 2.  
\[\mbox{(4)} \ g_2^{-1}(S_A) = V(qs-r^2, \ p^2q - 4p r +  4s) \cup S_A.\qquad \qquad \qquad \qquad \qquad \qquad \qquad\]

If \ \(y \in S_A \setminus C_A\), \ then \  \(\sharp g_2^{-1}(y) = 6\).  Two points with multiplicity 2 lie on  
\ \(V(qs-r^2, \ p^2q - 4p r +  4s)\) \  and four points with multiplicity 1 lie on \(S_A\).

If \ \(y \in C_A\), \ then \  \(\sharp g_2^{-1}(y) = 4\).  Two points with multiplicity 3 lie on  
\ \(V(qs-r^2, \ p^2q - 4p r +  4s) \cap S_A\) \  and other two points with multiplicity 1 lie on \(S_A\).
\end{pro}
%%%%pro4.12
 \begin{pro} \label{pro4.10}
(1)  The branch locus of \ \(g_2\):  \(X_Q \to X_Q\) \ consists of the three affine algebraic sets  \ \(L_1, \ S_P\) \ and \ \(S_A\).\\
(2)  The branch locus of \ \(\tilde{F} : {\mathbb C}^3 \to X_Q\) \ consists of  \ \(L_1\) \ and \ \(S_P\).\\
\end{pro}
%%%%%%%%%%%%%%%%%%%%%%%%%%%proof
\begin{proof}
(1)  This assertion follows from the above arguments.\\
(2)  Let \(\tilde{F}(x, y, z) = (p, q, r, s)\).  Clearly \ \(\tilde{F}\) is finite.  By direct computations we have the following.  If \(q = 0\) and \(s= 0\), then \ \(\sharp \tilde{F}^{-1}(p, 0, 0, 0) \le 4,\) else if \ \(p^2 = s\), then  \ \(\sharp \tilde{F}^{-1}(p, q, r, s) \le 4,\) \ else  
\(\sharp \tilde{F}^{-1}(p, q, r, s) = 8.\)   Then the assertion (2) follows.
\end{proof}
In the following subsections, we consider the behavior of \(g_d\)  on the sets \ \(L_1, S_P, S_A\).   The line \(L_1\)  is preperiodic under \(g_d\) and the branch divisors \(S_P\) and \(S_A\) are invariant   under \(g_d\) for any \ \(d \in {\mathbb N}\).

%%%%%%%%Subsection4.3
\subsection{\rm{The line \(L_1\) and the morphism  \ \(g_d\)}}\hspace{1cm}

In this section we will show the  behavior of \(g_d\)  on  \ \(L_1\).  We prepare two lemmas. 
%%%%%%%%%%%lemma4.13
\begin{lemma} \label{lemma9}
Let \ \(T_d(w)\) \ be the \ \(d^{th}\) \ Chebyshev polynomial in one variable.  Then 
\[T_d(w) =  \left\{
\begin{array}{ll}
\displaystyle{p_{d/2}(w^2)}  & \quad \ \mbox{if}\ d \ \mbox{is even},\\
\displaystyle{ wq_{(d-1)/2}(w^2)}& \quad \ \mbox{if}\ d  \ \mbox{is odd},
\end{array} \right.\]
where \ \(p_n(x)\) \ and \ \(q_n(x)\) \ are monic polynomials of degree \(n\).
\end{lemma}
%%%%%%%%%%%%%%%%%%%%%%%%%%%proof
\begin{proof}
Chebyshev polynomials \ \(T_d(x)\) \ satisfies the recurrence (See \cite[6.2]{S1}) : \\
\(T_0(w) = 2, \ T_1(w) = w\), \ and \ \(T_{d+2}(w) = wT_{d+1}(w) - T_d(w) \) \ for \ \(d \geq 0\).  Then this lemma follows by induction.
 \end{proof}

By (4.2), we see that the morphism \(g_2\) maps \ \(L_1 = \{(t, 0, 0, 0)\}\) \ to \ \(\{(t^2, \ (2t-2)^2, \ (2t-2)t^2, \ t^4)\}\).

Then we consider a polynomial parametric representation given by
\[p = t^2, \ q = (2t-2)^2, \ r = (2t-2)t^2,  \ s = t^4.\]
Let \ \(G_1(t) = (t^2, \ (2t-2)^2,\  (2t-2)t^2, \ t^4)\).  Then by the implicitization algorithm for polynomial parametrization \cite[Chapter 3, \S3, Theorem 1]{CLS},  we compute the Gr\(\ddot{o}\)bner basis for\ \(I = (t^2 - p, \ (2 - 2 t)^2 - q, \ (2 t - 2) t^2 - r,\  s - p^2),\) \ using graded reverse lex order with \ \(t > p > q > r > s\).  The  Gr\(\ddot{o}\)bner basis for \(I\) is given by
\[\{-4 - 4 p + q + 8 t, \ r^2 - q s, -16 p - 12 r + q r + 20 s - 4 p s + q s, \  4 p r + 4 s - 4 p s + q s, \]
\[16 - 8 q + q^2 + 32 r - 16 s, \ 4 p + p q + 4 r - 4 s, p^2 - s\}.\] 
Set
%%%%%%%%%%%%%%%%%%%%%%%%equation 4.25
\begin{equation}
\begin{split}
C_1 = V( r^2 - q s, -16 p - 12 r + q r + 20 s - 4 p s + q s, \\  
4 p r + 4 s - 4 p s + q s, \ 16 - 8 q + q^2 + 32 r - 16 s, \ 4 p + p q + 4 r - 4 s, p^2 - s).
\end{split}
\end{equation}
By \cite[Chapter 3, \S3, p.131]{CLS}, we know that \ \(G_1({\mathbb C}) \) \ covers all points of \(C_1\).  By Proposition 5 in \cite[Chapter 4, \S5]{CLS}, we have that \(C_1\)  is irreducible.  Clearly, \ \(dim \ C_1 = 1\).  Then \(C_1\) is an affine algebraic curve.
%%%%%%%%%%%lemma4.14
\begin{lemma} \label{lemma10}
Let \(C_1\) be the algebraic set defined in (4.26).  Then \(C_1\) is an affine algebraic curve.
\end{lemma}
%%%%pro4.15
 \begin{pro} \label{pro4.14}
Under the above notations,  we have  :\\
(1)  If \  \(d\) is odd,  then \ \(g_d(L_1) = L_1.\)  In particular, \ \(g_d(t, 0, 0, 0) = (T_d(\sqrt{t})^2, 0, 0, 0).\)\\
(2)  \(g_2(L_1) = C_1\).\\
(3)  \(g_d(C_1) = C_1\), \ for any \ \(d \geq 1\).  In particular, 
\[g_d(G_1(w+2)) = G_1(T_d(w) + 2), \ \mbox{for any} \ w \in {\mathbb C} .\]
 \end{pro}
%%%%%%%%%%%%%%%%%%%%%%%%%%%proof
\begin{proof}
(1) We note that by Lemma \ref{lemma9},  \ \(T_d(\sqrt{t})^2\) \ is a polynomial in \(t\).  Set 
\[g_d (p, q, r, s) = (P_d (p, q, r, s), \ Q_d (p, q, r, s), \ R_d (p, q, r, s), \ S_d (p, q, r, s)).\]
We must show that 
%%%%%%%%%%%%%%%%%%%%%%%%equation 4.27
\begin{equation}
\begin{split}
P_d(t, 0, 0, 0) = T_d(\sqrt{t})^2, \hspace{3cm}\\
Q_d(t, 0, 0, 0) = R_d(t, 0, 0, 0) = S_d(t, 0, 0, 0) = 0.
\end{split}
\end{equation}
Since \ \(T_d(\sqrt{t})^2, \ P_d(t), \ Q_d(t), \ R_d(t)\) \ and \  \(S_d(t)\) \ are polynomials, it is enough to prove (4.27) for infinitely many values \ \(t\).  Then it suffices to show (4.27) for \ \(t = 4\cos^2\theta, \ 0 \le \theta \le 2\pi\). See \cite[Chapter V, \S4]{Ls}.   We consider  the diagram (4.15)  in the case \ \({\bar z}_3 = z_1\)  and \(z_2 \in {\mathbb R}\).
 
We consider a correspondence between the top part and the bottom part in the diagram.  The operators \(R\) and \(C\) defined in Section 2.1 act on the space 
\[U_4 = \{\{t_1, t_2, t_3, t_4\} : t_j \in {\mathbb C}^*, \ \mid t_j \mid =1, \ t_1 t_2 t_3 t_4 = 1\}.\]
Then we consider the quotient space \ \(U_4/<R, C>\).  We claim that
%%%%%%%%%%%%%%%%%%%%%%%%equation 4.28
\begin{equation}
\begin{split}
(4\cos^2\theta, 0, 0, 0) \in X_Q \ \mbox{ if and only if}\hspace{3cm}\\
[\{\xi,\  -\xi, \ \xi e^{i \theta}, \ \xi e^{-i \theta}\}] \in U_4/<R, C>, \ \mbox{where} \  \xi = e^{2\pi i/8}, \ 0 \le \theta \le 2\pi.
\end{split}
\end{equation}
Proof of the "if " part.    Recall that \ \(R(t_j) = \xi^2t_j, j = 1, 2, 3, 4.\)  Then 
\[\{\xi, \ R(\xi), \ R^2(\xi), \ R^3(\xi)\} = \{\xi, \ \xi^3, \ \xi^5, \ \xi^7\} = \{(\pm 1 \pm i)/\sqrt{2}\}.\]
Since \ \(C(\xi) = \xi^7,\) \  we may consider the sets \ \(\{\xi^k,\  -\xi^k, \ \xi^k e^{i \theta}, \ \xi^k e^{-i \theta}\}\) \ for \ \(k = 1, 3, 5, 7\).
Hence \ \(z_1 = \xi^k\cdot 2\cos\theta, \ z_2 = 0,\)\ and so \ \(p = 4\cos^2\theta, \ q = 0, \ r = 0, \ s = 0\).\\
Proof of the "only if  "  part.  By the assumption \ \(q = r = s = 0\), we have \ \(x = \pm y,\) \ and \ \(z = 0.\)  Hence \ \(x = \pm\sqrt{2}\cos\theta\) \ and \ \(y = \pm\sqrt{2}\cos\theta\). \ Then \({t_j}^\prime s\) \  satisfy one of the equations
\[t^4 - (2\cos \theta)\xi^kt^3 - (2\cos \theta)\bar{\xi}^{k}t + 1 = 0, \  k = 1,3, 5, 7.\]
Then the roots satisfy the property that
\[\{t_1, t_2, t_3, t_4\}  \in \{\{\xi^k,\  -\xi^k, \ \xi^k e^{i \theta}, \ \xi^k e^{-i \theta}\} : k = 1,3, 5, 7\}.\]
This completes the proof of the claim (4.28).

 We consider the morphism \ \(g_d\) \ with \(d\) odd.  Hence by (4.15) and (4.28), \ \([\{\xi,\  -\xi, \ \xi e^{i \theta}, \ \xi e^{-i \theta}\}]\) \ is mapped to \ \([\{\xi^d,\  -\xi^d, \ \xi^d e^{i d\theta}, \ \xi^d e^{-i d\theta}\}]\).  Note that if \(\xi\) is a primitive eighth root of unity, then \  \(\xi^d\) is also a primitive eighth root of unity.
Thus 
\[g_d(4\cos^2\theta, 0, 0, 0) = (4(\cos d\theta)^2, 0, 0, 0) .\]
Setting \ \(t = (2\cos\theta)^2\),  we have (4.27), for all \ \(0 \le \theta \le 2\pi\).\\
(2) Since \ \(g_2(t, 0, 0, 0) = G_1(t),\) the assertion (2) is  proved by the argument before Lemma \ref{lemma10}.\\
(3) Clearly, \ \(t = 2 + 2\cos 2\theta\).  Then the parametric representation  \(G_1\)  implies
\[p  = (2 + 2\cos 2\theta)^2, \ q  = (2 + 4\cos 2\theta)^2, \ r  = (2 + 4\cos 2\theta)( 2 + 2\cos 2\theta)^2, \ s = (2 + 2\cos 2\theta)^4.\]
To the element \ \((p, q, r, s)\) we consider the corresponding set \ \(\{t_1, t_2, t_3, t_4\}\).  Then we claim the following.
%%%%%%%%%%%%%%%%%%%%%%%%equation 4.29
\begin{equation}
\begin{split}
( (2 + 2\cos 2\theta)^2, \  (2 + 4\cos 2\theta)^2, \  (2 + 4\cos 2\theta)( 2 + 2\cos 2\theta)^2, \  (2 + 2\cos 2\theta)^4) \in X_Q \\
 \mbox{ if and only if} \quad [\{1,\  1, \  e^{i 2\theta}, \  e^{-i 2\theta}\}]  \in  U_4/<R, C>.
\end{split}
\end{equation}
We can prove (4.29) by the similar argument given in the proof of (4.28).  Then  we can prove the assertion (3) similarly.
%\end{proof}
\end{proof}
The line \(L_1\) is the singular locus of \(X_Q\).  For any invariant  subvariety \(V\)  under \(g_d\) of \(X_Q\),  we define the set \ \(K(g_d \mid_ V)\) \ of bounded orbits of \(V\) under \(g_d\) by
\[K(g_d\mid_{V}) = \{w \in V : \{g_d^k(w) : k = 1, 2, \cdots \} \  \mbox{is bounded in }\ {\mathbb C}^4\}.\]
%%%%%%%%%%%%%%%%%%%%%%%%%%%%%%%%%%%%%%Corollary 4.16
\begin{cor} \label{cor4}
(1) If \(d\) \ is odd and greater than 2, then 
\[K(g_d\mid_{L_1})   =  \{ ((2\cos \theta)^2, \ 0, 0, 0) : \ 0 \le \theta \le 2\pi\}.\qquad \qquad \qquad \qquad \qquad \qquad\]
\[(2) K(g_d\mid_{C_1})   =  \{G_1 (2\cos \theta + 2) : 0 \le \theta \le 2\pi\}, \ \mbox{for} \ d \ge 2.\qquad \qquad \qquad \qquad \qquad \qquad\]
\((3) \ C_1 \subset S_A \cap S_P\).
\end{cor}
%%%%%%%%%%%%%%%%%%%%%%%%%%%proof
\begin{proof}
(1) and (2) \  The proof is similar to the proof of Theorem \ref{theorem4.8} (2).\\
(3) \ If we put \ \(p = t^2, \ q = (2-2t)^2, \ r = (2t-2)t^2, s = t^4\) \ in the polynomial \ \(A_h\) in (4.21), 
then the polynomial vanishes.  Note that \ \(S_A = V(A_h, \ qs-r^2)\).  Hence \ \(C_1 \subset S_A\).  Then
from (4.26), we have this assertion.
\end{proof}

We will show that two surfaces \(S_A\) and \(S_P\) are invariant under \(g_d\). 

%%%%%%%%Subsection4.4
\subsection{\rm{The dynamics of \(g_d\) on the surface \(S_P\)}}\hspace{0cm}  

We will show that \(S_P\) is an affine algebraic surface and is invariant under \(g_d\).  Recall that \ \(S_P = V(p^2-s, qs-r^2)\). 
%%%%pro4.17
 \begin{pro} \label{pro4.16}
The algebraic set \(S_P\) is an affine algebraic surface and is birationally equivalent to \({\mathbb C}^2\).
 \end{pro}
%%%%%%%%%%%%%%%%%%%%%%%%%%%proof
\begin{proof}
We consider a polynomial parametric representation of \(S_P\) given by \\
 \(\phi(u, v) = (u, v^2, uv, u^2)\).  By the similar arguments given in the argument above Lemma  \ref{lemma10}, we can verify that \ \(\phi({\mathbb C}^2)\) \ fills up  all of \(S_P\).  Then \(S_P\) is irreducible.  Clearly, \ dim \(S_P = 2\).  We consider a rational map from \(S_P\) to \ \({\mathbb C}^2\)  given by \ \(\psi (p, q, r, s) = (p, r/p)\).  Then \ \(\psi \circ \phi = id_{{\mathbb C}^2}\) \ and \ \(\phi \circ \psi = id_{S_P}\).
\end{proof}

Next we consider the map  \(g_d\) on \(S_P\).
We consider a mapping \(\rho\) from \({\mathbb C}^2\) \ to\  \({\mathbb C}^2\) defined  by \  \(\rho(w, t) = ((w+t)^2, (wt+2))\). 
Then we consider a mapping \  \(G_P : {\mathbb C}^2 \to {\mathbb C}^4\)  defined by \  \(G_P = \phi \circ \rho\). 
%%%%%%%%%%%%%%%%%%%%%%%%theorem 4.18
\begin{theorem} \label{theorem17} 
Let \(g_d\) be a morphism in Theorem 4.8.\\ 
(1) \(g_d(G_P(w, t)) = G_P(T_d(w), T_d(t)),\)
\ for any \ \(w, \ t \in {\mathbb C}\) \ and \ \(d \in {\mathbb N}\).\\
(2) \(g_d(S_P) = S_P\) ,  \ for any \ \(d \in {\mathbb N}\).\\
(3)  \(K(g_d \mid _{S_P}) = \{G_P(2\cos \alpha,  2\cos \beta)  : \ 0 \le \alpha, \beta \le 2\pi\}, \ \mbox{for} \ d \ge 2.\)
\end{theorem}
%%%%%%%%%%%%%%%%%%%%%%%%%%%proof
\begin{proof}
(1) We consider the condition \ \(p^2 = s\) \ in \ \(V(p^2-s, \ qs-r^2)\).  Since \ \(p = x^2 + y^2\) \ and \  \(s = (x^2 - y^2)^2,\) \ then  the condition \ \(p^2 = s\) \ is equivalent to the condition \ \(xy = 0.\)  We consider this equation \ \(xy = 0\).
In (4.16), we set 
%%%%%equation%%%%%%%% 4.15%%%%%%%%%%%%%%%%
\begin{equation*}
t_1 = e^{i(\alpha + \gamma)}, \ t_2 = e^{i(\alpha - \gamma)}, \ t_3 = e^{i\beta}, \ t_4 = e^{-i(2\alpha + \beta)} .
 \end{equation*}
Then by (4.17), we have 
%%%%%equation 4.16%%%%%%%%%%%%%%%%%%%%%
\begin{equation*}
\begin{split}
x =  \cos(\alpha + \gamma) +  \cos(\alpha -\gamma) +  \cos \beta +  \cos(2\alpha + \beta), \qquad\qquad\\
y =  \sin(\alpha + \gamma) +  \sin(\alpha -\gamma) +  \sin \beta -  \sin(2\alpha + \beta), \qquad\qquad \\
= 4\sin\alpha \sin(\frac12(\alpha + \beta + \gamma))\sin(\frac 12(\alpha + \beta-\gamma)), \qquad\qquad \\
z =  2\cos 2\alpha  +  (2\cos(\alpha + \beta))(2\cos \gamma). \qquad\qquad\qquad\qquad\qquad
\end{split}
\end{equation*}
We assume that \ \(\alpha = 0\).  Then \ \(y = 0\) \ and so \ \(xy = 0\).  If \ \(\alpha = 0\),  then \ \( x = (2\cos\beta + 2\cos\gamma),  \ y = 0\) \ and \ \(z = (2\cos\beta)(2\cos\gamma) +2\).
Hence putting \ \(\alpha = 0\) \ in (4.18), we have \ \(p, q, r\) and \(s\),  and so  we put  \(u = 2\) \ in (4.19).  Then from the parametrization \ \(G_X\) we get \(G_P\).

To prove the assertion (1) we use an argument similar to that used in the proof of Theorem 3.3(1).  \\
(2)  Clearly \(\rho\) is a surjection.  Set \  \(\rho(w, t) = (u, v) \)\ and consider a polynomial parametric representation \  \(\phi (u, v) \) \ of \(S_P\).  Then we have the assertion (2).\\
(3)  The proof  is similar to the proof of Theorem \ref{theorem4.8} (2).\\
\end{proof}
We remark that if we assume that \ \(\alpha + \beta = \pm\gamma\),  then we have the same result.
 
By computing the rank of the Jacobian matrix \ \(\partial(p^2-s, qs-r^2)/\partial(p, q, r, s)\), \ we can prove that the singular locus of \(S_P\) is \ \(V(p, r, s)\).  We consider the singular locus  \(L_2 = \{(0, q, 0, 0) : q \in {\mathbb C}\}\) \ of  \(S_P\).   The inverse images of \(L_2\) under \(g_2\) is described in Proposition \ref{pro4.11}(3).  We have the following proposition that is similar to  Proposition \ref{pro4.14}. 
%%%%pro4.19
 \begin{pro} \label{pro4.18}
Let \(g_d\) be a morphism in Theorem \ref{theorem4.8}.\\
(1)  If \(d\) is odd, \ \(g_d(L_2) = L_2\).  In particular,   \(g_d(0, t^2, 0, 0) = (0, T_d(t)^2, 0, 0).\)\\
(2)  \(g_2(L_2) = C_2\),  where 
\[C_2 = V(-64 + 16 q - 8 r + s, -8 p + 4 r - s,  16 r^2 - 64 s - 8 r s + s^2).\]
\qquad \(C_2\) is an affine algebraic curve.  The curve  \(C_2\) has a polynomial parametric representation \ \(G_2(T) = (4T, \ (T+2)^2, \ 4T(T+2), \ 16T^2).\)\\
(3)   \(g_d(C_2) = C_2\),  for any \ \(d \ge 1\).  In particular, 
\[g_d(G_2(t^2)) =  G_2(T_d(t)^2), \ \mbox{for any} \ t \in {\mathbb C}.\]
\end{pro}
%%%%%%%%%%%%%%%%%%%%%%%%%%%proof
\begin{proof}
(1)  The proof is similar to the argument given in Proposition \ref{pro4.14}.  Then it suffices to prove the following claim : \\
%%%%%%%%%%%%%%%%%%%%%%%%equation 4.30
\begin{equation}
\begin{split}
(0,  4\cos^2\theta,  0, 0) \in X_Q \ \mbox{ if and only if}\\
[\{e^{i \theta/2},\  -e^{i \theta/2}, \ e^{-i \theta/2},\  -e^{-i \theta/2}\}] \in U_4/<R, C> \ \mbox{where} \ 0 \le \theta \le 2\pi.
\end{split}
\end{equation}
Proof of the "if " part.    Under the condition, we have \ \(z_1 = 0, \ z_2 = \pm2\cos\theta\) \ and so \ \(q = 4\cos^2\theta\).\\
Proof of the "only if "  part.  Then \ \([\{e^{i \theta/2},\  -e^{i \theta/2}, \ e^{-i \theta/2},\  -e^{-i \theta/2}\}]\) \ is the set of roots of the equation : 
\(t^4 \pm 2\cos \theta t^2 + 1 = 0.\)  This completes the proof of the claim (4.30).\\
(2)  By (4.2), we have \ \(g_2(0, t^2, 0, 0) = (4t^2,\  (t^2+2)^2, \ 4t^2(t^2+2), \ 16t^4) \).  We consider the polynomial parametric representation
\[G_2(T) = (4T, \ (T+2)^2, \ 4T(T+2), \ 16T^2).\]
Then we see that \ \(G_2({\mathbb C})\) fills up all of \(C_2\).  Hence \(C_2\)  is an affine algebraic curve.\\
(3)  We consider the case \ \(t = 2\cos\theta\).  We see from (2) and (4.30) that
\[ (4(2\cos\theta)^2, \ ((2\cos\theta)^2+2)^2, \ 4(2\cos\theta)^2((2\cos\theta)^2+2), \ 16(2\cos\theta)^4) \in X_Q.\]
if and only if 
\[[\{e^{i \theta},\  e^{i \theta}, \ e^{-i \theta},\  e^{-i \theta}\}] \in U_4/<R, C>.\]
Then the assertion (3) follows.
\end{proof}
%%%%%%%%%%%%%%%%%%%%%%%%%%%%%%%%%%%%%%Corollary 4.20
\begin{cor} \label{cor5}\hspace{0.1cm} 
 (1) If \(d\) is odd and greater than 2, then\\ 
\quad \ \(K(g_d\mid_{L_2})   =  \{ (0, (2\cos \theta)^2,  0, 0) :  0 \le \theta \le 2\pi\}.\)\\
\((2) K(g_d\mid_{C_2})   =  \{G_2( (2\cos \theta )^2) : 0 \le \theta \le 2\pi\}, \ \mbox{for} \ d \ge 2.\)\\
\((3) \ C_2 \subset S_A \cap S_P\).
\end{cor}
%%%%%%%%%%%%%%%%%%%%%%%%%%%proof
\begin{proof}
(1) and (2) are obvious.\\
(3)  We use the polynomial parametric representation \(G_2(T)\).  Then \ \( C_2 \subset S_P\).  If we put \ \(p = 4T, \ q = (T+2)^2, \ r = 4T(T+2), \) \ and \ \(s= 16T^2\) \ in the polynomial \ \(A_h\) 
in (4.21) of \(S_A\), then the polynomial vanishes.
\end{proof}
Then combining Corollary \ref{cor4} with Corollary \ref{cor5}, we have \ \(C_1 \cup C_2 \subset S_A \cap S_P\).  In the next subsection we will show that \  
 \[C_1 \cup C_2 = S_A \cap S_P \ \mbox{and } \ 
K(g_d\mid_{C_1}) \cup K(g_d\mid_{C_2}) = K(g_d\mid_{S_A}) \cap K(g_d\mid_{S_P}).\]
%%%%%%%%Subsection4.5
\subsection{\rm{The dynamics of \(g_d\) on the surface \(S_A\)}}\hspace{0cm}  

We begin to show that \(S_A\) is an affine algebraic surface and is invariant under \(g_d\).  The surface \ \(S_A\)  is related to the astroidalhedron \(\mathcal{A}\)  that is studied in \cite[p.205]{U3}.  The astroidalhedron \(\mathcal{A}\) has connection with catastrophe theory (See \cite{PS1, PS2}).  It is a ruled surface in \( {\mathbb R}^3\).  In particular, it is the tangent developable of the astroid in space  (See \cite[Figures 3 and 4]{U3}).   The astroidalhedron \(\mathcal{A}\) lies in \(R_3\)  and it is the boundary of \ \(K(P_{A_3}^d)\) .
We will show that \(\mathcal{A}\)  is invariant under the action of \(D_4\) and that the relative orbit variety \ \(\mathcal{A}/D_4\)  is related to the surface \(S_A\).

The implicit representation of \(\mathcal{A}\) is given in \cite[p.205]{U3} by the following ;
\[ 256 -27(z_1^4+{\bar z}_1^4) + (z_1^2+{\bar z}_1^2)(144z_2-4z_2^3 + 18z_1{\bar z}_1z_2)\]
\[-80z_1{\bar z}_1z_2^2+z_1^2{\bar z}_1^2z_2^2-192z_1{\bar z}_1-4z_1^3{\bar z}_1^3-6z_1^2{\bar z}_1^2 -128z_2^2 + 16z_2^4 = 0.\]
Setting \ \(z_1 = x + iy\) \ and \ \(z_2 = z\), we have the equation ;
\[A(x, y, z) = 256 - 192 x^2 - 60 x^4 - 4 x^6 - 192 y^2 + 312 x^2 y^2 - 12 x^4 y^2 - 
 60 y^4 - 12 x^2 y^4 - 4 y^6 + 288 x^2 z \]
\[+ 36 x^4 z - 288 y^2 z - 
 36 y^4 z - 128 z^2 - 80 x^2 z^2 + x^4 z^2 - 80 y^2 z^2 + 
 2 x^2 y^2 z^2 + y^4 z^2 - 8 x^2 z^3 + 8 y^2 z^3 + 16 z^4 = 0.\]
We can write \ \(A(x, y, z) \) \ in the terms \ \(p, q, r, s \) \ in (4.1).  Then we have
\[A_h = 256 - 192 p + 48 p^2 - 4 p^3 - 128 q - 80 p q + p^2 q + 16 q^2 + 
 288 r + 36 p r - 8 q r - 108 s.\]
That is,
\[A_h(p({\bf x}), \ q({\bf x}), \ r({\bf x}), \ s({\bf x})) = A({\bf x}), \ \mbox{where} \ {\bf x} = (x, y, z).\]
Then \(\mathcal{A}\) is invariant under \(D_4\).  Hence we can construct the relative orbit variety \ 
  \(\mathcal{A}/D_4\) \ by Algorithm 2.6.2 in \cite[\S 2.6]{St}.  By computing a Gr\(\ddot{o}\)bner basis for the ideal
\[(A(x, y, z), \ x^2+y^2-p, \ z^2-q,\ z(x^2-y^2)-r, \ (x^2-y^2)^2-s),\]
we know that the ideal of \ \(\mathcal{A}/D_4\) is \ \((A_h(p, q, r, s), \ r^2-qs)\).  We view the ideal \ \((A_h(p, q, r, s), \ r^2-qs)\) \ as an ideal in 
\({\mathbb C}[(p, q, r, s].\)  Then we have the affine algebraic surface \ \(S_A = V(A_h(p, q, r, s), \ r^2-qs)\).

Next, using parametric representation of  \(\mathcal{A}\), we construct a polynomial parametric representation of the affine algebraic surface  \(S_A\).  The astroidalhedron \(\mathcal{A}\) is defined in \(R_3\) using the coordinates \ \((t_1, t_2, t_3, t_4)\) \ satisfying 
\[t_1 = e^{i\alpha}, \ t_2 = e^{i\alpha}, \ t_3 = e^{i\beta}, \ t_4 = e^{-i(2\alpha+\beta)},  \ (0 \le \alpha, \beta < 2\pi),\]
 (See \cite[pp. 204-205]{U3}).   Note that this is the case \ \(\gamma = 0\) \ in (4.16). 
Then  putting \ \(w = 2\) \ in (4.19), we have 
%%%%%%%%%%%%%%%%%%%%%%%equation 4.31
\begin{equation}
\begin{split}
p = 4 + 2uv+ v^2, \quad  q = (u+2v)^2, \\
r = (u + 2v)(2u + 4v +uv^2/2), \quad s = (2u + 4v +uv^2/2)^2.
\end{split}
\end{equation}
We define a mapping \(G_A\) from \({\mathbb C}^2\) to \({\mathbb C}^4\) by 
\[G_A(u, v) = (4 + 2uv+ v^2, \  (u+2v)^2, \  (u + 2v)(2u + 4v +uv^2/2), \ (2u + 4v +uv^2/2)^2).\]
Then we can show that \ \(G_A(u, v)\) \ is the polynomial parametrization of \ \(S_A\).  Indeed.  Using lex order with \ \(u > v> p> q> r> s,\) \ we compute the  Gr\(\ddot{o}\)bner basis for the ideal
\[(p - (4 + 2uv+ v^2), \ q -  (u+2v)^2, \ r - (u + 2v)(2u + 4v +uv^2/2), \ s - (2u + 4v +uv^2/2)^2).\] 

The Gr\(\ddot{o}\)bner basis has 12 elements altogether.  The following two elements contain only \ \(p, q, r, s\) \ terms ; 
\[-r^2 + q s, -256 + 192 p - 48 p^2 + 4 p^3 + 128 q + 80 p q - p^2 q - 
  16 q^2 - 288 r - 36 p r + 8 q r + 108 s.\]
Then \ \(G_A(u, v)\) \ is a polynomial parametrization of \ 
\(S_A = V(-r^2 + q s, A_h)\).
The Gr\(\ddot{o}\)bner basis contains also the following two polynomials ;  \(-16 + 8 p -  p^2 - 24 q + 12 r - 2 q v^2 + 9 v^4, \  -8 + 2 p - q + u^2 + 2 v^2.\)
Then we know from \cite[Chapter 3, \S3, p.131]{CLS} that \ \(G_A({\mathbb C}^2)\) \ fills up all of \ \(V(A_h, r^2-qs)\).  By Proposition 5 in \cite[Chapter 4, \S5]{CLS}, \ \(S_A\) is irreducible.  Also we have \ dim\(S_A = 2\) \ by the argument after Theorem 8 in \cite[Chapter 9, \S3]{CLS}. 

We remark that if we assume \ \(\alpha + \beta = 0\) \ in (4.16) in stead of \ \(\gamma = 0\), \  then we  get also the same variety \ \(S_A\).
%%%%pro4.21
 \begin{pro} \label{pro4.20}
The algebraic set \(S_A\) is an affine algebraic surface. 
 \end{pro}
We call \ \(S_A\) an astroid surface.

Next we consider the morphisms \(g_d\) on the  astroid surface  \(S_A\).  \\

%%%%%%%%%%%%%%%%%%%%%%%%theorem 4.22
\begin{theorem} \label{theorem8} 
Let \(g_d\) be a morphism in Theorem 4.8.\\ 
(1) \(g_d(G_A(u, v)) = G_A(T_d(u), T_d(v)). \)\\
(2) \(g_d(S_A) = S_A\) .  
\end{theorem}
%%%%%%%%%%%%%%%%%%%%%%%%%%%proof
\begin{proof}
(1)  In this case we consider the case \ \(\gamma = 0\) \ in (4.16). 
Hence by the similar argument given in Theorems \ref{theorem2b}(1)  and \ref{theorem17}, we have the assertion (1).
\\
(2)  This assertion follows from the fact that the polynomial parametrization  \ \(G_A({\mathbb C}^2)\) \ covers all points of \(S_A\).
\end{proof}
%%%%pro4.23
\begin{pro} \label{pro4.22}\hspace{10cm}\\ 
\((1) K(g_d\mid_{S_A})   =  \{G_A (2\cos 2\alpha, 2\cos(\alpha+\beta)) : 0 \le \alpha, \beta < 2\pi\}, \ \mbox{for} \ d \ge 2. \)\\
\((2) \  S_A \cap S_P = C_1 \cup C_2 \).\\
\((3) \ K(g_d\mid_{S_A})  \cap  K(g_d\mid_{S_P}) = K(g_d\mid_{C_1})  \cup  K(g_d\mid_{C_2}), \ \mbox{for} \ d \ge 2.\)\\
\((4) \ K(g_d\mid_{S_A})  \cup  K(g_d\mid_{S_P}) \)  \ is a real 2-dimensional surface.\\
The surface \ \(K(g_d\mid_{S_A})  \cup  K(g_d\mid_{S_P})\) \ is  invariant under \(g_d\) for any \ \(d \in {\mathbb N}\).
 \end{pro}
%%%%%%%%%%%%%%%%%%%%%%%%%%%proof
\begin{proof}
(1)  The proof is similar to the proof of Theorem \ref{theorem4.8} (2).\\
(2)  By Corollaries \ref{cor4} and \ref{cor5},  We have \ \(C_1 \cup C_2 \subset S_A \cap S_P\).  To prove the opposite inclusion, we consider the parametrization (4.31) of \(S_A\).  Suppose 
\[(p(u, v), \ q(u, v), \ r(u, v), \ s(u, v)) \in  S_A \cap S_P.\]
Hence if \ \(  p^2(u, v) - s(u, v) = 0,\) then
\[0 = p^2(u, v) - s(u, v) = \frac 14(2-u)(2+u)(2-v)^2(2+v)^2.\]

Then the two cases occur : \ \(u = \pm 2\) \ and \ \(v = \pm 2\). \\
Case 1: \ \(u = \pm 2\).  Then  
\[G_A( \pm 2, \ v) = ((2 \pm v)^2, \  4(1 \pm v)^2, \  2(1 \pm v) (2 \pm v)^2, \  (2 \pm v)^4).\] 
By the parametrization \ \(G_1\),   we have
\[C_1 = \{(t^2, \ (2t-2)^2, \ (2t-2)t^2, \ t^4) : t \in {\mathbb C}\}.\]
Then  \ \(G_A(\pm 2, v) \in C_1.\)\\
\\
Case 2: \ \(v = \pm 2\).  Then  
\[G_A(u, \ \pm 2) = (4(2 \pm u), \  (4 \pm u)^2, \  4(2 \pm u) (4 \pm u), \ 16(2 \pm u)^2).\] 
By the parametrization \(G_2\)  in Proposition \ref{pro4.18}(2), we have
\[C_2 = \{(4T, \ (T+2)^2, \ 4T(T+2), \ 16T^2) : T \in {\mathbb C}\}.\]
Then  \ \(G_A(u, \ \pm 2) \in C_2.\)\\
(3)  From (2), we have \  \(K(g_d\mid_{C_1})  \cup  K(g_d\mid_{C_2}) \subset K(g_d\mid_{S_A}) \cap K(g_d\mid_{S_P}) \).  To prove the opposite inclusion we note that the points \ \(G_A(\pm2, 2\cos(\alpha + \beta))\) \ lie on \(K(g_d\mid_{C_1})\) \ by Corollary \ref{cor4}(2).
  Also the points \ \(G_A(2\cos2\alpha, \pm2)\) \ lie on \ \(K(g_d\mid_ {C_2})\).\\
(4)  \( K(g_d\mid_{S_A}) \) \ is described in (1) and \ \( K(g_d\mid_{S_P}) \) \ is described in Theorem \ref{theorem17}(3).  These are connected by \  \(K(g_d\mid_{C_1})  \cup  K(g_d\mid_{C_2})\).  Clearly
\( K(g_d\mid_{S_A}) \) \ is invariant under  \(g_d\) and \ \( K(g_d\mid_{S_P}) \) \ is invariant under  \(g_d\).
\end{proof}

We have the inequalities (3.15) for \ \(K(g_d\mid_{{\mathbb C}^2})\) \ in \S3.3.   Similarly we have the following ; for \ \(0 \le \alpha, \beta, \gamma \le 2\pi\),
%%%%%%%%%%%%%%%%%%%%%%%%equation 4.32
\begin{equation}
%\begin{split}
p^2 - s = 64(\cos\alpha)^2(\sin\alpha)^2(\cos (\alpha+\beta) - \cos \gamma)^2(\cos (\alpha+\beta) + \cos \gamma)^2 \ge 0,  \\
\end{equation}
\[256 - 192 p + 48 p^2 - 4 p^3 - 128 q - 80 pq + p^2 q + 16q^2 + 288r + 36 pr - 8 qr - 108 s  = \]
\[ 64 (1 + \cos 4\alpha + \cos2(\alpha + \beta) - 4\cos2\alpha \cos(\alpha + \beta) \cos\gamma +\cos2\gamma)^2 (\sin(\alpha + \beta))^2(\sin\gamma)^2 \ge 0.\]
Then we conjecture that \ \(K(g_d\mid_{S_A})  \cup  K(g_d\mid_{S_P})\) \ is the boundary of \ \(K(g_d\mid_{X_Q})\) \ in \({\mathbb R}^4\).

We have seen that the  astroid surface  \(S_A\) has relation to the  astroidalhedron \(\mathcal{A}\).  We know that the astroidalhedron \(\mathcal{A}\) \  is a ruled surface in the sense of differential geometry.  It is natural to ask whether the affine algebraic surface \(S_A\) is related to a ruled surface in the sense of algebraic geometry.   The projective closure of \(S_A\) is a birationally ruled surface.  Indeed. We will show that  \(S_A\) is birationally equivalent to an affine quadric cone \ \(S_ C\) \ in \({\mathbb C}^3\).   It is known that the  blowing-up of the projective closure of  \(S_C\) at the vertex \ \(P_0 = (1 : 0 : 0 : 0)\) \ is isomorphic to the Hirzebruch surface \ \({\mathbb P}({\mathcal O} \oplus {\mathcal O}(2))\) \ which is a geometrically ruled surface. (See \cite{H}.)   We assume that the coordinate ring \ \({\mathbb C}[S_C]\)  \ is  \  \({\mathbb C}[a, b, c]/(ac-b^2)\).  
 The cone \ \(V(ac-b^2)\) has a polynomial parametric representation given by  \ \(a = u^2, \ b = uv\) \ and \ \ \(c = v^2\). 

%%%%pro4.24
 \begin{pro} \label{pro4.23}
Under the above notation, the astroid surface \(S_A\)  is  birationally equivalent to the affine quadric cone \(S_C\).
 \end{pro}
%%%%%%%%%%%%%%%%%%%%%%%%%%%proof
\begin{proof}
We construct rational mappings \ \(\phi : S_A -\to S_C\) \ and \ \(\psi : S_C -\to S_A\).  We define the mapping \(\psi\) by
\[\psi(a, b, c) =  (4 + 2 b + c, \ a + 4 b + 4 c, \ 2 a + 8 b + 8 c + \frac12 a c + b c, \ 4 a + 16 b + 16 c + 2 a c + 4 b c + \frac14 a c^2).\]
And we define the mapping  \(\phi\) by
\[\phi(p, q, r, s) = (\frac{-64 + 32 p - 4 p^2 + 4 q + 5 p q - q^2 - 12 r}{-12 + 3 p - q}, \ 
\frac{64 - 32 p + 4 p^2 + 16 q - p q - 6 r}{2 (-12 + 3 p - q)}, \] 
\[\frac{-16 + 8 p - p^2 - 12 q + 6 r}{-12 + 3 p - q}).\]
We use the polynomial parametrizations of \(S_C\)  and \(S_A\).  By the definition of  \(\psi\), we have 
%%%%%%%%%%%%%%%%%%%%%%%equation 4.33
\begin{equation}
\psi(u^2, uv, u^2) = (p(u, v), \ q(u, v), \ r(u, v), \ s(u, v)), 
\end{equation}
where \ \(p(u, v), \ q(u, v), \ r(u, v)\) \ and \  \(s(u, v)\) \ are the polynomials  in (4.31).\\
We set 
\[\phi (p(u, v), \ q(u, v), \ r(u, v), \ s(u, v)) =  (a(u, v), \ b(u, v), \ c(u, v)).\]
We compute \ \(a(u, v), \ b(u, v) \) \ and \ \(c(u, v)\).  Then the numerator of \ \(a(u, v) = -u^2(u-v)^2\) \ and the denominator of \ \(a(u, v) = -(u-v)^2\) .  Then  if \ \(u \ne v\), the function \ \(a(u, v)\)\ is reduced to \(u^2\).  Similarly,  if \ \(u \ne v\),
\[b(u, v) = \frac{-2uv(u-v)^2}{-2(u-v)^2} = uv, \ \mbox{and} \ c(u, v) = \frac{-v^2(u-v)^2}{-(u-v)^2} = v^2.\]
Hence if \ \(u \ne v\), then \ \(\phi \circ \psi(u^2, uv, v^2) = (u^2, uv, v^2)\).

Next we compute \ \(\psi \circ \phi\).  Let \( (p(u, v), \ q(u, v), \ r(u, v), \ s(u, v))\) \ be the parametrization of \(S_A\) in (4.31). Hence if \ \(u \ne v\), then 
%%%%%%%%%%%%%%%%%%%%%%%equation 4.34
\begin{equation}
 \phi(p(u, v), \ q(u, v), \ r(u, v), \ s(u, v)) =  (u^2, \ u v, \ v^2 ).
\end{equation}
Then by (4.33), we have 
\[\psi \circ \phi(p(u, v), \ q(u, v), \ r(u, v), \ s(u, v)) =  (p(u, v), \ q(u, v), \ r(u, v), \ s(u, v)).\]
We consider a proper subvariety \ \(V_{S_C}(I)\) \ of \(S_C\) satisfying \ \(u = v\).  Then \ \(V_{S_C}(I) = V(a-b, b-c).\)  We denote  a proper subvariety \ \(V_{S_A}(J)\) \ of \(S_A\) satisfying \ \(u = v\) \ by \(C_A\).  Then 
\ \(\phi \circ \psi\) \ is the identity on \ \(S_C \setminus V(a-b, b-c)\) , and
\ \(\psi \circ \phi\) \ is the identity on \ \(S_A \setminus C_A\).  Hence \ \(\phi \circ \psi =id_{S_C}\) \ and \  \(\psi \circ \phi =id_{S_A}\).
\end{proof}

The subvariety \(C_A\) will be studied in the next subsection.  It is an affine algebraic curve.

We consider the mapping \ \(\phi \circ g_d \circ \psi (a, b, c)\) \ from \(S_C\) to \(S_C\).  By (4.33),  (4.34) and Theorem \ref{theorem8}, it follows that if \ \(u \ne v\), then 
%%%%%%%%%%%%%%%%%%%%%%%equation 4.35
\begin{equation}
\phi \circ g_d \circ \psi(u^2, uv, u^2) = (T_d(u)^2, \ T_d(u)T_d(v), \ T_d(v)^2).
\end{equation}
We use lemma \ref{lemma9} to study Chebyshev polynomials.  Set \ \(j = d/2\) \ if \(d\) is even and \ \ \(j = (d-1)/2\) \ if \(d\) is odd.
Then
\[T_d(w) =  \left\{
\begin{array}{ll}
\displaystyle{p_{j}(w^2)}  & \quad \ \mbox{if}\ d \ \mbox{is even},\\
\displaystyle{ wq_{j}(w^2)}& \quad \ \mbox{if}\ d  \ \mbox{is odd},
\end{array} \right.\]
where \ \(p_j(x)\) \ and \ \(q_j(x)\) \ are monic polynomials of degree \(j\).  First, we assume that \(d\) is odd. Recall that \ \(a = u^2, \ b=uv\) \ and \ \(c = v^2\).  Then by (4.35), we have 
%%%%%%%%%%%%%%%%%%%%%%%equation 
\begin{equation*}
\phi \circ g_d \circ \psi(a, b, c) = (aq_j(a)^2, \ bq_j(a)q_j(c), \ cq_j(c)^2).
\end{equation*}
Note that \ \(bq_j(a)q_j(c) \equiv b^d + b_{d-1}(a, b, c) \quad mod \ (b^2 - ac),\) \ where  \ \(b_{d-1}(a, b, c) = bq_j(a)q_j(c) - ba^jc^j.\)  Next we assume that \(d\) is even.  Then by (4.35), we have 
%%%%%%%%%%%%%%%%%%%%%%%equation 
\begin{equation*}
\phi \circ g_d \circ \psi(a, b, c) = (p_j(a)^2, \ p_j(a)p_j(c), \ p_j(c)^2).
\end{equation*}
Note that \ \(p_j(a)p_j(c) \equiv b^d + b^\prime_{d-1}(a, b, c) \quad mod \ (b^2 - ac),\) \ where  \ \(b^\prime_{d-1}(a, b, c) = p_j(a)p_j(c) - a^jc^j.\)  

Then we can define a morphism \ \(h_d : S_C \to S_C\) \ by the following: \\
if \(d\) is odd, then
\[h_d(a, b, c) =  (aq_j(a)^2, \ b^d + b_{d-1}(a, b, c), \ cq_j(c)^2), \]
if \(d\) is even, then
\[h_d(a, b, c) =  (p_j(a)^2, \ b^d + b^\prime_{d-1}(a, b, c), \ p_j(c)^2). \]
Hence
\[h_d(a, b, c) = (a^d + A_{d-1}(a), \ b^d + B_{d-1}(a, b, c), \ c^d + C_{d-1}(c)), \]
where \ \(A_{d-1}, \ B_{d-1}, \ C_{d-1}\) \ are polynomials  of degree \ \(\le d-1\).
The morphism \(h_d\) can be extended to a morphism \ \(\bar{h}_d : {\mathbb P}^3 \to {\mathbb P}^3\).\\

Note that the projective closure \ \(\bar{S}_c\) \ of \ \({S}_c\) \ is given by \ \(V(x_2^2 - x_1x_3)\), where \ \((x_0 : x_1 : x_2 : x_3)\) \ are homogeneous coordinates on\ \({\mathbb P}^3\).
  Then by the argument similar to that in the proof of Proposition \ref{pro3.9}, we get the following proposition.
%%%%pro4.25
 \begin{pro} \label{pro4.25}
Under the above notation, we have \ \(\bar{h}_d(\bar{S}_c) = \bar{S}_c\).  
 \end{pro}
%%%%%%%%Subsection4.6
\subsection{\rm{The dynamics of \(g_d\) on an astroid curve  \(C_A\)}}\hspace{0cm}  

From the previous subsection,  the parametrization of \(C_A\) is obtained  by setting \ \(v = u\) \ in \ \(G_A\) ;
\[G_A(u, u) = (4+3u^2, \ (3u)^2, \ 3u(6u+u^3/2), \ (6u+u^3/2)^2).\]
Then we define a polynomial parametrization by
\[G_{CA}(t) = (4+3t, \ 9t, \ 3t(6+t/2), \ t(6+t/2)^2).\]
Hence by polynomial implicitization algorithm in Theorem 1 in \cite[Chapter 3, \S3]{CLS}, we compute implicit representation of the variety.
We compute the  Gr\(\ddot{\mbox o}\)bner basis for the ideal 
\[(4 + 3 t - p, \  9 t - q,  \ 3 t (6 + t/2) - r, \ 
   t (6 + t/2)^2 - s),\]
using graded reverse lex order with \ \(t >  p > q > r> s.\)  The  Gr\(\ddot{\mbox o}\)bner basis is
\[(-12 + 3 p - q, \ -q + 9 t, \ r^2 - q s, \ 108 r + q r - 54 s, \ 
 108 q + q^2 - 54 r).\]
Set
\[I = (-12 + 3 p - q,  \ r^2 - q s, \ 108 r + q r - 54 s, \ 
 108 q + q^2 - 54 r).\]
Then the smallest variety containing \ \(G_{CA}({\mathbb C})\) \ is \(V(I)\).
And  \ \(G_{CA}({\mathbb C})\) \ fills up all of \(V(I)\).
Then \ \(C_A = V(I)\) \ and \(C_A\) is an affine algebraic curve.  The curve  \(C_A\) is related to the astroid in space that is studied in \cite[p. 205]{U3}.  

The astroid in space is the curve \ \(\{(4\cos^3\theta, \ 4\sin^3\theta, \ 6\cos2\theta) : 0 \le \theta < 2\pi\}\) \ in \({\mathbb R}^3\).  It is a part of the edge of the astroidalhedron \(\mathcal{A}\).  Setting \ \(\cos\theta = (1-t^2)/(1+t^2),\) \ and \ \(\sin\theta = 2t/(1+t^2), \) \ we have a rational parametric representation of the astroid in space.
We define a mapping \ \(\rho : {\mathbb R} \to {\mathbb R}^3\), \(t\mapsto (x, y, z)\) \ defined by 
\[\rho(t) = (4((1-t^2)/(1+t^2))^3, \ 4(2t/(1+t^2))^3, \ 6(1-6t^2 + t^4)/(1+t^2)^2).\]
Using rational  implicitization algorithm  in \cite[Chapter 3, \S3]{CLS}, \ we compute a variety.  Then we have a set \ \(V({\mathcal F})\), \ where 
\[{\mathcal F} = (f_1 : = -216 + 108x^2 - 108z -18z^2 -z^3, \ f_2 : = -216 + 108y^2 + 108z -18z^2 + z^3).\]
The set \ \(V({\mathcal F})\) is the smallest variety \ in \({\mathbb R}^3\) \ containing  \ \(\rho ({\mathbb R})\).  The variety 
\ \(V({\mathcal F})\) \ is invariant under the action of \ \(D_4 = <R, C>\).  Indeed.   Note that \(R\) and \(C\) are elements of \ \(GL(3, {\mathbb R})\) \ and are described in Section 2.1.  We see that \ \(f_1 \circ R = f_2\), \ \(f_2 \circ R = f_1\) \ and \ \(f_1 \circ C = f_1\), \ \(f_2 \circ C = f_2\).  Then we can construct a relative orbit variety
\ \(V({\mathcal F})/D_4\).  From now on we view \(f_1\) and \(f_2\) as elements of  \ \({\mathbb C}[x, y, z]\).
Using Algorithm 2.6.2  in \cite[\S2.6]{St}, \ we compute a relative variety \ \(V({\mathcal F})/D_4\). Then we have \ \(V({\mathcal F})/D_4 =V(I)\). 
Hence \(C_A\) has  relation to the relative orbit variety of the astroid in space.  Then we call \(C_A\) an astroid curve.

Since \ \(C_A = \{G_A(u, u) : u \in {\mathbb C} \},\) \ from Theorem 4.22, we have the following proposition.
%%%%pro4.26
 \begin{pro} \label{pro4.26}
Let \ \(g_d\) be a morphism in Theorem 4.8.\\
(1) \ \(g_d(G_{CA}(t^2)) = G_{CA}(T_d(t)^2).\)\\
(2) \ \(g_d(C_A) = C_A\).
\end{pro}
We consider the singular locus of the surface \(S_A\).  We compute the Jacobian matrix \ \(\partial(A_h, \ -r^2+qs)/\partial(p, q, r, s).\)
Then using the parametrization \ \(G_A\), we see that the rank of the Jacobian matrix at \ \(G_A(u, v)\) \ is less than 2 if and only if \ \(u = v\).
%%%%pro4.27
 \begin{pro} \label{pro27}
The singular locus of  \(S_A\) is equal to the astroid curve \(C_A\).
 \end{pro}

We show a line which is preperiodic for the morphisms \(g_d\).
%%%%pro4.28
 \begin{pro} \label{pro28}
Let \ \(L_3 = \{(1, 0, 0, s) : s \in {\mathbb C} \}.\) \\
(1) If \ \(3\not| d,\)   then \ \ \(g_d(L_3) = L_3\). \\
(2) If \ \(3\mid d\), \ then \ \(g_d(L_3) = C_A\).
 \end{pro}
%%%%%%%%%%%%%%%%%%%%%%%%%%%proof
\begin{proof}
(1) The proof is similar to the argument given in Proposition \ref{pro4.14}.  We consider the diagram (4.15) in the case \ \(z_3 = {\bar z_1}\) \ and \ \(z_2 \in {\mathbb R}\).  We consider the condition that \ \(p = 1\) \ and \ \(q = r = 0\) \ in \(R_3\).  
Then  \(x^2 + y^2 = 1\) \ and \ \(z = 0\).  
Hence we set \ \(x = \cos\theta, \ y = \sin\theta, \ (0 \le \theta \le 2\pi) \  \mbox{and} \ z = 0\).  
Then \ \(z_1 = e^{i\theta}, \ z_2 = 0.\)

We consider the equation \ \(t^4 -  e^{i\theta}t^3 - e^{-i\theta}t +1 = 0.\)  The solutions are  
\[t_1 =  \omega e^{-i\theta/3}, \ t_2 =  \omega^2 e^{-i\theta/3}, \ t_3 =  e^{-i\theta/3}, \ t_4 = e^{i\theta}, \ \mbox{where} \ \omega = e^{2\pi i/3}.\]
This is the case that \(\alpha = \beta = -\theta/3\)  and  \(\gamma = 2\pi/3\) \  in (4.16).
 If \(d\) is not divisible by \(3\),    then we can write \ \((X, Y, Z)\) in (4.15) as \ \(X = \cos d\theta, \ Y = \sin d\theta, \ Z = 0.\)  Hence \ \(g_d(1, 0, 0, (\cos2\theta)^2) = (1, 0, 0, (\cos2d\theta)^2)\).  Then the  assertion (1) follows by the similar argument given in the proof of Proposition \ref{pro4.14} (1).\\
(2)  If \ \(d = 3\), \  \(t_1^3 = t_2^3 = t_3^3 = e^{-i\theta}\) \ and \ \(t_4^3 = e^{3i\theta}\). \  Then \ \(X = 4\cos^3\theta, \ Y = -4\sin^3\theta, \ Z = 6\cos2\theta\).  As we have noted above, it is a parametrization of the astroid in space.  Then by setting \ \(u = 2\cos2\theta\) \ we have 
\[g_3(1, 0, 0, (u/2)^2) = (3u^2 + 4,\ 9u^2, \ 3u^2(6+u^2/2), \  u^2(6+u^2/2)^2) = G_{CA}(u^2).\]
Then by Proposition 2.6 we have the assertion (2).
\end{proof}
%%%%%%%%Subsection4.7
\subsection{\rm{Other morphisms  and conjugacy}}\hspace{0cm} 

We studied a fundamental system of invariants of \ \( { \mathbb C}[x, y, z]^{D_{4}}\) \ in Section 4.1.  In (4.1),  we select a fundamental system \ \(\{p, q, r, s\}\) \ of invariants of  \ \( { \mathbb C}[x, y, z]^{D_{4}}\).  In Section 4.1, we considered the ideal  \ \(F = (p, q, r, s)\) \ and studied the orbit variety \ \(V_F\) \ and morphisms \ \(g_d\) \ on \ \(V_F\). 

In this section, we consider the ideal  \(F^\prime \) generated by  any other fundamental system   of invariants of  \ \( { \mathbb C}[x, y, z]^{D_{4}}\).  Then we have another syzygy ideal \ \(I_{F'},\)  an affine algebraic set \ \(V_{F'} = V(I_{F'}),\) \ and morphisms
\ \(g^\prime _d\) \ on \ \(V_{F^\prime }\).  We will show that there is an isomorphism \ \(\varphi : V_F \to V_{F^\prime }\) \ such that \ \(\varphi \circ g_d = g^\prime _d \circ \varphi\).

To start, we consider invariants of \ \( { \mathbb C}[x, y, z]^{D_4} \).   By Molien' theorem \cite[Theorem 2.2.1]{St}, the Hilbert series 
\(\Phi_{D_4}(t)\)\  of  \ \( { \mathbb C}[x, y, z]^{D_4} \) \ is written as 
\[\Phi_{D_4}(t) = 1 + 2t^2 + t^3 +  4t^4 + 2t^5 +  6t^6 + 4t^7 +  9t^8 + 6t^9  + \it{O}(t^{10}).\]
Then \ \(dim({ \mathbb C}[x, y, z]^{D_4}_2) = 2, \)  \ \(dim({ \mathbb C}[x, y, z]^{D_4}_3) = 1\) \ and \ \(dim({ \mathbb C}[x, y, z]^{D_4}_4) = 4.\) \\
Hironaka  decomposition for \ \({ \mathbb C}[p, q, r, s]\)  is written in the form 
\( { \mathbb C}[p, q, r, s]  =  { \mathbb C}[p, q, s] + r { \mathbb C}[p, q, s].\) 
The fundamental system \ \(\{p, q, r, s\}\) \ of invariants is complete.   From the above facts we know that any homogeneous invariant of degree 2  is written in the form 
\ \(\alpha_1(x^2 + y^2) + \alpha_2z^2, \ (\alpha_1, \alpha_2 \in {\mathbb C}),\) \ and that  any homogeneous invariant of degree 3  is written in the form \ \(\beta z(x^2 - y^2), \ (\beta \in {\mathbb C}), \) \ and that any homogeneous invariant of degree 4  is written in the form \ 
\(\gamma_1s + \gamma_2p^2 + \gamma_3pq + \gamma_4q^2, \ (\gamma_1,  \gamma_2, \gamma_3, \gamma_4 \in {\mathbb C}).\) 

We study  any other fundamental system of invariants.   From the above facts we see that the ideal  \(F^\prime\) generated by any other fundamental system  of invariants is written as  \ \(F^\prime = (p^\prime, q^\prime, r^\prime, s^\prime), \)  where 
%%%%%%%%%%%%%%%%%%%%%%%equation 4.36
\begin{equation}
 \left(
\begin{array}{c}
  p^\prime \\ q^\prime 
\end{array}
\right)
=
\left(
\begin{array}{cc}
  a_{11} \  a_{12} \\  a_{21} \  a_{22}
\end{array}
\right)
\left(
\begin{array}{c}
  p \\ q
\end{array}
\right)
,
\end{equation}
where 
\[ A
=
\left(
\begin{array}{cc}
  a_{11} \  a_{12} \\  a_{21} \  a_{22}
\end{array}
\right)
\in \ GL(2, {\mathbb C})
,\]
\[r^\prime  = br, \ (b \in {\mathbb C}^*), \enskip 
s^\prime  = cs + kp^2 + mpq + nq^2, \ (c \in {\mathbb C}^*, \ k, m, n \in {\mathbb C}).\]
We construct another affine algebraic variety \ \(V_{F^\prime}\) for the ideal \ \(F^\prime = (p^\prime, q^\prime, r^\prime, s^\prime).\) 

We define a morphism \(\varphi\) from \( { \mathbb C}^4\) to \( { \mathbb C}^4\) by
%%%%%%%%%%%%%%%%%%%%%%%equation 4.37
\begin{equation}
 \varphi(p, q, r, s) = (a_{11}p + a_{12}q, \ a_{21}p + a_{22}q, \ br, \ cs + kp^2 + mpq + nq^2).
\end{equation}
Clearly \(\varphi\) is a one-to-one map.

We consider the restriction \ \(\varphi \mid_{V_F}\) \ of the map \ \(\varphi\) \ to \ \(V_F\).  
%%%%%%%%%%%lemma4.29
\begin{lemma} \label{lemma4.29}
The map \ \(\varphi \mid_{V_F}\) \ maps \ \(V_F\) \ to \ \(V_{F^\prime}\). 
It is an isomorphism.
\end{lemma}
%%%%%%%%%%%%%%%%%%%%%%%%%%%proof
\begin{proof}
To prove this lemma, we use the following three facts:\\
(i) (\cite[Chapter 7, \S4, Proposition 2]{CLS}).
\[ { \mathbb C}[V_F] \simeq { \mathbb C}[p, q, r, s]/I_F \simeq { \mathbb C}[x, y, z]^{D_4} \simeq { \mathbb C}[p^\prime, q^\prime, r^\prime, s^\prime]/I_{F^\prime} \simeq { \mathbb C}[V_{F^\prime}].\]
(ii) (\cite[Chapter 5, \S4, Proposition 8(ii)]{CLS}).  Let \ \(\phi : { \mathbb C}[V_{F^\prime}] \to { \mathbb C}[V_{F}]\) \ be a ring homomorphism which is the identity on constants.  Then there is a unique polynomial mapping \(\alpha : V_F \to V_{F^\prime}\) \ such that \ \(\phi = \alpha^*\).\\
(iii) (\cite[Chapter 5, \S4, Theorem 9]{CLS}).  The map \ \(\alpha\) is a isomorphism if and only if \(\phi\) is a isomorphism of coordinate rings which is identity on constant functions.\\

Then to prove this lemma, it suffices to show that we construct a ring homomorphism \(\phi\) satisfying (ii) and (iii) and the map \(\alpha\) and that the map 
\(\alpha\)\  coincides with  the map \(\varphi \mid_{V_F}\).

Based on the proof of Proposition 2 in  (\cite[Chapter 7, \S4]{CLS}),   we construct a map
\[ \phi_1 : { \mathbb C}[V{_F^\prime}] \simeq { \mathbb C}[p^\prime, q^\prime, r^\prime, s^\prime]/I_{F^\prime} \to  { \mathbb C}[x, y, z]^{D_4}.\]
That is, for \ \(g \in {\mathbb C}[p^\prime, q^\prime, r^\prime, s^\prime]\) we define
\[ \phi_1([g]) =  g(p^\prime({\bf x}), q^\prime({\bf x}), r^\prime({\bf x}), s^\prime({\bf x})), \ \mbox{where} \ {\bf x} = (x, y, z).\]
We also construct a map 
\[\phi_2 : {\mathbb C}[V_{F}] \to  { \mathbb C}[x, y, z]^{D_4}.\]
The maps \(\phi_1\) \ and \ \(\phi_2\) \ are one-to-one maps.
Then we set 
\[\phi : = \phi_2^{-1} \circ \phi_1 : {\mathbb C}[V_{F^\prime}] \to  { \mathbb C}[V_F].\]
Hence
\[\phi([p^\prime]) = [a_{11}p + a_{12}q], \ \phi([q^\prime]) = [a_{21}p + a_{22}q],\]
\[\phi([r^\prime]) = [br], \ \phi([s^\prime]) = [cs + kp^2 + mpq + nq^2].\]
By the proof of Proposition 8 (ii) in
(\cite[Chapter 5, \S4]{CLS}), from \ \(\phi\)  we construct the map \ \(\alpha : V_F \to V_{F^\prime}\) \ satisfying \ \(\phi = \alpha^*.\)  Then 
\[\alpha = (a_{11}p + a_{12}q,  a_{21}p + a_{22}q, br,  cs + kp^2 + mpq + np^2).\]
Thus \ \(\alpha = \varphi \mid_{V_F} : V_F \to V_{F^\prime}\)  is an isomorphism.
\end{proof}

Let \ \(g^\prime_d\) be the morphism in Definition 2.4  
\ from \ \(V_{F^\prime}\)  \ to \ \(V_{F^\prime}\).  Then we will show that \ \(g^\prime_d : V_{F^\prime}  \to V_{F^\prime}\) \ is conjugate to \ \(g_d : V_{F}  \to V_{F}\) \ via the isomorphism 
\ \(\varphi \mid_{V_F}\) \ in the following proposition.\\
%%%%pro4.30
 \begin{pro} \label{pro30}
Let \  \(\{p^\prime, q^\prime, r^\prime, s^\prime\} \) \ be any other fundamental system of invariants of \ \( { \mathbb C}[x, y, z]^{D_{4}}\).  Set
\ \(F^\prime = (p^\prime, q^\prime, r^\prime, s^\prime).\)  By Definition 2.4 we define the map \ \(g^\prime_d\) 
\ from \ \(V_{F^\prime}\)  \ to \ \(V_{F^\prime}\).  Let
\ \(\varphi \mid_{V_F}\) \ be the map in Lemma 4.29.  Then we have  \ \(\varphi \mid_{V_F} \circ g_d =  g^\prime_d \circ \varphi \mid_{V_F}\).
 \end{pro}
%%%%%%%%%%%%%%%%%%%%%%%%%%%proof
\begin{proof}
From (2.11), we have 
\[g_d \circ \tilde{F}({\bf x}) = \tilde{F} \circ f_d({\bf x}), \ {\mbox {and}} \  g^\prime_d \circ \tilde{F}^\prime({\bf x}) = \tilde{F}^\prime \circ f_d({\bf x}), \]
where \ \(\tilde{F}^\prime\) \ is the polynomial map from \ \( { \mathbb C}^3\) \ to \  \(V_{F^\prime}\) \ in Theorem 2.3.
Since \ \(\varphi(p, q, r, s) =  (p^\prime, q^\prime, r^\prime, s^\prime),\)  \ it follows that \  \(\tilde{F}^\prime = \varphi \circ \tilde{F} \).  Then 
\[\varphi \circ g_d \circ \tilde{F} = \varphi \circ \tilde{F} \circ f_d =  \tilde{F}^\prime \circ f_d = g^\prime_d \circ \tilde{F}^\prime =  g^\prime_d \circ \varphi \circ \tilde{F}.\]
By Theorem 2.3(1), we know that \ \(\tilde{F}\) \ is surjective.  Then 
\(\varphi \mid_{V_F} \circ g_d  = g^\prime_d  \circ \varphi \mid_{V_F}.\)
\end{proof}

Acknowledgment.\\
The author thanks Professor Hirokazu Nasu for some useful advices.  The author also thanks the referee for his careful reading of the paper and many useful advices.

%%%%%%%%%%%%%%%%%%%%%%%%%%%%%%%%%%%%%%%%%%%%%%%%%%%%%%%%%%%%%%%%%%%%%%%%%%%%%%%%%%%%%%%%%%%%%%%%%%%%%%%%%%%%%%%%%%%% 
%%%%%%%%%fig1 fig2
%\begin{figure}[htbp]
%\begin{tabular}{cc}
%\begin{minipage}{0.5\hsize}
%\begin{center}
%\includegraphics[scale=0.29]{kfig-1.eps} \\
%\caption{ The natural domain \(R^\prime\).  }
%\label{fig1}
%\end{center}
%\end{minipage}
%\begin{minipage}{0.55\hsize}
%\begin{center}
%\includegraphics[scale=0.29]{kfig-2.eps}\\
%\caption{ The fundamental region $R$.}
%\label{kfig-2.eps}
%\end{center}
%\end{minipage} 
%\end{tabular}
%\end{figure} 
%%%%%%%%%%%%%%%%%%%%%%%%%%%%%%%%%%%%%%%%%%%%%%%%%%%%%%%%%%%%%%%%%%%%%%%%%%%%%%%%%%%%%%%%%%%%%%%%%%%%%%%%%%%%%%%%%%%%%%%%%%%%%%%%%%%%%%%%%%%%%%%%%%%%%%%%%%%%%%%%%%%55

\bibliographystyle{amsplain}

\end{document}